\documentclass[pdflatex,sn-mathphys-num]{sn-jnl}

\usepackage{graphicx}%
\usepackage{multirow}%
\usepackage{amsmath,amssymb,amsfonts}%
\usepackage{amsthm}%
\usepackage{mathrsfs}%
\usepackage[title]{appendix}%
\usepackage{xcolor}%
\usepackage{textcomp}%
\usepackage{manyfoot}%
\usepackage{booktabs}%
\usepackage{algorithm}%
\usepackage{algorithmicx}%
\usepackage{algpseudocode}%
\usepackage{listings}%

\usepackage{placeins}

\usepackage{subcaption}

\usepackage{float}

\usepackage{rotating}

\usepackage{tabularx}
\usepackage{array}

\usepackage{algorithm}
\usepackage{algpseudocode}

\usepackage{adjustbox}

\algrenewcommand\algorithmicrequire{\textbf{Input:}}
\algrenewcommand\algorithmicensure{\textbf{Output:}}

\theoremstyle{thmstyleone}%
\newtheorem{theorem}{Theorem}
\theoremstyle{thmstyletwo}%
\newtheorem{remark}{Remark}%

\newtheorem{lemma}{Lemma}
\newtheorem{corollary}[lemma]{Corollary}

\DeclareMathOperator{\Exp}{Exp}
\DeclareMathOperator{\Log}{Log}

\theoremstyle{thmstylethree}%

\theoremstyle{thmstylethree}
\newtheorem{assumption}{Assumption}[section]

\begin{document}

\title[Centered Weak Discrete Riemannian Gradients: A Unified Framework for Riemannian Optimization]{Centered Weak Discrete Riemannian Gradients: A Unified Framework for Riemannian Optimization}


\author*[1,2]{\fnm{Derun} \sur{Zhou}}\email{zhouderun@nii.ac.jp}

\author[1,2]{\fnm{Mahito} \sur{Sugiyama}}\email{mahito@nii.ac.jp}

\affil[1]{\orgdiv{National Institute of Informatics}, \orgaddress{\street{Hitotsubashi}, \city{Chiyoda-ku}, \postcode{101-8430}, \state{Tokyo}, \country{Japan}}}

\affil[2]{\orgdiv{The Graduate University for Advanced Studies}, \orgname{SOKENDAI}, \orgaddress{\street{Shonan Village}, \city{Hayama}, \postcode{240-0193}, \state{Kanagawa}, \country{Japan}}}


\abstract{
We introduce the centered weak discrete Riemannian gradient (c-wDRG)
framework for the unified analysis of optimization methods on
Riemannian manifolds. The framework uses a center point to represent
the relevant logarithmic differences in a common tangent space and
covers Riemannian steepest descent, proximal point, proximal gradient,
implicit midpoint, geodesic average-vector-field, Gonzalez, and
Itoh--Abe methods. We derive c-wDRG certificates for these schemes and
establish curvature-aware convergence results. The framework yields sublinear and linear
convergence for nonaccelerated methods in the geodesically convex and
strongly convex settings, respectively. We further develop accelerated
c-wDRG schemes, obtaining accelerated sublinear convergence in the
convex case and accelerated linear convergence in the strongly convex
case. Our results provide explicit curvature-dependent stepsize
conditions and convergence rates within a common framework for Riemannian gradient-descent, proximal, and discrete-gradient methods. Numerical results validate the theoretical results of the proposed methods.
}




\maketitle

\section{Introduction}

In this paper, we consider the minimization problem
\[
\min_{x\in \mathcal{X} \subseteq \mathcal M} f(x),
\]
where $(\mathcal M,g)$ is a Riemannian manifold with metric $g$,
and $f:\mathcal M\to\mathbb R$ is the objective function.
Our analysis is conducted on a bounded geodesically uniquely
convex region $\mathcal X\subseteq\mathcal M$ with
$\operatorname{diam}(\mathcal X)\leq D<\infty$ and sectional
curvature bounded by $K_{\min}\leq\operatorname{Sec}\leq K_{\max}$.
We assume that $\mathcal X$ contains a minimizer $x^\star$
and that all iterates and auxiliary points remain in $\mathcal X$.

When $\mathcal M=\mathbb R^n$ is equipped with the standard
Euclidean metric, the problem reduces to Euclidean optimization.
For smooth unconstrained convex problems, gradient descent
achieves an $O(k^{-1})$ objective-gap bound, while Nesterov's
accelerated method improves this rate to
$O(k^{-2})$~\cite{nesterov1983method, nesterov2013introductory}.
For $L$-smooth and $\mu$-strongly convex objectives, the
corresponding bounds are $O((1-\mu/L)^k)$ and
$O((1-\sqrt{\mu/L})^k)$, respectively.
These results provide fundamental benchmarks for the design
and analysis of optimization algorithms. Moreover,
proximal methods extend these convergence guarantees to
possibly nonsmooth convex objectives.

The proximal-point method generates iterates by solving
\[
x_{k+1}
=
\operatorname*{arg\,min}_{x\in\mathbb R^n}
\left\{
f(x)+\frac{1}{2s}\|x-x_k\|^2
\right\},
\qquad s>0.
\]
For a fixed parameter $s>0$ and exact subproblem solves,
the standard method achieves an $O(k^{-1})$ objective-gap
bound~\cite{guler1991convergence}, while G\"uler's accelerated method achieves $O(k^{-2})$~\cite{guler1992new}.
If the objective is $\mu$-strongly convex, a standard linear
bound for the proximal-point method is
$O((1+\mu s)^{-k})$.
A strongly convex accelerated proximal-point variant admits
the bound
$
O\!\left(
\left[
1-\sqrt{\frac{\mu s}{1+\mu s}}
\right]^k
\right)
$
under suitable acceleration parameters~\cite{lin2018catalyst}.

For composite objectives $f=f_1+f_2$, where $f_1$ is convex
and $L$-smooth and $f_2$ is convex, locally Lipschitz continuous, possibly nonsmooth, the proximal-gradient method combines a gradient
step for $f_1$ with a proximal step for $f_2$.
With step size $1/L$, the standard method achieves an
$O(k^{-1})$ objective-gap bound, while FISTA improves this
rate to $O(k^{-2})$~\cite{beck2009fast}.
When $f_1$ is additionally $\mu$-strongly convex,
the standard proximal-gradient method and a strongly convex
accelerated variant achieve
$
O\!\left((1-\mu/L)^k\right)$ and $
O\!\left((1-\sqrt{\mu/L})^k\right),
$
respectively, with suitable acceleration
parameters~\cite{nesterov2013gradient}.
All these bounds concern objective gaps and count outer
iterations, assuming exact evaluation of the proximal mappings.

Furthermore, discrete-gradient methods provide a complementary
approach to smooth optimization through structure-preserving
discretizations of gradient flows.
These methods use a discrete gradient $\overline{\nabla}f(x,y)$
satisfying the exact discrete chain rule
$
f(y)-f(x)
=
\left\langle
\overline{\nabla}f(x,y),\,y-x
\right\rangle, x,y\in\mathbb R^n.
$
This identity allows the resulting gradient-flow discretizations
to preserve energy dissipation and ensures nonincreasing
objective values when the implicit steps are solved exactly.
Classical constructions include the Gonzalez, mean-value,
and Itoh--Abe discrete gradients.
Ehrhardt et al.~\cite{ehrhardt2025geometric} analyzed the well-posedness and convergence of these methods.
Under appropriate regularity and step-size assumptions,
they established $O(k^{-1})$ objective-gap bounds for convex
objectives and linear convergence under the
Polyak--{\L}ojasiewicz inequality.
Since strong convexity implies this inequality, their analysis
also provides linear convergence guarantees for strongly
convex objectives.

To unify these methods, Ushiyama et al.~\cite{Ushiyama2023}
introduced the notion of a weak discrete gradient (wDG)
for Euclidean optimization.
Their framework replaces the exact discrete chain rule
with a three-point inequality that incorporates quadratic
error terms and convexity information.
It encompasses gradient descent, proximal-point,
proximal-gradient, and classical discrete-gradient methods
under their respective assumptions.
Through discrete Lyapunov arguments, they establish
$O(k^{-1})$ and $O(k^{-2})$ objective-gap bounds for
nonaccelerated and accelerated schemes, respectively,
in the convex setting.
For wDG parameters satisfying $\beta+\gamma>0$,
their strongly convex bounds are
$
O\!\left(
\left[
1-\frac{2(\beta+\gamma)s}{1+2\gamma s}
\right]^k
\right)$
and
$O\!\left(
\left[
1+\sqrt{2(\beta+\gamma)}\,s
\right]^{-k}
\right),
$
for the nonaccelerated and accelerated schemes, respectively,
with $s>0$ satisfying the corresponding step-size conditions
in Theorems~5.3 and~5.5.

When $\mathcal{M}$ is a general Riemannian manifold rather than a Euclidean space, even for geodesically convex objectives, however, transferring
Euclidean convergence arguments to manifolds is not straightforward. The basic difficulty is that displacement vectors based at different
points belong to different tangent spaces, while curvature distorts the relation between tangent-space quantities and geodesic distances.

Zhang and Sra~\cite{zhang2016first} established global
iteration-complexity bounds for first-order methods on
Hadamard manifolds in both the geodesically convex and
strongly geodesically convex settings, explicitly quantifying
the effect of sectional curvature on convergence rates.
Their results show that, even for nonaccelerated methods,
geometry affects both the formulation of the algorithms
and their convergence guarantees.

A natural question is whether accelerated convergence can be
achieved despite this geometric distortion.
Liu et al.~\cite{liu2017accelerated} extended Nesterov's
acceleration to geodesically convex and strongly geodesically
convex optimization using nonlinear operators on manifolds.
Their methods require solving nonlinear equations at each
iteration, which can be computationally demanding.
Zhang and Sra~\cite{zhang2018estimate} developed a tractable
accelerated gradient method based on Riemannian estimate
sequences and bounds on metric distortion.
Their accelerated guarantee requires initialization in a
neighborhood of the minimizer whose size depends on curvature
and the condition number. Ahn and Sra~\cite{ahn2020nesterov} established global  convergence with eventual acceleration for strongly geodesically convex objectives.
Their analysis controls metric distortion along the iterates,
allowing the convergence factors to approach an accelerated
regime as the algorithm progresses.
For geodesically convex objectives,
Alimisis et al.~\cite{alimisis2021momentum} proposed a momentum
method using small-dimensional relaxation.
Their objective-gap bound contains an $O(k^{-2})$ term and
an additional curvature-dependent $O(k^{-1})$ term, yielding
an accelerated regime during the initial stages.
Mart\'inez-Rubio~\cite{martinez2022global} obtained global
accelerated guarantees for both convexity classes in hyperbolic
space and suitable spherical domains.
These guarantees recover Euclidean dependence on smoothness,
strong convexity, and accuracy up to geometric and logarithmic
factors, while relying on the structure of constant-curvature
spaces. Kim and Yang~\cite{KimYang2022} developed tractable accelerated
methods for both geodesically convex and strongly geodesically
convex objectives on bounded geodesically uniquely convex
regions with bounded sectional curvature.
Using metric-distortion inequalities and discrete potential
functions, they obtained accelerated iteration-complexity
bounds without condition-number-dependent local initialization.

Riemannian proximal-point methods provide another approach
to possibly nonsmooth geodesically convex optimization.
Ferreira and Oliveira~\cite{ferreira2002proximal}
established convergence on Hadamard manifolds, while
Bento et al.~\cite{bento2017iteration} subsequently derived
an $O(k^{-1})$ objective-gap bound.
Wang et al.~\cite{wang2016proximal} established linear
convergence of inexact proximal-point algorithms under
metric subregularity, a condition satisfied by the
subdifferential of a geodesically strongly convex objective.
For acceleration, Jin and Sra~\cite{jin2022understanding}
introduced a Riemannian accelerated hybrid
proximal-extragradient framework that encompasses several
accelerated gradient constructions.
Mart\'inez-Rubio and Pokutta~\cite{martinez2023accelerated}
developed an accelerated inexact proximal-point method
on Hadamard manifolds with bounded sectional curvature.
Their Theorem~4 establishes an $O(k^{-2})$ objective-gap
bound for geodesically convex objectives and an accelerated
linear guarantee through restarting in the strongly convex
case.
The method permits nonsmooth objectives, enforces feasibility
of proximal outputs, and controls the region containing
auxiliary points.

For composite objectives, Huang and
Wei~\cite{huang2022riemannian} developed Riemannian
proximal-gradient methods and established an $O(k^{-1})$
objective-gap bound under retraction convexity.
Bergmann et al.~\cite{bergmann2507intrinsic} studied an
intrinsic proximal-gradient method based on the manifold
proximal map, obtaining sublinear convergence for
geodesically convex problems and linear convergence under
strong geodesic convexity, subject to their geometric and
regularity assumptions.
Feng et al.~\cite{feng2025riemannian} developed a Riemannian
accelerated proximal-gradient method for composite objectives
with a possibly nonsmooth $\rho$-retraction-convex term.
Their analysis establishes an $O(k^{-2})$ objective-gap
bound and accelerated linear convergence under the
corresponding convexity and parameter conditions.

In the Riemannian setting, Celledoni
et al.~\cite{Celledoni2018,Celledoni2020} developed discrete
Riemannian gradients using retractions and center functions,
including Gonzalez, average-vector-field and Itoh--Abe constructions.
For gradient systems, these methods yield energy-dissipating
discretizations that ensure monotonic decrease of the objective
when the implicit steps are solved exactly.

These developments motivate a Riemannian extension of the weak
discrete-gradient framework that provides a common convergence
theory for gradient, proximal, and discrete-gradient methods.
Such an extension must address two geometric difficulties.
First, displacement vectors based at different points belong
to different tangent spaces and cannot be combined directly
as in Euclidean weak discrete-gradient arguments.
Second, expressing these vectors in a common tangent space
does not recover the Euclidean distance identities:
although $\|\Log_c x\|_c=d(c,x)$, the quantity
$\|\Log_c x-\Log_c y\|_c^2$ generally differs from $d^2(x,y)$.
This curvature-induced discrepancy must be controlled when relating the discrete updates to the distance terms in the convergence analysis.

To address these difficulties, we introduce the
\emph{centered weak discrete Riemannian gradient}
(c-wDRG) framework.
For each update, we choose a center $c$ and express the
relevant logarithmic differences in the common tangent
space $T_c\mathcal M$, where the vector identities underlying
the Euclidean analysis remain valid.
Curvature-dependent comparison estimates then relate these
tangent-space quantities to geodesic distances.
Within this formulation, each method is characterized by
a c-wDRG certificate, while the convergence analysis is
developed through inequalities shared by all admissible
certificates.
This allows gradient, proximal, and discrete-gradient
schemes to be analyzed and accelerated within the same
framework under geodesic convexity and geodesic strong convexity.

Our main contributions are as follows.
\begin{itemize}

    \item We introduce the c-wDRG framework and show that a broad class
    of Riemannian optimization methods can be represented by explicit
    certificates. These include Riemannian steepest descent, proximal
    point, proximal gradient, implicit midpoint, geodesic
    average-vector-field, Gonzalez, and Itoh--Abe methods.
   The c-wDRG framework is introduced in Section~\ref{sec1},
and its concrete instances are specified in
Definition~\ref{def:structure-preserving-methods}, and their c-wDRG
    certificates are established in
    Theorem~\ref{thm:structure-preserving-certificates}.

    \item We develop a unified curvature-aware convergence theory for
    nonaccelerated c-wDRG methods. This yields an $O(1/k)$ rate for geodesically convex
    objectives and linear convergence for geodesically strongly convex
    objectives. The curvature correction is established in
    Corollary~\ref{cor:center-compatible-curvature-correction}, the unified one-step estimate in Theorem~\ref{thm:unified-wdrg-estimate},
    and the convex and strongly convex convergence rates in
    Corollaries~\ref{cor:convex-rate}
    and~\ref{cor:strongly-convex-rate}, respectively.
    The optimized method-specific stepsizes and convergence rates are
    summarized in Table~\ref{tab:best-step-rate}.

    \item We further show that the same c-wDRG certificates can be
    accelerated.
    Lemma~\ref{thm:endpoint-reduction} reduces a general centered
    c-wDRG update to a common endpoint model, which is then combined
    with metric-distortion estimates and a unified potential argument.
    The resulting accelerated schemes for geodesically convex and
strongly geodesically convex objectives are presented in
Algorithms~\ref{alg:accelerated-cwdrg-convex}
and~\ref{alg:accelerated-cwdrg-strong}, respectively.
    Their convergence rates are established in
    Theorems~\ref{thm:cwdrg-accelerated-convex}
    and~\ref{thm:cwdrg-accelerated-strong},
    yielding an $O(1/k^2)$ rate in the geodesically convex setting and
    an accelerated linear rate in the geodesically strongly convex
    setting, respectively.
    The corresponding optimized strongly convex convergence rates are
    summarized in Table~\ref{tab:optimal-accelerated-strongly-convex-rates}.

\item Specifically, within the c-wDRG framework, we establish the first linear convergence
guarantees under geodesic strong convexity for optimization
schemes based on the Riemannian discrete gradients of
Celledoni et al.~\cite{Celledoni2020}, and develop accelerated
variants under the same convexity assumption.
We also construct accelerated proximal-gradient and proximal-point
methods. The former improves Feng et al.'s optimized
certified strongly convex rate,
while the latter permits positive sectional curvature on bounded
geodesically uniquely convex regions under our assumptions,
beyond Riemacon's Hadamard setting.
These comparisons are detailed in
Remarks~\ref{rem:comparison-rapg}
and~\ref{rem:app-comparison}.

\item We empirically demonstrate the convergence of the proposed
accelerated c-wDRG methods through experiments on SPD Karcher means, and $\ell_1$-regularized composite
optimization on a mixed-curvature product manifold.

\end{itemize}

The remainder of this paper is organized as follows.
Section~\ref{sec:riemannian-preliminaries} reviews the necessary
background, and Section~\ref{sec:assumptions} states the standing
assumptions.
Section~\ref{sec1} introduces the c-wDRG framework, and
Section~\ref{sec:examples} derives c-wDRG certificates for the
methods considered in this paper.
Section~\ref{sec:convergence} establishes the nonaccelerated and
accelerated convergence results.
Section~\ref{sec:experiments} presents the numerical experiments.
Section~\ref{sec:conclusion} concludes the paper, while technical
proofs and optimized method-specific convergence rates are
deferred to the appendices.

\section{Preliminaries on Riemannian Geometry}
\label{sec:riemannian-preliminaries}

We briefly recall the basic notions of Riemannian geometry and
geodesic convexity used throughout the paper.

Let $(\mathcal M,g)$ be a finite-dimensional Riemannian manifold.
For each $x\in\mathcal M$, let $T_x\mathcal M$ denote the tangent space
at $x$. The Riemannian metric induces the inner product and norm
\[
\langle u,v\rangle_x:=g_x(u,v),
\qquad
\|v\|_x:=\sqrt{\langle v,v\rangle_x},
\qquad
u,v\in T_x\mathcal M,
\]
and the tangent bundle is
$T\mathcal M:=\bigsqcup_{x\in\mathcal M}T_x\mathcal M$.

For a piecewise smooth curve $\gamma:[a,b]\to\mathcal M$, its length,
the Riemannian distance, and the diameter of a nonempty set
$\mathcal X\subseteq\mathcal M$ are defined, respectively, by
\[
\ell(\gamma):=\int_a^b\|\dot\gamma(t)\|_{\gamma(t)}\,dt,
\qquad
d(x,y):=\inf_{\gamma:x\to y}\ell(\gamma),
\qquad
\operatorname{diam}(\mathcal X):=
\sup_{x,y\in\mathcal X}d(x,y).
\]

Let $\nabla$ denote the Levi--Civita connection. A smooth curve
$\gamma$ is a geodesic if $\nabla_{\dot\gamma}\dot\gamma=0$.
For $x\in\mathcal M$ and $v\in T_x\mathcal M$, the exponential map is
defined by
\[
\Exp_x(v):=\gamma_v(1),
\]
where $\gamma_v$ is the geodesic satisfying
$\gamma_v(0)=x$ and $\dot\gamma_v(0)=v$. Whenever $\Exp_x$ is
invertible on the relevant domain, its inverse is denoted by $\Log_x$.
If $x$ and $y$ are joined by a unique minimizing geodesic, we write
\[
\gamma_{x,y}(t)
:=
\Exp_x\bigl(t\Log_x y\bigr),
\qquad
t\in[0,1],
\qquad
d(x,y)=\|\Log_x y\|_x.
\]

For such $x$ and $y$, let
$P_{x\to y}:T_x\mathcal M\to T_y\mathcal M$ denote parallel transport
along the unique minimizing geodesic. It is a linear isometry satisfying
\[
\langle P_{x\to y}u,P_{x\to y}v\rangle_y
=
\langle u,v\rangle_x,
\qquad
\|P_{x\to y}v\|_y=\|v\|_x,
\qquad
P_{y\to x}=P_{x\to y}^{-1}.
\]

For a differentiable function $f:\mathcal M\to\mathbb R$, its
Riemannian gradient $\operatorname{grad}f(x)\in T_x\mathcal M$ is
defined by
\[
Df(x)[v]
=
\langle\operatorname{grad}f(x),v\rangle_x,
\qquad
v\in T_x\mathcal M.
\]
If $f$ is twice differentiable, its Riemannian Hessian is
\[
\operatorname{Hess}f(x)[v]
:=
\nabla_v\operatorname{grad}f,
\]
and, for any smooth curve $\gamma$,
\[
\frac{d}{dt}f(\gamma(t))
=
\left\langle
\operatorname{grad}f(\gamma(t)),\dot\gamma(t)
\right\rangle_{\gamma(t)}.
\]
For linearly independent $u,v\in T_x\mathcal M$,
$\operatorname{Sec}_x(u,v)$ denotes the sectional curvature of the
two-dimensional subspace spanned by $u$ and $v$; when no ambiguity
arises, we simply write $\operatorname{Sec}$.

A set $\mathcal X\subseteq\mathcal M$ is called
\emph{geodesically convex} if every $x,y\in\mathcal X$ can be joined
by a minimizing geodesic contained in $\mathcal X$, and
\emph{geodesically uniquely convex} if this minimizing geodesic is
unique.

A function $f:\mathcal X\to\mathbb R$ is called
\emph{geodesically convex} if
\[
f\bigl(\gamma_{x,y}(t)\bigr)
\leq
(1-t)f(x)+tf(y),
\qquad
x,y\in\mathcal X,\quad t\in[0,1].
\]
For differentiable $f$, this is equivalently characterized by
\begin{equation}
\label{eq:geodesic-convexity}
f(y)
\geq
f(x)
+
\left\langle
\operatorname{grad}f(x),\Log_x y
\right\rangle_x,
\qquad
x,y\in\mathcal X.
\end{equation}

More generally, $f$ is
\emph{$\mu$-strongly geodesically convex}, with $\mu\geq0$, if
\begin{equation}
\label{eq:strong-geodesic-convexity}
f(y)
\geq
f(x)
+
\left\langle
\operatorname{grad}f(x),\Log_x y
\right\rangle_x
+
\frac{\mu}{2}d^2(x,y),
\qquad
x,y\in\mathcal X.
\end{equation}
The case $\mu=0$ reduces to geodesic convexity.

A differentiable function $f:\mathcal X\to\mathbb R$ is called
\emph{$L$-smooth} if its Riemannian gradient is $L$-Lipschitz with
respect to parallel transport:
\begin{equation}
\label{eq:riemannian-L-smoothness}
\left\|
P_{x\to y}\operatorname{grad}f(x)
-
\operatorname{grad}f(y)
\right\|_y
\leq
L\,d(x,y),
\qquad
x,y\in\mathcal X.
\end{equation}
This condition implies the Riemannian descent lemma
\begin{equation}
\label{eq:riemannian-descent-lemma}
f(y)
\leq
f(x)
+
\left\langle
\operatorname{grad}f(x),\Log_x y
\right\rangle_x
+
\frac{L}{2}d^2(x,y),
\qquad
x,y\in\mathcal X.
\end{equation}

\section{Assumptions}
\label{sec:assumptions}

We impose the following assumptions throughout the paper. These assumptions constitute a standard,  bounded-region setting for Riemannian acceleration. In particular, the core requirements of a bounded geodesically uniquely convex domain, two-sided sectional-curvature bounds, and all iterates remaining in the prescribed domain are also adopted by Alimisis et al.~\cite{alimisis2021momentum}, Kim and Yang~\cite{KimYang2022}, and Feng et al.~\cite{feng2025riemannian}. Our additional center-dependent diameter condition~\eqref{eq:diameter-condition} strengthens this standard setting only for the accelerated schemes.

\begin{assumption}
\label{ass:geometry}
The domain $\mathcal X$ is an open geodesically uniquely convex subset
of $\mathcal M$. The diameter of the domain is bounded as
\[
\operatorname{diam}(\mathcal X)\leq D<\infty.
\]
The sectional curvature inside $\mathcal X$ is bounded below by
$K_{\min}$ and bounded above by $K_{\max}$, i.e.,
\begin{equation}
\label{eq:sectional-curvature-bound}
K_{\min}
\leq
\operatorname{Sec}
\leq
K_{\max}.
\end{equation}
Associated with the curvature and diameter bounds, define
\begin{equation}
\label{eq:curvature-comparison-constants}
\begin{aligned}
\delta_D
&:=
\begin{cases}
1,
& K_{\max}\leq0,\\[1mm]
\sqrt{K_{\max}}D
\cot\!\left(\sqrt{K_{\max}}D\right),
& K_{\max}>0,
\end{cases}
\\[2mm]
\zeta_D
&:=
\begin{cases}
\sqrt{-K_{\min}}D
\coth\!\left(\sqrt{-K_{\min}}D\right),
& K_{\min}<0,\\[1mm]
1,
& K_{\min}\geq0.
\end{cases}
\end{aligned}
\end{equation}
If $K_{\max}>0$, we assume that
$
D<\frac{\pi}{\sqrt{K_{\max}}}.
$
For the accelerated schemes, we additionally require
\begin{equation}
\label{eq:diameter-condition}
D<\frac{\kappa_\theta}{\sqrt{K_{\max}}},
\qquad 0<\theta\leq1,
\end{equation}
where $\theta$ specifies the center
$c=\gamma_{z,y}(\theta)$ along the geodesic between the two endpoints,
and $\kappa_\theta\in(0,\pi)$ is the unique solution of
$
\kappa_\theta\cot(\kappa_\theta)
=
1-\frac{1}{2\theta}.
$
For Riemannian gradient descent, Riemannian proximal gradient, and the Riemannian Itoh--Abe method,
$\theta=0$; therefore, the original condition
$D<\pi/\sqrt{K_{\max}}$ is retained.
For the midpoint and Gonzalez methods, $\theta=1/2$, and hence
$\kappa_{1/2}=\pi/2$, yielding
$D<\pi/(2\sqrt{K_{\max}})$. Moreover, for the proximal-point method, $\theta=1$, so that
$\kappa_1\cot(\kappa_1)=1/2$ and hence
$D<\kappa_1/\sqrt{K_{\max}}$, where
$\kappa_1\approx1.16556$. Furthermore,
Assumption~\ref{ass:geometry} ensures that the exponential map
$\Exp_x$ is a diffeomorphism on the relevant domain for every
$x\in\mathcal X$, so that the logarithm maps and parallel transports
used below are well-defined \cite{KimYang2022}.

\end{assumption}

\begin{remark}
The additional center-dependent restriction in
\eqref{eq:diameter-condition} is imposed only for the
accelerated schemes. It ensures that
\[
1-2(1-\delta_D)\theta>0,
\]
as required by the accelerated convergence analysis.
For the nonaccelerated schemes, this positivity condition
is unnecessary: we retain the other geometric assumptions,
including $D<\pi/\sqrt{K_{\max}}$ when $K_{\max}>0$,
together with the admissibility conditions stated in
the corresponding convergence theorems.

This assumption implies that the exponential map $\Exp _x$ is a diffeomorphism for any $x \in \mathcal{X}$ in \cite{alimisis2021momentum}.
\end{remark}

\begin{assumption}
\label{ass:objective}
Except for the Riemannian proximal-gradient method, the objective
function $f:\mathcal X\to\mathbb R$ is continuously differentiable and
geodesically $L$-smooth. Moreover, $f$ is bounded below and has
minimizers, all of which lie in $\mathcal X$. A global minimizer is
denoted by $x^\star$. For the Riemannian proximal-gradient method, the objective admits the composite structure $f=f_1+f_2$ specified
for the Riemannian proximal-gradient method in
Definition~\ref{def:structure-preserving-methods}(iii).
\end{assumption}

\begin{assumption}
\label{ass:iterates}
All iterates and auxiliary points generated by the algorithms are
well-defined on the manifold $\mathcal M$ and remain in $\mathcal X$.
\end{assumption}

\section{Centered Weak Discrete Riemannian Gradients}\label{sec1}

We now extend the weak discrete-gradient framework of
Ushiyama et al.~\cite{Ushiyama2023} to Riemannian manifolds by
introducing a center point that provides a common tangent space for the
two endpoints of each update.

A center map assigns a point $c=c(y, z) \in \mathcal{X}$ to every pair $(y, z)$. Define
$$
\Delta_c(x, y):=\Log _c x-\Log _c y \in T_c \mathcal{M} .
$$
Because all logarithmic vectors lie in the same vector space $T_c \mathcal{M}$,
$$
\Delta_c(y, x)=\Delta_c(y, z)+\Delta_c(z, x) .
$$
In Euclidean space, $\Delta_c(a, b)=a-b$ independently of the center. Moreover, let $\mathcal D\subseteq\mathcal X\times\mathcal X$ denote the set of
admissible update pairs. A map
\[
\mathcal G_f(y,z;c)\in T_c\mathcal M,
\qquad
(y,z)\in\mathcal D,
\]
is called a \emph{centered weak discrete Riemannian gradient} (c-wDRG) with
parameters $(\alpha,\beta,\gamma)$ if, for every
$(y,z)\in\mathcal D$ and every $x\in\mathcal X$,
\begin{equation}
\label{eq:c-wdrg}
f(y)-f(x)
\leq
\left\langle
\mathcal G_f(y,z;c),
\Delta_c(y,x)
\right\rangle_c
+
\alpha\|\Delta_c(y,z)\|_c^2
-
\beta\|\Delta_c(z,x)\|_c^2
-
\gamma\|\Delta_c(y,x)\|_c^2.
\end{equation}

If $f$ is differentiable and $(x,x)\in\mathcal D$ for every
$x\in\mathcal X$, we additionally call $\mathcal G_f$ consistent if
\[
\mathcal G_f(x,x;x)
=
\operatorname{grad}f(x),
\qquad
x\in\mathcal X.
\]
For the smooth gradient approximations considered below,
all pairs are admissible, so that
$\mathcal D=\mathcal X\times\mathcal X$, and the resulting c-wDRGs
are consistent. For the proximal-gradient approximation, only pairs $(y,z)$ in which
$y$ is obtained from $z$ by the corresponding proximal local model
are admissible.

Taking $x=z$ in \eqref{eq:c-wdrg} gives
\begin{equation}
\label{eq:c-wdrg-chain-rule}
f(y)-f(z)
\leq
\left\langle
\mathcal G_f(y,z;c),
\Delta_c(y,z)
\right\rangle_c
+
(\alpha-\gamma)
\|\Delta_c(y,z)\|_c^2.
\end{equation}
If instead the exact identity
\[
f(y)-f(z)
=
\left\langle
\mathcal G_f(y,z;c),
\Delta_c(y,z)
\right\rangle_c
\]
holds for all $y,z\in\mathcal X$, then $\mathcal G_f$ is called a
\emph{strict centered discrete Riemannian gradient}.
This notion recovers the discrete Riemannian-gradient framework of
Celledoni et al.~\cite{Celledoni2018,Celledoni2020}, in which the exact
discrete chain rule is the defining structure-preserving property.
In particular, standard discrete Riemannian gradients such as the
geodesic average-vector-field (AVF), Riemannian Gonzalez discrete gradient and Riemannian Itoh--Abe constructions
satisfy this identity.

\section{Examples of Centered Weak Discrete Riemannian Gradients}
\label{sec:examples}
We first introduce several methods that can be embedded into the c-wDRG framework. For convenience of convergence analysis, we specialize the general
discrete Riemannian-gradient construction of
Celledoni et al.~\cite{Celledoni2020}, which allows substantial freedom
in the choice of both the retraction and the center function. Here, we use the exponential map as the retraction and
restrict the center to the geodesic form
$c_\theta=\gamma_{z,y}(\theta)$, with $\theta=0$ for steepest descent, proximal gradient, 
and Itoh--Abe, $\theta=1$ for proximal point, $\theta=1/2$ for
midpoint and Gonzalez, and an admissible $\theta\in[0,1]$ for AVF.
These choices facilitate the derivation of explicit c-wDRG
certificates and convergence rates without changing the 
discrete-gradient structure of the strict methods.

\label{def:structure-preserving-methods}
Let $x,y,z\in\mathcal{X}$. We consider the following methods.

\begin{enumerate}

\item[(i)] \textbf{Riemannian steepest descent.}
Choose $c=z$ and define
\begin{equation}
\mathcal{G}^{\mathrm{SD}}_f(y,z;z)
:=
\operatorname{grad}f(z).
\label{eq:sd-gradient}
\end{equation}

\item[(ii)] \textbf{Riemannian proximal point.}
Choose $c=y$ and define
\begin{equation}
\mathcal{G}^{\mathrm{PP}}_f(y,z;y)
:=
\operatorname{grad}f(y).
\label{eq:pp-gradient}
\end{equation}

\item[(iii)]
\label{def:proximal-gradient}
\textbf{Riemannian proximal gradient.}
Suppose that the objective admits the composite decomposition
\[
f:=f_1+f_2,
\]
where $f_1:\mathcal M\to\mathbb R$ is differentiable and,
on $\mathcal X$, geodesically $L$-smooth and
$\mu$-strongly geodesically convex, and
$f_2:\mathcal M\to\mathbb R$ is locally Lipschitz and
$\rho_{f_2}$-retraction-convex with respect to the exponential map
on $\mathcal X$, where $\rho_{f_2}\in\mathbb R$.

More precisely, for $z\in\mathcal X$, define
\[
\widetilde f_{2,z}(\eta)
:=
f_2\!\left(\Exp_z(\eta)\right)
+
\frac{\rho_{f_2}}{2}\|\eta\|_z^2.
\]
We say that $f_2$ is $\rho_{f_2}$-retraction-convex with respect to
the exponential map on $\mathcal X$ if, for every
$U\subseteq T_z\mathcal M$ satisfying
$\Exp_z(U)\subseteq\mathcal X$, every $\eta,\omega\in U$, and every
$g\in\partial\widetilde f_{2,z}(\omega)$,
\begin{equation}
\label{eq:rho-retraction-convex}
\widetilde f_{2,z}(\eta)
\geq
\widetilde f_{2,z}(\omega)
+
\left\langle
g,\eta-\omega
\right\rangle_z.
\end{equation}

Choose $c=z$. For the same stepsize $h>0$ used in the c-wDRG update,
define the local model
\begin{equation}
\label{eq:pg-local-model}
\ell_{z,h}(\eta)
:=
\left\langle
\operatorname{grad}f_1(z),\eta
\right\rangle_z
+
\frac{1}{2}
\left(
\rho_{f_2}+\frac{1}{h}
\right)
\|\eta\|_z^2
+
f_2\!\left(\Exp_z(\eta)\right),
\qquad
\eta\in T_z\mathcal M.
\end{equation}
Equivalently,
\[
\ell_{z,h}(\eta)
=
\left\langle
\operatorname{grad}f_1(z),\eta
\right\rangle_z
+
\frac{1}{2h}\|\eta\|_z^2
+
\widetilde f_{2,z}(\eta).
\]

Let $\eta_z\in\Log_z(\mathcal X)$ be a stationary point of
$\ell_{z,h}$, i.e.,
\begin{equation}
\label{eq:pg-stationarity}
0
\in
\operatorname{grad}f_1(z)
+
\frac{1}{h}\eta_z
+
\partial\widetilde f_{2,z}(\eta_z).
\end{equation}
Set
\begin{equation}
\label{eq:pg-step}
y:=\Exp_z(\eta_z).
\end{equation}
By the choice of $\eta_z$, we have $\eta_z=\Log_z y$.

A pair $(y,z)$ obtained in this way is admissible for the
proximal-gradient method. For every such pair, define
\begin{equation}
\label{eq:pg-gradient}
\mathcal G_{f,h}^{\mathrm{PG}}(y,z;z)
:=
-\frac{1}{h}\eta_z.
\end{equation}

\item[(iv)] \textbf{Riemannian implicit midpoint method.}
Let
\[
m:=\operatorname{mid}(y,z)
\]
be the midpoint of the selected shortest geodesic from $z$ to $y$.
Choose $c=m$ and define
\begin{equation}
\mathcal{G}^{\mathrm{MP}}_f(y,z;m)
:=
\operatorname{grad}f(m).
\label{eq:mp-gradient}
\end{equation}

    \item[(v)] \textbf{Geodesic AVF discrete gradient.}
    Let $\gamma_{z,y}:[0,1]\to\mathcal{X}$ be the constant-speed shortest
    geodesic from $z$ to $y$, and choose
    \[
    c_\theta:=\gamma_{z,y}(\theta),
    \qquad
    0\leq\theta\leq1.
    \]
    Then
    \[
    \Log_{c_\theta}z
    =
    -\theta\,\Delta_{c_\theta}(y,z),
    \qquad
    \Log_{c_\theta}y
    =
    (1-\theta)\,\Delta_{c_\theta}(y,z).
    \]
    For $t\in[0,1]$, define
    \[
    u_t
    :=
    (1-t)\Log_{c_\theta}z
    +
    t\Log_{c_\theta}y
    =
    (t-\theta)\Delta_{c_\theta}(y,z).
    \]
    Therefore,
    \[
    \Exp_{c_\theta}(u_t)
    =
    \gamma_{z,y}(t).
    \]
    The geodesic AVF discrete gradient is defined by
    \begin{equation}
    \mathcal{G}^{\mathrm{AVF},\theta}_f(y,z;c_\theta)
    :=
    \int_0^1
    \left(
    \mathrm{D}\Exp_{c_\theta}(u_t)
    \right)^{\ast}
    \operatorname{grad}f\!\left(\Exp_{c_\theta}(u_t)\right)\,dt,
    \label{eq:avf-gradient}
    \end{equation}
    where
    \[
    \left(
    \mathrm{D}\Exp_{c_\theta}(u_t)
    \right)^{\ast}
    :
    T_{\gamma_{z,y}(t)}\mathcal{X}
    \to
    T_{c_\theta}\mathcal{X}
    \]
    denotes the adjoint of the differential of the exponential map.

\item[(vi)] \textbf{Riemannian Gonzalez discrete gradient.}
Let
\[
m:=\operatorname{mid}(y,z),
\qquad
w:=\Delta_m(y,z).
\]
For $y\neq z$, define
\begin{equation}
\mathcal{G}^{\mathrm{G}}_f(y,z;m)
:=
\operatorname{grad}f(m)
+
\frac{
f(y)-f(z)
-
\left\langle
\operatorname{grad}f(m),w
\right\rangle_m
}{
\|w\|_m^2
}\,w,
\label{eq:gonzalez-gradient}
\end{equation}
with the continuous extension
\[
\mathcal{G}^{\mathrm{G}}_f(x,x;x)
=
\operatorname{grad}f(x).
\]

\item[(vii)] \textbf{Riemannian Itoh--Abe discrete gradient.}
Choose the center $c=z$ and let
\[
\{E_1,\ldots,E_d\}
\]
be an orthonormal basis of $T_z\mathcal{M}$. Write
\[
\Log_z y
=
\sum_{j=1}^d s_jE_j,
\]
and define
\[
u^{[0]}:=0,
\qquad
u^{[j]}:=\sum_{i=1}^j s_iE_i,
\qquad
w^{[j]}:=\Exp_z\bigl(u^{[j]}\bigr),
\qquad
j=1,\ldots,d.
\]
Thus
\[
w^{[0]}=z,
\qquad
w^{[d]}=y.
\]
For each $j=1,\ldots,d$, define the coordinate quotient
\begin{equation}
a_j(y,z)
:=
\begin{cases}
\dfrac{f(w^{[j]})-f(w^{[j-1]})}{s_j},
& s_j\neq0,\\[2ex]
\left\langle
\operatorname{grad}f(w^{[j-1]}),
\,\mathrm{D}(\Exp_z(u^{[j-1]}))[E_j]
\right\rangle_{w^{[j-1]}},
& s_j=0.
\end{cases}
\label{eq:ia-coordinate-quotient}
\end{equation}
The Riemannian Itoh--Abe discrete gradient is then defined by
\begin{equation}
\mathcal{G}^{\mathrm{IA}}_f(y,z;z)
:=
\sum_{j=1}^d a_j(y,z)E_j
\in T_z\mathcal{M}.
\label{eq:ia-gradient}
\end{equation}
\end{enumerate}

To derive explicit bounds for the parameters $\alpha$, $\beta$, and
$\gamma$ in the Riemannian Itoh--Abe and geodesic AVF constructions,
we first establish the following lemma.

At first, let
\[
K:=\max\{|K_{\min}|,|K_{\max}|\}.
\]
For $c\in\mathcal{X}$ and $u\in T_c\mathcal{M}$ such that
$q:=\Exp_c(u)\in\mathcal{X}$, define
\begin{equation}
\label{eq:exp-distortion-operator}
A_c(u)
:=
P_{q\to c}\circ\mathrm{D}\Exp_c(u),
\qquad
q:=\Exp_c(u),
\end{equation}
and
\begin{equation}
\label{eq:exp-distortion-constants}
\begin{aligned}
J_D
&:=
\sup_{\substack{c\in\mathcal{X},\ q=\Exp_c(u)\in\mathcal{X}\\
\|u\|_c\leq D}}
\|A_c(u)\|_{\mathrm{op}},
\qquad
H_D
:=
\sup_{\substack{c\in\mathcal{X},\ q=\Exp_c(u)\in\mathcal{X}\\
0<\|u\|_c\leq D}}
\frac{\|A_c(u)-I\|_{\mathrm{op}}}{\|u\|_c},
\\
\Lambda_D
&:=J_D+DH_D.
\end{aligned}
\end{equation}
Here, for a linear operator
$B:T_c\mathcal{M}\to T_c\mathcal{M}$, its operator norm is defined by
\begin{equation}
\label{eq:operator-norm-definition}
\|B\|_{\mathrm{op}}
:=
\sup_{v\neq 0}
\frac{\|Bv\|_c}{\|v\|_c}
=
\sup_{\|v\|_c=1}
\|Bv\|_c.
\end{equation}
The following lemma provides bounds for these constants.

\begin{lemma}
\label{lem:exp-distortion-bound}
Under Assumption~\ref{ass:geometry}, for $K>0$,
\[
J_D
\leq
\frac{\sinh(\sqrt{K}D)}{\sqrt{K}D},
\qquad
DH_D
\leq
\frac{\sinh(\sqrt{K}D)}{\sqrt{K}D}-1,
\]
and consequently
\begin{equation}
\label{eq:lambda-D-bound}
\Lambda_D
\leq
2\frac{\sinh(\sqrt{K}D)}{\sqrt{K}D}-1.
\end{equation}
For $K=0$, the right-hand sides are understood by continuity, and
\[
A_c(u)=I,
\qquad
J_D=1,
\qquad
H_D=0,
\qquad
\Lambda_D=1.
\]
\end{lemma}
\begin{proof}
See Appendix~\ref{app:exp-distortion-bound}.
\end{proof}

\begin{theorem}
\label{thm:structure-preserving-certificates}
Suppose that $f:\mathcal X\to\mathbb R$ is $L$-smooth and
$\mu$-strongly geodesically convex for all methods except the
Riemannian proximal-gradient method, for which the composite
assumptions in Definition~\ref{def:structure-preserving-methods}(iii)
hold. Then the methods in Definition~\ref{def:structure-preserving-methods}
admit the following c-wDRG parameters.

\begin{enumerate}

\item[(i)]
The Riemannian steepest-descent approximation
$\mathcal{G}^{\mathrm{SD}}_f$ is a c-wDRG with $c=z$ and
\[
(\alpha,\beta,\gamma)
=
\left(
\frac{L}{2},
\frac{\mu}{2},
0
\right).
\]

\item[(ii)]
The Riemannian proximal-point approximation
$\mathcal{G}^{\mathrm{PP}}_f$ is a c-wDRG with $c=y$ and
\[
(\alpha,\beta,\gamma)
=
\left(
0,
0,
\frac{\mu}{2}
\right).
\]

\item[(iii)] The Riemannian proximal-gradient approximation
$\mathcal G_{f,h}^{\mathrm{PG}}$ is a c-wDRG with
$c=z$ and
\[
(\alpha,\beta,\gamma)
=
\left(
\frac{L-\rho_{f_2}}{2},
\frac{\mu-\rho_{f_2}}{2},
0
\right).
\]

\item[(iv)]
The Riemannian implicit-midpoint approximation
$\mathcal{G}^{\mathrm{MP}}_f$ is a c-wDRG with
$c=m=\operatorname{mid}(y,z)$ and
\[
(\alpha,\beta,\gamma)
=
\left(
\frac{L+\mu}{8},
\frac{\mu}{4},
\frac{\mu}{4}
\right).
\]

\item[(v)]
If $\mu>0$, then, for every $0<\rho<\mu$, the geodesic AVF
discrete gradient $\mathcal{G}^{\mathrm{AVF},\theta}_f$ is a
c-wDRG with $c=c_\theta$ and
\[
\alpha
=
\frac{L+\mu-\rho}{2}\theta(1-\theta)
+
\frac{L^2\bar{\Lambda}_D^2\omega_\theta^2}{2\rho},
\qquad
\beta
=
\frac{\mu-\rho}{2}(1-\theta),
\qquad
\gamma
=
\frac{\mu-\rho}{2}\theta,
\]
where
\[
\bar{\Lambda}_D
:=
\begin{cases}
\displaystyle
2\frac{\sinh(\sqrt{K}D)}{\sqrt{K}D}-1,
& K>0,\\[2mm]
1,
& K=0,
\end{cases}
\qquad
K:=\max\{|K_{\min}|,|K_{\max}|\},
\]
is an explicit curvature-dependent upper bound for the geometric
distortion constant $\Lambda_D$ defined in Lemma~\ref{lem:exp-distortion-bound}, namely,
\[
\Lambda_D\leq\bar{\Lambda}_D.
\]
Moreover,
\begin{equation}
\omega_\theta
:=
\int_0^1|t-\theta|\,dt
=
\frac{\theta^2+(1-\theta)^2}{2}.
\label{eq:omega-theta}
\end{equation}

\item[(vi)]
If $\mu>0$, the Riemannian Gonzalez discrete gradient
$\mathcal{G}^{\mathrm{G}}_f$ is a c-wDRG with
$c=m=\operatorname{mid}(y,z)$ and
\[
(\alpha,\beta,\gamma)
=
\left(
\frac{L+\mu}{8}
+
\frac{(L-\mu)^2}{16\mu},
\frac{\mu}{4},
0
\right).
\]

\item[(vii)]
 If $\mu>0$, then,
for every $0<\rho<\mu$, the Riemannian Itoh--Abe discrete gradient
$\mathcal{G}^{\mathrm{IA}}_f$ is a c-wDRG with $c=z$ and
\[
(\alpha,\beta,\gamma)
=
\left(
\frac{dL^2\bar{\Lambda}_D^2}{2\rho},
\frac{\mu-\rho}{2},
0
\right).
\]
\begin{proof}
See Appendix~\ref{app:proof-structure-preserving-certificates}.
\end{proof}

\end{enumerate}
\end{theorem}

Compared with their Euclidean counterparts in~\cite{Ushiyama2023}, the geodesic AVF and Riemannian
Itoh--Abe certificates acquire an additional geometric factor
$\bar\Lambda_D^2$ in the coefficient $\alpha$. The reason is that $\mathrm D\Exp=I_d$ in Euclidean space, while on a
Riemannian manifold the differential of the exponential map is controlled
only through the curvature-dependent bound
$\|\mathrm D\Exp_c(u)[v]\|
\leq
\bar\Lambda_D\|v\|.$ The parameter $\rho\in(0,\mu)$ is introduced through Young's
inequality to balance the remainder arising from the $L$-smoothness
estimate with the strong geodesic-convexity term, resulting in the
residual coefficient $\mu-\rho$.

\section{Convergence analysis}
\label{sec:convergence}

We now establish convergence guarantees for the c-wDRG methods
introduced in Section~\ref{sec:examples}, considering both their nonaccelerated
and accelerated variants. Our analysis derives convergence rates in the
geodesically convex and geodesically strongly convex settings, with
the effects of the manifold geometry explicitly reflected through the
curvature-dependent constants $\delta_D$ and $\zeta_D$.

A key ingredient in the analysis is the comparison of squared
Riemannian distances when the base point moves along a curve. We begin
by recalling the following Hessian comparison estimate, which provides
the geometric bounds needed throughout the subsequent convergence
arguments.

\begin{lemma}
\label{lem:hessian-comparison-squared-distance}
(Alimisis et al.~\cite{Alimisis2020}, Lemma~2)
Let $\gamma:I\to\mathcal X$ be a smooth curve and let $p\in\mathcal X$.
Then, for every $t\in I$,
\[
\delta_D\|\dot{\gamma}(t)\|^2
\leq
\left\langle
\nabla_{\dot{\gamma}(t)}\Log_{\gamma(t)}p,
-\dot{\gamma}(t)
\right\rangle_{\gamma(t)}
\leq
\zeta_D\|\dot{\gamma}(t)\|^2.
\]
\end{lemma}

Applying Lemma~\ref{lem:hessian-comparison-squared-distance} along the minimizing geodesic from $c$ to $q$
yields the following two-sided comparison for the squared distance.

\begin{corollary}
\label{cor:normal-coordinate-comparison}
Let $c,q,p\in\mathcal{X}$. Then
\[
\|\Log_c p\|_c^2
-2\left\langle \Log_c q,\Log_c p\right\rangle_c
+\delta_D\|\Log_c q\|_c^2
\leq
d^2(q,p)
\leq
\|\Log_c p\|_c^2
-2\left\langle \Log_c q,\Log_c p\right\rangle_c
+\zeta_D\|\Log_c q\|_c^2.
\]
Equivalently,
\[
(1-\zeta_D)\|\Log_c q\|_c^2
\leq
\|\Delta_c(q,p)\|_c^2-d^2(q,p)
\leq
(1-\delta_D)\|\Log_c q\|_c^2.
\]
\end{corollary}

\begin{proof}
Consider the radial geodesic
\[
\gamma(t)=\Exp_c\bigl(t\Log_c q\bigr),
\qquad t\in[0,1].
\]
Then
\[
\gamma(0)=c,
\qquad
\dot{\gamma}(0)=\Log_c q,
\qquad
\|\dot{\gamma}(t)\|_{\gamma(t)}
=
\|\Log_c q\|_c.
\]
Define
\[
\phi(t):=\frac12 d^2(\gamma(t),p).
\]
By the first variation formula for the squared Riemannian distance
(see, e.g., Lemma~12 in \cite{Alimisis2020}),
\[
\phi'(t)
=
\left\langle
\Log_{\gamma(t)}p,
-\dot{\gamma}(t)
\right\rangle_{\gamma(t)}.
\]
Hence
\[
\phi(0)=\frac12\|\Log_c p\|_c^2,
\qquad
\phi'(0)
=
-\left\langle\Log_c p,\Log_c q\right\rangle_c.
\]
Since $\gamma$ is a geodesic,
$\nabla_{\dot{\gamma}(t)}\dot{\gamma}(t)=0$, and therefore
\[
\phi''(t)
=
\left\langle
\nabla_{\dot{\gamma}(t)}\Log_{\gamma(t)}p,
-\dot{\gamma}(t)
\right\rangle_{\gamma(t)}.
\]
By Lemma~\ref{lem:hessian-comparison-squared-distance},
\[
\delta_D\|\Log_c q\|_c^2
\leq
\phi''(t)
\leq
\zeta_D\|\Log_c q\|_c^2.
\]
Taylor's formula with integral remainder gives
\[
\phi(1)
=
\phi(0)+\phi'(0)
+
\int_0^1(1-t)\phi''(t)\,dt.
\]
Using
\[
\int_0^1(1-t)\,dt=\frac12
\]
and $\phi(1)=\frac12 d^2(q,p)$, we obtain
\[
\|\Log_c p\|_c^2
-2\left\langle \Log_c q,\Log_c p\right\rangle_c
+\delta_D\|\Log_c q\|_c^2
\leq
d^2(q,p)
\leq
\|\Log_c p\|_c^2
-2\left\langle \Log_c q,\Log_c p\right\rangle_c
+\zeta_D\|\Log_c q\|_c^2.
\]

Finally, since
\[
\|\Delta_c(q,p)\|_c^2
=
\|\Log_c p\|_c^2
-2\left\langle\Log_c q,\Log_c p\right\rangle_c
+\|\Log_c q\|_c^2,
\]
subtracting the above two-sided estimate from this identity yields
\[
(1-\zeta_D)\|\Log_c q\|_c^2
\leq
\|\Delta_c(q,p)\|_c^2-d^2(q,p)
\leq
(1-\delta_D)\|\Log_c q\|_c^2.
\]
\end{proof}

\subsection{Convergence of Nonaccelerated c-wDRG Methods}
\label{subsec:nonaccelerated-convergence}

In this subsection, we analyze the convergence of the nonaccelerated c-wDRG methods in both the geodesically convex and geodesically strongly convex settings. The iterates are generated according to the centered
update~\eqref{eq:centered-update}, and the corresponding convergence rates are established in
Corollaries~\ref{cor:convex-rate} and
\ref{cor:strongly-convex-rate}, respectively.

Let $x^\star\in\operatorname*{argmin}_{x\in\mathcal{X}} f(x)$, and let
\[
c_\theta
=
\gamma_{x_k,x_{k+1}}(\theta),
\qquad
0\leq\theta\leq 1,
\qquad
c_0=x_k,\quad c_1=x_{k+1}.
\]
where $\gamma_{x_k,x_{k+1}}$ is the geodesic
from $x_k$ to $x_{k+1}$. Since $c_\theta$ lies on this geodesic,
\[
\Log_{c_\theta}x_k
=
-\theta
\Delta_{c_\theta}(x_{k+1},x_k),
\qquad
\Log_{c_\theta}x_{k+1}
=
(1-\theta)
\Delta_{c_\theta}(x_{k+1},x_k).
\]
Suppose that $\mathcal{G}_f$ is an $(\alpha,\beta,\gamma)$-c-wDRG and that
$x_{k+1}$ satisfies
\begin{equation}
\label{eq:centered-update}
\Delta_{c_\theta}(x_{k+1},x_k)
=
-h\,\mathcal{G}_f(x_{k+1},x_k;c_\theta).
\end{equation}
which defines the update algorithm for the nonaccelerated c-wDRG method.

Moreover, applying Corollary~\ref{cor:normal-coordinate-comparison} with
$(c,q,p)=(c_\theta,x_k,x^\star)$ gives
\begin{equation}
\label{eq:distortion-old-endpoint}
(1-\zeta_D)\theta^2\|\Delta_{c_\theta}(x_{k+1},x_k)\|_{c_\theta}^2
\leq
\|\Delta_{c_\theta}(x_k,x^\star)\|_{c_\theta}^2-d^2(x_k,x^\star)
\leq
(1-\delta_D)\theta^2\|\Delta_{c_\theta}(x_{k+1},x_k)\|_{c_\theta}^2.
\end{equation}

Similarly, applying Corollary~\ref{cor:normal-coordinate-comparison} with
$(c,q,p)=(c_\theta,x_{k+1},x^\star)$ yields
\begin{equation}
\label{eq:distortion-new-endpoint}
(1-\zeta_D)(1-\theta)^2\|\Delta_{c_\theta}(x_{k+1},x_k)\|_{c_\theta}^2
\leq
\|\Delta_{c_\theta}(x_{k+1},x^\star)\|_{c_\theta}^2-d^2(x_{k+1},x^\star)
\leq
(1-\delta_D)(1-\theta)^2\|\Delta_{c_\theta}(x_{k+1},x_k)\|_{c_\theta}^2.
\end{equation}

Therefore, the following corollary follows directly from the preceding estimates.

\begin{corollary}
\label{cor:center-compatible-curvature-correction}
For every admissible step $h$, the curvature distortion satisfies
\[
\begin{aligned}
&\left(\frac{1}{2h}-\beta\right)
\left[
\|\Delta_{c_\theta}(x_k,x^\star)\|_{c_\theta}^2
-d^2(x_k,x^\star)
\right]
-
\left(\frac{1}{2h}+\gamma\right)
\left[
\|\Delta_{c_\theta}(x_{k+1},x^\star)\|_{c_\theta}^2
-d^2(x_{k+1},x^\star)
\right]
\\
&\qquad\leq
\chi_h^{(\theta)}
\|\Delta_{c_\theta}(x_{k+1},x_k)\|_{c_\theta}^2.
\end{aligned}
\]
where
\[
\begin{aligned}
\chi_h^{(\theta)}
={}&
\theta^2
\left[
\left[\frac{1}{2h}-\beta\right]_+
(1-\delta_D)
+
\left[\frac{1}{2h}-\beta\right]_-
(\zeta_D-1)
\right]
\\
&+
(1-\theta)^2
\left[
\left[\frac{1}{2h}+\gamma\right]_+
(\zeta_D-1)
+
\left[\frac{1}{2h}+\gamma\right]_-
(1-\delta_D)
\right],
\end{aligned}
\]
with
\[
[a]_+:=\max\{a,0\},
\qquad
[a]_-:=\max\{-a,0\}.
\]
In particular, if
\[
\frac{1}{2h}-\beta\geq0,
\qquad
\frac{1}{2h}+\gamma\geq0,
\]
then
\[
\chi_h^{(\theta)}
=
\left(\frac{1}{2h}-\beta\right)(1-\delta_D)\theta^2
+
\left(\frac{1}{2h}+\gamma\right)(\zeta_D-1)(1-\theta)^2.
\]
\end{corollary}

\begin{proof}
By \eqref{eq:distortion-old-endpoint},
\[
(1-\zeta_D)\theta^2
\|\Delta_{c_\theta}(x_{k+1},x_k)\|_{c_\theta}^2
\leq
\|\Delta_{c_\theta}(x_k,x^\star)\|_{c_\theta}^2
-d^2(x_k,x^\star)
\leq
(1-\delta_D)\theta^2
\|\Delta_{c_\theta}(x_{k+1},x_k)\|_{c_\theta}^2,
\]
while \eqref{eq:distortion-new-endpoint} gives
\[
(1-\zeta_D)(1-\theta)^2
\|\Delta_{c_\theta}(x_{k+1},x_k)\|_{c_\theta}^2
\leq
\|\Delta_{c_\theta}(x_{k+1},x^\star)\|_{c_\theta}^2
-d^2(x_{k+1},x^\star)
\leq
(1-\delta_D)(1-\theta)^2
\|\Delta_{c_\theta}(x_{k+1},x_k)\|_{c_\theta}^2.
\]
For each coefficient, the appropriate endpoint of the corresponding
interval is selected according to its sign. Separating the positive and
negative parts therefore gives exactly the stated expression for
$\chi_h^{(\theta)}$. The simplified formula follows immediately when
$\frac{1}{2h}-\beta\geq0$ and $\frac{1}{2h}+\gamma\geq0$.
\end{proof}

We next establish a general convergence theorem, from which the iteration complexities of the nonaccelerated c-wDRG methods follow as corollaries.

\begin{theorem}
\label{thm:unified-wdrg-estimate}
Let the iterates be generated by \eqref{eq:centered-update}, and assume
that the c-wDRG inequality \eqref{eq:c-wdrg} and the
curvature-correction estimate in Corollary~\ref{cor:center-compatible-curvature-correction} hold, together with
$
1-(\alpha-\gamma)h>0.
$
Define
\[
\lambda_h
:=
\frac{\left[\alpha h-\frac{1}{2}+h\chi_h^{(\theta)}\right]_+}{1-(\alpha-\gamma)h}.
\]
Then
\begin{equation}
\label{eq:one-step-estimate}
(1+\lambda_h)\bigl(f(x_{k+1})-f(x^\star)\bigr)
-\lambda_h\bigl(f(x_k)-f(x^\star)\bigr)
+\beta d^2(x_k,x^\star)
+\gamma d^2(x_{k+1},x^\star)
\leq
\frac{d^2(x_k,x^\star)-d^2(x_{k+1},x^\star)}{2h}.
\end{equation}
Moreover, assume that $1+2\gamma h>0$, define the Lyapunov function
\[
\mathcal{E}_k
:=
(1+\lambda_h)\bigl(f(x_k)-f(x^\star)\bigr)
+
\left(\frac{1}{2h}+\gamma\right)d^2(x_k,x^\star).
\]
Then
\[
\mathcal{E}_{k+1}
\leq
q_h\mathcal{E}_k,
\]
where
\[
q_h
:=
\max\left\{
\frac{\lambda_h}{1+\lambda_h},
\left[
\frac{1-2\beta h}{1+2\gamma h}
\right]_+
\right\}.
\]
\end{theorem}

\begin{proof}
By the additivity of the centered logarithmic difference,
\[
\Delta_{c_\theta}(x_{k+1},x^\star)
=
\Delta_{c_\theta}(x_{k+1},x_k)
+
\Delta_{c_\theta}(x_k,x^\star).
\]
Hence,
\[
2\left\langle \Delta_{c_\theta}(x_{k+1},x_k),\Delta_{c_\theta}(x_{k+1},x^\star)\right\rangle_{c_\theta}
=
\|\Delta_{c_\theta}(x_{k+1},x^\star)\|_{c_\theta}^2
-
\|\Delta_{c_\theta}(x_k,x^\star)\|_{c_\theta}^2
+
\|\Delta_{c_\theta}(x_{k+1},x_k)\|_{c_\theta}^2.
\]
By \eqref{eq:centered-update},
we obtain
\[
\left\langle \mathcal{G}_f(x_{k+1},x_k;c_\theta),\Delta_{c_\theta}(x_{k+1},x^\star)\right\rangle_{c_\theta}
=
\frac{1}{2h}\left[\|\Delta_{c_\theta}(x_k,x^\star)\|_{c_\theta}^2-\|\Delta_{c_\theta}(x_{k+1},x^\star)\|_{c_\theta}^2-\|\Delta_{c_\theta}(x_{k+1},x_k)\|_{c_\theta}^2\right].
\]
Taking $y=x_{k+1}$, $z=x_k$, and $x=x^\star$ in \eqref{eq:c-wdrg} gives
\[
f(x_{k+1})-f(x^\star)
\leq
\left(\frac{1}{2h}-\beta\right)\|\Delta_{c_\theta}(x_k,x^\star)\|_{c_\theta}^2
-\left(\frac{1}{2h}+\gamma\right)\|\Delta_{c_\theta}(x_{k+1},x^\star)\|_{c_\theta}^2
+\left(\alpha-\frac{1}{2h}\right)\|\Delta_{c_\theta}(x_{k+1},x_k)\|_{c_\theta}^2.
\]
By Corollary~\ref{cor:center-compatible-curvature-correction}, we obtain
\[
f(x_{k+1})-f(x^\star)
+\beta d^2(x_k,x^\star)
+\gamma d^2(x_{k+1},x^\star)
\leq
\frac{d^2(x_k,x^\star)-d^2(x_{k+1},x^\star)}{2h}
+
\left(\alpha-\frac{1}{2h}+\chi_h^{(\theta)}\right)
\|\Delta_{c_\theta}(x_{k+1},x_k)\|_{c_\theta}^2.
\]
By \eqref{eq:c-wdrg-chain-rule} and \eqref{eq:centered-update},
\begin{equation}
\label{eq:c-wdrg-descent}
\begin{aligned}
f(x_{k+1})-f(x_k)
&\leq
\left\langle
\mathcal{G}_f(x_{k+1},x_k;c_\theta),
\Delta_{c_\theta}(x_{k+1},x_k)
\right\rangle_{c_\theta}
+
(\alpha-\gamma)
\|\Delta_{c_\theta}(x_{k+1},x_k)\|_{c_\theta}^2
\\
&=
-\left(\frac{1}{h}-\alpha+\gamma\right)
\|\Delta_{c_\theta}(x_{k+1},x_k)\|_{c_\theta}^2.
\end{aligned}
\end{equation}
Since $1-(\alpha-\gamma)h>0$,
\[
\|\Delta_{c_\theta}(x_{k+1},x_k)\|_{c_\theta}^2
\leq
\frac{h}{1-(\alpha-\gamma)h}
\left[f(x_k)-f(x_{k+1})\right].
\]
Consequently, by the definition of $\lambda_h$,
\[
\left(\alpha-\frac{1}{2h}+\chi_h^{(\theta)}\right)
\|\Delta_{c_\theta}(x_{k+1},x_k)\|_{c_\theta}^2
\leq
\lambda_h\left[f(x_k)-f(x_{k+1})\right].
\]
Substitution gives
\[
(1+\lambda_h)\bigl(f(x_{k+1})-f(x^\star)\bigr)
-\lambda_h\bigl(f(x_k)-f(x^\star)\bigr)
+\beta d^2(x_k,x^\star)
+\gamma d^2(x_{k+1},x^\star)
\leq
\frac{d^2(x_k,x^\star)-d^2(x_{k+1},x^\star)}{2h},
\]
which proves the first claim. Rearranging the above inequality yields
\[
(1+\lambda_h)\bigl(f(x_{k+1})-f(x^\star)\bigr)
+
\left(\frac{1}{2h}+\gamma\right)d^2(x_{k+1},x^\star)
\leq
\lambda_h\bigl(f(x_k)-f(x^\star)\bigr)
+
\left(\frac{1}{2h}-\beta\right)d^2(x_k,x^\star).
\]
The left-hand side is $\mathcal{E}_{k+1}$. Moreover,
by the definition of $q_h$,
\[
\lambda_h\leq q_h(1+\lambda_h),
\qquad
\frac{1}{2h}-\beta
\leq
q_h\left(\frac{1}{2h}+\gamma\right).
\]
Therefore,
\[
\mathcal{E}_{k+1}
\leq
q_h\left[(1+\lambda_h)\bigl(f(x_k)-f(x^\star)\bigr)+\left(\frac{1}{2h}+\gamma\right)d^2(x_k,x^\star)\right]
=
q_h\mathcal{E}_k.
\]
\end{proof}
\begin{corollary}[Convex case]
\label{cor:convex-rate}
Under the assumptions of Theorem~\ref{thm:unified-wdrg-estimate},
assume further that $\beta,\gamma\geq0$. Then, for every $k\geq1$,
\[
f(x_k)-f(x^\star)
\leq
\frac{\dfrac{d^2(x_0,x^\star)}{2h}
+\lambda_h\bigl(f(x_0)-f(x^\star)\bigr)}
{k+\lambda_h}.
\]
\end{corollary}

\begin{proof}
By \eqref{eq:c-wdrg-descent},
\[
f(x_{j+1})-f(x_j)
\leq
-\left(\frac{1}{h}-\alpha+\gamma\right)
\|\Delta_{c_\theta}(x_{j+1},x_j)\|_{c_\theta}^2.
\]
Since $1-(\alpha-\gamma)h>0$, the sequence $\{f(x_j)\}$ is
nonincreasing.
Summing \eqref{eq:one-step-estimate} from $j=0$ to $k-1$ and using
$\beta,\gamma\geq0$, $d^2(x_k,x^\star)\geq0$, and the monotonicity of
$\{f(x_j)\}$, we obtain
\[
k\bigl(f(x_k)-f(x^\star)\bigr)
+\lambda_h\bigl(f(x_k)-f(x_0)\bigr)
\leq
\frac{d^2(x_0,x^\star)}{2h}.
\]
Therefore,
\[
(k+\lambda_h)\bigl(f(x_k)-f(x^\star)\bigr)
\leq
\frac{d^2(x_0,x^\star)}{2h}
+\lambda_h\bigl(f(x_0)-f(x^\star)\bigr),
\]
which proves the claim.
\end{proof}

\begin{corollary}[Strongly convex case]
\label{cor:strongly-convex-rate}
Under the assumptions of Theorem~\ref{thm:unified-wdrg-estimate},
assume further that
$\beta+\gamma>0,$ 
then
\[
q_h
=
\max\left\{
\frac{\lambda_h}{1+\lambda_h},
\left[
\frac{1-2\beta h}{1+2\gamma h}
\right]_+
\right\}
<1.
\]
Therefore, for every $k\geq0$,
\[
f(x_k)-f(x^\star)
\leq
\left[
f(x_0)-f(x^\star)
+
\frac{\frac{1}{2h}+\gamma}{1+\lambda_h}
d^2(x_0,x^\star)
\right]q_h^k.
\]
\end{corollary}

\begin{proof}
By Theorem~\ref{thm:unified-wdrg-estimate},
\[
\mathcal{E}_{k+1}\leq q_h\mathcal{E}_k.
\]
Iterating this inequality gives
\[
\mathcal{E}_k\leq q_h^k\mathcal{E}_0.
\]
Since
\[
\mathcal{E}_k
=
(1+\lambda_h)\bigl(f(x_k)-f(x^\star)\bigr)
+
\left(\frac{1}{2h}+\gamma\right)d^2(x_k,x^\star)
\geq
(1+\lambda_h)\bigl(f(x_k)-f(x^\star)\bigr),
\]
we obtain
\[
f(x_k)-f(x^\star)
\leq
\frac{\mathcal{E}_k}{1+\lambda_h}
\leq
\frac{q_h^k\mathcal{E}_0}{1+\lambda_h}.
\]
Substituting the definition of $\mathcal{E}_0$ yields
\[
f(x_k)-f(x^\star)
\leq
\left[
f(x_0)-f(x^\star)
+
\frac{\frac{1}{2h}+\gamma}{1+\lambda_h}
d^2(x_0,x^\star)
\right]q_h^k.
\]
\end{proof}

For the methods considered in
Theorem~\ref{thm:structure-preserving-certificates}, the optimal admissible
stepsizes minimizing $q_h$ and the corresponding contraction factors are
given explicitly in Table~\ref{tab:best-step-rate} of
Appendix~\ref{app:optimal-strongly-convex-rates}.

\begin{remark}
\label{rem:recover-ushiyama-53}
In Euclidean space, the curvature correction vanishes, i.e., $\chi_h=0$.
Accordingly, the first branch of the contraction factor reduces to
\[
\frac{\lambda_h}{1+\lambda_h}
=
\frac{[2\alpha h-1]_+}{1+2\gamma h},
\]
and hence
\[
q_h
=
\max\left\{
\frac{[2\alpha h-1]_+}{1+2\gamma h},
\frac{[1-2\beta h]_+}{1+2\gamma h}
\right\}.
\]
For $0<h\le 1/(\alpha+\beta)$, the first branch is no larger than the
second, so that
\[
q_h
=
\frac{1-2\beta h}{1+2\gamma h}
=
1-\frac{2(\beta+\gamma)h}{1+2\gamma h}.
\]
The two branches are balanced at
\[
h_\star=\frac{1}{\alpha+\beta},
\qquad
q_\star
=
1-\frac{2(\beta+\gamma)}
{\alpha+\beta+2\gamma}.
\]
Therefore, in the Euclidean setting, our curvature-aware strongly convex
result recovers exactly Ushiyama's Theorem~5.3~\cite{Ushiyama2023}.
\end{remark}

\begin{remark}
\label{rem:recovery-linear-rates}
For $\mu$-strongly geodesically convex objectives,
our nonaccelerated analysis recovers existing linear
convergence factors under the corresponding assumptions.

For Riemannian steepest descent, choosing $h=1/L$ gives
\[
\lambda_h
=
\frac{[Lh+\zeta_D-2]_+}{2-Lh}
=
\zeta_D-1.
\]
Consequently,
\[
q_h
=
\max\left\{
\frac{\zeta_D-1}{\zeta_D},
1-\frac{\mu}{L}
\right\}
=
1-\min\left\{
\frac{1}{\zeta_D},\frac{\mu}{L}
\right\},
\]
recovering the factor in
\cite[Thm.~15]{zhang2016first}
in the common Hadamard setting.

For Riemannian proximal point, assume
$D\leq\pi/(2\sqrt{K_{\max}})$ when $K_{\max}>0$,
so that $\delta_D\geq0$; when $K_{\max}\leq0$,
we have $\delta_D=1$.
Then any fixed step size $h>0$ gives
\[
\lambda_h
=
\frac{\left[-\frac12+\frac{1-\delta_D}{2}\right]_+}
{1+\frac{\mu h}{2}}
=
\frac{[-\delta_D/2]_+}{1+\frac{\mu h}{2}}
=
0,
\qquad
q_h=\frac{1}{1+\mu h}.
\]
The same factor follows by combining
\cite[Lem.~4.6(I), Eq.~(4.5)]{ohta2015discrete}
with geodesic strong convexity under their domain assumptions $D<\pi/(2\sqrt{K_{\max}})$.
Our estimate also permits the boundary case $\delta_D=0$.
\end{remark}



\subsection{Convergence analysis of Accelerated c-wDRG Methods}
\label{sec:accelerated-cwdrg}

In this subsection, we study the convergence properties of the accelerated c-wDRG methods. 
The algorithms for the geodesically convex and geodesically strongly convex settings are given in
Algorithms~\ref{alg:accelerated-cwdrg-convex} and
\ref{alg:accelerated-cwdrg-strong}, respectively.
We then establish their corresponding convergence guarantees in
Theorems~\ref{thm:cwdrg-accelerated-convex} and
\ref{thm:cwdrg-accelerated-strong}.
When the c-wDRG is instantiated with the Riemannian gradient, these algorithms recover the corresponding accelerated Riemannian gradient methods of Kim and Yang~\cite{KimYang2022}, showing that their methods arise as special cases of our accelerated c-wDRG framework.

First of all, the convergence analysis relies on the following two metric-distortion inequalities.

\subsubsection{Metric distortion lemmas}
\label{subsec:kim-yang-distortion}

\begin{lemma}[(Lemma~5.2 in \cite{KimYang2022})]
\label{lem:kim-yang-distortion-I}
Let $p_A,p_B,x\in\mathcal X$ and $v_A\in T_{p_A}\mathcal M$.
Suppose that there exists $r\in[0,1]$ such that
\[
\Log_{p_A}p_B=r v_A.
\]
Define
\[
v_B
:=
P_{p_A\to p_B}
\bigl(v_A-\Log_{p_A}p_B\bigr)
\in T_{p_B}\mathcal M.
\]
Then
\begin{equation}
\label{eq:kim-yang-distortion-I}
\bigl\|v_B-\Log_{p_B}x\bigr\|_{p_B}^2
+(\zeta_D-1)\|v_B\|_{p_B}^2
\le
\bigl\|v_A-\Log_{p_A}x\bigr\|_{p_A}^2
+(\zeta_D-1)\|v_A\|_{p_A}^2.
\end{equation}
\end{lemma}

\begin{lemma}[(Lemma~5.3 in \cite{KimYang2022})]
\label{lem:kim-yang-distortion-II}
Let $p_A,p_B,x\in\mathcal X$ and $v_A\in T_{p_A}\mathcal M$, and define
\[
v_B
:=
P_{p_A\to p_B}
\bigl(v_A-\Log_{p_A}p_B\bigr)
\in T_{p_B}\mathcal M.
\]
Suppose that there exist $a,b\in T_{p_A}\mathcal M$ and $r\in(0,1)$
such that
\[
v_A=a+b,
\qquad
\Log_{p_A}p_B=rb.
\]
Then, for every $\xi\ge\zeta_D$,
\begin{equation}
\label{eq:kim-yang-distortion-II}
\|v_B-\Log_{p_B}x\|_{p_B}^2
+(\xi-1)\|v_B\|_{p_B}^2
\le
\|v_A-\Log_{p_A}x\|_{p_A}^2
+(\xi-1)\|v_A\|_{p_A}^2
+\frac{\xi-\delta_D}{2}
\left(\frac{1}{1-r}-1\right)
\|a\|_{p_A}^2.
\end{equation}
\end{lemma}

Lemma~\ref{lem:kim-yang-distortion-I} controls the change of the
momentum energy when the extrapolation point is moved from the current
iterate to the accelerated center. Lemma~\ref{lem:kim-yang-distortion-II}
controls the second change of tangent space when the accelerated center
is moved to the new iterate. The additional term in
\eqref{eq:kim-yang-distortion-II} quantifies the metric distortion
caused by curvature. To derive the convergence rates of the accelerated c-wDRG methods,
we first establish the following lemma.

\begin{lemma}
\label{thm:endpoint-reduction}
Let $\theta\in[0,1]$ satisfy
\[
1-2(1-\delta_D)\theta>0,
\]
and define
\begin{equation}
\label{eq:endpoint-effective-constants}
L_{\mathrm R}
:=
\frac{
2\left[
\alpha+\gamma
-(1-\delta_D)(\beta+\gamma)\theta^2
\right]
}{
1-2(1-\delta_D)\theta
},
\qquad
\mu_{\mathrm R}
:=
2(\beta+\gamma).
\end{equation}
Assume
\[
\beta+\gamma\geq0,
\qquad
L_{\mathrm R}\geq0.
\]
For any $h>0$ satisfying
\[
h>0,
\qquad
hL_{\mathrm R}\leq1,
\qquad
1-2\gamma h>0,
\qquad
h\theta\mu_{\mathrm R}\leq1,
\]
define
\begin{equation}
\label{eq:endpoint-chord-step}
s_h
:=
\frac{h}{1-2\gamma h}.
\end{equation}
Let $y,z\in\mathcal X$, set
\[
c:=\gamma_{z,y}(\theta),
\]
and suppose that
\begin{equation}
\label{eq:endpoint-cwdrg-step}
\Delta_c(y,z)
=
-s_h\mathcal G_f(y,z;c).
\end{equation}
Then
\begin{equation}
\label{eq:endpoint-step-representation}
y
=
\Exp_z\left(
-\frac{h}{1-2\gamma h}
P_{c\to z}\mathcal G_f(y,z;c)
\right),
\end{equation}
and, for every $x\in\mathcal X$,
\begin{equation}
\label{eq:endpoint-oracle-model}
f(y)-f(x)
\leq
-\frac{1}{1-2\gamma h}
\left\langle P_{c\to z}\mathcal G_f(y,z;c),\Log_z x\right\rangle_z
-\frac{h}{2(1-2\gamma h)^2}
\left\|\mathcal G_f(y,z;c)\right\|_c^2
-\frac{\mu_{\mathrm R}}{2}d^2(z,x).
\end{equation}
\end{lemma}

\begin{proof}
Set
\[
G:=\mathcal G_f(y,z;c).
\]
Since
\[
c=\gamma_{z,y}(\theta),
\]
we have
\[
\Log_c z
=
-\theta\Delta_c(y,z),
\qquad
\Log_c y
=
(1-\theta)\Delta_c(y,z).
\]
By \eqref{eq:endpoint-cwdrg-step},
\[
\Log_c z
=
\theta s_hG,
\qquad
\Log_c y
=
-(1-\theta)s_hG.
\]
Consequently,
\[
\Delta_c(y,x)
=
-(1-\theta)s_hG-\Log_cx,
\]
and
\[
\Delta_c(z,x)
=
\theta s_hG-\Log_cx.
\]

Substituting these identities into \eqref{eq:c-wdrg} gives
\begin{equation}
\label{eq:endpoint-center-estimate}
\begin{aligned}
f(y)-f(x)
\leq{}&
-
\left(
1-2s_h[\beta\theta-\gamma(1-\theta)]
\right)
\left\langle G,\Log_cx\right\rangle_c
\\
&-
s_h
\left[
(1-\theta)
-
s_h\left\{
\alpha-\beta\theta^2-\gamma(1-\theta)^2
\right\}
\right]
\|G\|_c^2
-
(\beta+\gamma)d^2(c,x).
\end{aligned}
\end{equation}

Since
\[
1-2s_h[\beta\theta-\gamma(1-\theta)]
=
\frac{1-h\theta\mu_{\mathrm R}}
{1-2\gamma h}
\geq0,
\]
the coefficient of the first term is nonnegative.

Moreover, since $c$ lies on the geodesic from $z$ to $y$,
parallel transport along this geodesic gives
\[
\Log_z y
=
-s_hP_{c\to z}G.
\]

For $\theta>0$, consider
\[
\sigma(t):=\gamma_{z,c}(t),
\qquad
\varphi(t):=\frac12d^2(\sigma(t),x).
\]
Then
\[
\dot\sigma(0)
=
-\theta s_hP_{c\to z}G,
\qquad
\dot\sigma(1)
=
-\theta s_hG.
\]
Since
\[
\operatorname{grad}_p\frac12d^2(p,x)
=
-\Log_px,
\]
we obtain
\[
\varphi'(0)
=
\theta s_h
\left\langle
P_{c\to z}G,\Log_zx
\right\rangle_z,
\]
and
\[
\varphi'(1)
=
\theta s_h
\left\langle
G,\Log_cx
\right\rangle_c.
\]
By Lemma~\ref{lem:hessian-comparison-squared-distance},
\[
\varphi''(t)
\geq
\delta_D\theta^2s_h^2\|G\|_c^2.
\]
Hence
\begin{equation}
\label{eq:endpoint-inner-product-comparison}
\left\langle
G,\Log_cx
\right\rangle_c
\geq
\left\langle
P_{c\to z}G,\Log_zx
\right\rangle_z
+
\delta_D\theta s_h\|G\|_c^2,
\end{equation}
and
\begin{equation}
\label{eq:endpoint-distance-comparison}
d^2(c,x)
\geq
d^2(z,x)
+
2\theta s_h
\left\langle
P_{c\to z}G,\Log_zx
\right\rangle_z
+
\delta_D\theta^2s_h^2\|G\|_c^2.
\end{equation}
For $\theta=0$, both inequalities hold with equality.

Substituting
\eqref{eq:endpoint-inner-product-comparison}
and
\eqref{eq:endpoint-distance-comparison}
into \eqref{eq:endpoint-center-estimate} yields
\begin{equation}
\label{eq:endpoint-pre-oracle}
\begin{aligned}
f(y)-f(x)
\leq{}&
-(1+2\gamma s_h)
\left\langle
P_{c\to z}G,\Log_z x
\right\rangle_z
-(\beta+\gamma)d^2(z,x)
\\
&-
s_h
\left[
(1-\theta)+\delta_D\theta
-
s_h
\left\{
\alpha-\gamma
+
(1-\delta_D)
\bigl[
\gamma\theta(2-\theta)-\beta\theta^2
\bigr]
\right\}
\right]
\|G\|_c^2.
\end{aligned}
\end{equation}
Using
\[
s_h=\frac{h}{1-2\gamma h},
\]
a direct calculation gives
\begin{equation}
\label{eq:endpoint-quadratic-gap}
\begin{aligned}
&
s_h
\left[
(1-\theta)+\delta_D\theta
-
s_h
\left\{
\alpha-\gamma
+
(1-\delta_D)
\bigl[
\gamma\theta(2-\theta)-\beta\theta^2
\bigr]
\right\}
\right]
-
\frac{h}{2(1-2\gamma h)^2}
\\
&\qquad=
\frac{
h\bigl[1-2(1-\delta_D)\theta\bigr]
}{
2(1-2\gamma h)^2
}
\left(1-hL_{\mathrm R}\right)
\geq 0.
\end{aligned}
\end{equation}
Therefore,
\[
f(y)-f(x)
\leq
-\frac{1}{1-2\gamma h}
\left\langle P_{c\to z}G,\Log_z x\right\rangle_z
-\frac{h}{2(1-2\gamma h)^2}\|G\|_c^2
-\frac{\mu_{\mathrm R}}{2}d^2(z,x),
\]
which proves \eqref{eq:endpoint-oracle-model}.

Finally,
\[
\Log_zy
=
-s_hP_{c\to z}G
=
-\frac{h}{1-2\gamma h}P_{c\to z}G,
\]
and therefore
\[
y
=
\Exp_z\left(
-\frac{h}{1-2\gamma h}P_{c\to z}G
\right),
\]
which proves \eqref{eq:endpoint-step-representation}.
\end{proof}

\begin{remark}
For $K_{\max}>0$ and $0<\theta\leq1$, the condition
$1-2(1-\delta_D)\theta>0$ is equivalent to
\eqref{eq:diameter-condition}.
\end{remark}

\subsubsection{Main results}
\label{subsec:accelerated-algorithms}

\begin{algorithm}[t]
\caption{Accelerated c-wDRG for geodesically convex functions}
\label{alg:accelerated-cwdrg-convex}
\begin{algorithmic}
\Require Initial point $x_0$, certificate
$(\alpha,\beta,\gamma,\theta,\xi,T)$, with $\beta+\gamma=0$, number of iterations $K$
\State \textbf{Initialize}
$\bar v_0=0\in T_{x_0}\mathcal M$
\State Set
\[
L_{\mathrm R}
=
\frac{
2\left(
\alpha+\gamma
\right)
}{
1-2(1-\delta_D)\theta
},
\quad
h>0,
\quad
hL_{\mathrm R}\leq1,
\quad
s=\frac{h}{1-2\gamma h},
\quad
\lambda_k=\frac{k+2 \xi+T}{2}, \quad \tau_k=\frac{\xi}{\lambda_k+\xi-1} .
\]
\For{$k=0$ \textbf{to} $K-1$}
    \State
    $\displaystyle
    z_k
    =
    \Exp_{x_k}\left(
     \tau_k  \bar v_k
    \right)$

    \State Compute $x_{k+1}\in\mathcal X$ by solving
\Statex
$\displaystyle\qquad
\Delta_{c_k}(x_{k+1},z_k)
=
-s\,
\mathcal G_f(x_{k+1},z_k;c_k)$
\Statex
$\displaystyle\qquad
\text{where }
c_k
:=
\gamma_{z_k,x_{k+1}}(\theta)$
\State Set
$\displaystyle
G_k
:=
\mathcal G_f(x_{k+1},z_k;c_k)$
    \State
    $\displaystyle
    \bar v_{k+1}
    =
    P_{z_k\to x_{k+1}}
    \left(
    P_{x_k\to z_k}
    \left(
    \bar v_k-\Log_{x_k}z_k
    \right)
    -
    \frac{\lambda_ks}{\xi}
    P_{c_k\to z_k}G_k
    -
    \Log_{z_k}x_{k+1}
    \right)$
\EndFor
\Ensure $x_K$
\end{algorithmic}
\end{algorithm}


\begin{algorithm}[t]
\caption{Accelerated c-wDRG for geodesically strongly convex functions}
\label{alg:accelerated-cwdrg-strong}
\begin{algorithmic}
\Require Initial point $x_0$, certificate
$(\alpha,\beta,\gamma,\theta,\xi)$ with $\mu_{\mathrm R}>0$,
number of iterations $K$

\State \textbf{Initialize}
$\bar v_0=0\in T_{x_0}\mathcal M$

\State 
Define
\[
t_{\star}(\xi):=\frac{2 \xi\left(\delta_D+2+3 \xi-4 \zeta_D\right)}{\delta_D+\xi+2 \xi\left(\xi-\zeta_D\right)+\sqrt{\left[\delta_D+\xi+2 \xi\left(\xi-\zeta_D\right)\right]^2+8 \xi^2\left(\xi-\zeta_D\right)\left(\delta_D+2+3 \xi-4 \zeta_D\right)}} .
\]
Set
\[
L_{\mathrm R}
:=
\frac{
2\left[
\alpha+\gamma
-(1-\delta_D)(\beta+\gamma)\theta^2
\right]
}{
1-2(1-\delta_D)\theta
},
\qquad
\mu_{\mathrm R}
:=
2(\beta+\gamma).
\]
Choose
\[
0<h
\leq
\min\left\{
\frac{1}{L_{\mathrm R}},
\frac{t^2_\star(\xi)}{\xi\mu_{\mathrm R}}
\right\}, \quad 1-2\gamma h>0, \quad h<\frac{1}{\xi \mu_R}
\] 
and set
\[
\vartheta
:=
\sqrt{\xi\mu_{\mathrm R}h},
\qquad
s
:=
\frac{h}{1-2\gamma h}.
\]

\For{$k=0$ \textbf{to} $K-1$}

    \State
    $\displaystyle
    z_k
    =
    \Exp_{x_k}\left(
    \frac{\vartheta}{1+\vartheta}
    \bar v_k
    \right)$

    \State Compute $x_{k+1}\in\mathcal X$ by solving
    \Statex
    $\displaystyle\qquad
    \Delta_{c_k}(x_{k+1},z_k)
    =
    -s\,
    \mathcal G_f(x_{k+1},z_k;c_k)$
    \Statex
    $\displaystyle\qquad
    \text{with }
    c_k:=\gamma_{z_k,x_{k+1}}(\theta)$

    \State Set
    $\displaystyle
    G_k:=\mathcal G_f(x_{k+1},z_k;c_k)$

   \State
$\displaystyle
\bar v_{k+1}
=
P_{z_k\to x_{k+1}}
\left(
\left[1-\frac{\vartheta}{\xi}\right]
P_{x_k\to z_k}\left(\bar v_k-\Log_{x_k}z_k\right)
-
\frac{s}{\vartheta}P_{c_k\to z_k}G_k
-
\Log_{z_k}x_{k+1}
\right)$
\EndFor

\Ensure $x_K$
\end{algorithmic}
\end{algorithm}


\begin{theorem}[Accelerated convergence from c-wDRG:
geodesically convex case]
\label{thm:cwdrg-accelerated-convex}
Suppose that the assumptions of
Lemma~\ref{thm:endpoint-reduction}
hold uniformly along the sequence generated by
Algorithm~\ref{alg:accelerated-cwdrg-convex}.
Let $\xi\geq\zeta_D$, and choose $T>0$ such that, for all $k\geq0$,
\begin{equation}
\label{eq:convex-distortion-condition}
\frac{\xi-\delta_D}{2}
\left(\frac{1}{1-r_k}-1\right)
\leq
(\xi-\zeta_D)
\left(\frac{1}{(1-\tau_k)^2}-1\right).
\end{equation}
where
\[
\lambda_k:=\frac{k+2\xi+T}{2},
\qquad
\tau_k:=\frac{\xi}{\lambda_k+\xi-1},
\qquad
r_k:=\frac{\xi}{\lambda_k}.
\]

Define
\begin{equation}
\label{eq:accelerated-convex-potential}
\Phi_k^{\mathrm C}
:=
h\lambda^2_{k-1}
\bigl(f(x_k)-f(x^\star)\bigr)
+
\frac{\xi}{2}
\left(
\left\|\bar v_k-\Log_{x_k}x^\star\right\|_{x_k}^2
+
(\xi-1)\|\bar v_k\|_{x_k}^2
\right).
\end{equation}
Then
\[
\Phi_{k+1}^{\mathrm C}
\leq
\Phi_k^{\mathrm C},
\]
and, for every $k\geq0$,
\begin{equation}
\label{eq:accelerated-convex-rate}
f(x_k)-f(x^\star)
\leq
\frac{
(2\xi+T-1)^2
\bigl(f(x_0)-f(x^\star)\bigr)
+
\dfrac{2\xi}{h}d^2(x_0,x^\star)
}{
(k-1+2\xi+T)^2
}.
\end{equation}
\end{theorem}

\begin{proof}
For $p\in\mathcal X$ and $v\in T_p\mathcal M$, define
\[
\mathcal E_p(v)
:=
\left\|v-\Log_p x^\star\right\|_p^2
+
(\xi-1)\|v\|_p^2.
\]
Let
\[
h>0,
\qquad
hL_{\mathrm R}\leq1,
\qquad
s:=\frac{h}{1-2\gamma h}.
\]
and define
\[
G_k
:=
\mathcal G_f(x_{k+1},z_k;c_k),
\qquad
g_k
:=
\frac{1}{1-2\gamma h}
P_{c_k\to z_k}G_k.
\]
By Lemma~\ref{thm:endpoint-reduction},
\begin{equation}
\label{eq:convex-endpoint-step}
x_{k+1}
=
\Exp_{z_k}(-hg_k),
\end{equation}
and, since $\mu_{\mathrm R}=0$ in the geodesically convex case,
\begin{equation}
\label{eq:convex-endpoint-oracle}
f(x_{k+1})-f(x)
\leq
-\left\langle
g_k,\Log_{z_k}x
\right\rangle_{z_k}
-\frac{h}{2}\|g_k\|_{z_k}^2
\qquad
\text{for all }x\in\mathcal X.
\end{equation}

Define
\[
v_k
:=
P_{x_k\to z_k}
\left(
\bar v_k-\Log_{x_k}z_k
\right)
\in T_{z_k}\mathcal M,
\]
and
\[
\widetilde v_{k+1}
:=
v_k
-
\frac{h\lambda_k}{\xi}g_k.
\]
Since
\[
\frac{h}{1-2\gamma h}=s,
\]
this is precisely
\[
\widetilde v_{k+1}
=
v_k
-
\frac{\lambda_ks}{\xi}
P_{c_k\to z_k}G_k,
\]
as in Algorithm~\ref{alg:accelerated-cwdrg-convex}. Moreover,
\[
\bar v_{k+1}
=
P_{z_k\to x_{k+1}}
\left(
\widetilde v_{k+1}
-
\Log_{z_k}x_{k+1}
\right).
\]

Since
\[
z_k=\Exp_{x_k}(\tau_k\bar v_k),
\]
we have
\[
v_k
=
P_{x_k\to z_k}
\bigl((1-\tau_k)\bar v_k\bigr),
\]
and therefore
\[
\Log_{z_k}x_k
=
-\frac{\tau_k}{1-\tau_k}v_k
=
-\frac{\xi}{\lambda_k-1}v_k.
\]
Applying \eqref{eq:convex-endpoint-oracle} with
$x=x^\star$ and $x=x_k$, multiplying the two inequalities by
$\lambda_k$ and $\lambda_k^2-\lambda_k$, respectively, and adding
them gives
\[
\lambda_k^2\bigl(f(x_{k+1})-f(x^\star)\bigr)
-(\lambda_k^2-\lambda_k)\bigl(f(x_k)-f(x^\star)\bigr)
\leq
\lambda_k\left\langle g_k,\xi v_k-\Log_{z_k}x^\star\right\rangle_{z_k}
-\frac{h\lambda_k^2}{2}\|g_k\|_{z_k}^2.
\]
On the other hand, the definition of $\widetilde v_{k+1}$ yields
\[
\frac{\xi}{2}
\left[
\mathcal E_{z_k}(\widetilde v_{k+1})
-
\mathcal E_{z_k}(v_k)
\right]
=
-h\lambda_k
\left\langle
g_k,\xi v_k-\Log_{z_k}x^\star
\right\rangle_{z_k}
+
\frac{h^2\lambda_k^2}{2}\|g_k\|_{z_k}^2.
\]
Since
\[
\lambda_k^2-\lambda_k
\leq
\lambda_{k-1}^2,
\]
we obtain
\begin{equation}
\label{eq:convex-endpoint-space-estimate}
h\lambda_k^2
\bigl(f(x_{k+1})-f(x^\star)\bigr)
+
\frac{\xi}{2}
\mathcal E_{z_k}(\widetilde v_{k+1})
\leq
h\lambda_{k-1}^2
\bigl(f(x_k)-f(x^\star)\bigr)
+
\frac{\xi}{2}
\mathcal E_{z_k}(v_k).
\end{equation}
Furthermore,
\[
\|v_k\|_{z_k}
=
(1-\tau_k)\|\bar v_k\|_{x_k}.
\]
Applying Lemma~\ref{lem:kim-yang-distortion-I} with
$p_A=x_k$, $p_B=z_k$, $v_A=\bar v_k$, $v_B=v_k$, and $x=x^\star$
gives
\[
\left\|v_k-\Log_{z_k}x^\star\right\|_{z_k}^2
+
(\zeta_D-1)\|v_k\|_{z_k}^2
\leq
\left\|\bar v_k-\Log_{x_k}x^\star\right\|_{x_k}^2
+
(\zeta_D-1)\|\bar v_k\|_{x_k}^2.
\]
Since $\xi\geq\zeta_D$, it follows that
\begin{equation}
\label{eq:convex-first-distortion-bound}
\mathcal E_{z_k}(v_k)
\leq
\mathcal E_{x_k}(\bar v_k)
-
(\xi-\zeta_D)
\left(
\frac{1}{(1-\tau_k)^2}-1
\right)
\|v_k\|_{z_k}^2.
\end{equation}
Next, define
\[
b_k
:=
\widetilde v_{k+1}-v_k
=
-\frac{h\lambda_k}{\xi}g_k.
\]
By \eqref{eq:convex-endpoint-step},
\[
\Log_{z_k}x_{k+1}
=
-hg_k
=
\frac{\xi}{\lambda_k}b_k
=
r_kb_k.
\]
Hence Lemma~\ref{lem:kim-yang-distortion-II}, applied with
$p_A=z_k$, $p_B=x_{k+1}$, $a=v_k$, $b=b_k$,
$v_A=\widetilde v_{k+1}$, $v_B=\bar v_{k+1}$,
$r=r_k$, and $x=x^\star$, yields
\begin{equation}
\label{eq:convex-second-distortion-bound}
\mathcal E_{x_{k+1}}(\bar v_{k+1})
\leq
\mathcal E_{z_k}(\widetilde v_{k+1})
+
\frac{\xi-\delta_D}{2}
\left(
\frac{1}{1-r_k}-1
\right)
\|v_k\|_{z_k}^2.
\end{equation}
Combining
\eqref{eq:convex-endpoint-space-estimate},
\eqref{eq:convex-first-distortion-bound}, and
\eqref{eq:convex-second-distortion-bound}, we obtain
\[
\begin{aligned}
h\lambda_k^2
\bigl(f(x_{k+1})-f(x^\star)\bigr)
+
\frac{\xi}{2}\mathcal E_{x_{k+1}}(\bar v_{k+1})
\leq{}&
h\lambda_{k-1}^2
\bigl(f(x_k)-f(x^\star)\bigr)
+
\frac{\xi}{2}\mathcal E_{x_k}(\bar v_k)
\\
&+
\frac{\xi}{2}
\left[
\frac{\xi-\delta_D}{2}
\left(\frac{1}{1-r_k}-1\right)
-
(\xi-\zeta_D)
\left(\frac{1}{(1-\tau_k)^2}-1\right)
\right]
\|v_k\|_{z_k}^2.
\end{aligned}
\]
Since the parameters $\tau_k$ and $r_k$ satisfy
\[
\frac{\xi-\delta_D}{2}
\left(
\frac{1}{1-r_k}-1
\right)
\leq
(\xi-\zeta_D)
\left(
\frac{1}{(1-\tau_k)^2}-1
\right),
\]
so the last term is nonpositive. Hence
\[
h\lambda_k^2
\bigl(f(x_{k+1})-f(x^\star)\bigr)
+
\frac{\xi}{2}\mathcal E_{x_{k+1}}(\bar v_{k+1})
\leq
h\lambda_{k-1}^2
\bigl(f(x_k)-f(x^\star)\bigr)
+
\frac{\xi}{2}\mathcal E_{x_k}(\bar v_k).
\]
By \eqref{eq:accelerated-convex-potential}, this is precisely
\[
\Phi_{k+1}^{\mathrm C}
\leq
\Phi_k^{\mathrm C}.
\]

Since $\mathcal E_{x_k}(\bar v_k)\geq0$ and
$\Phi_k^{\mathrm C}\leq\Phi_0^{\mathrm C}$,
\[
f(x_k)-f(x^\star)
\leq
\frac{
4\Phi_0^{\mathrm C}
}{
h(k-1+2\xi+T)^2
}.
\]
Since $\bar v_0=0$,
\[
\Phi_0^{\mathrm C}
=
\frac{h(2\xi+T-1)^2}{4}
\bigl(f(x_0)-f(x^\star)\bigr)
+
\frac{\xi}{2}d^2(x_0,x^\star).
\]
Therefore,
\[
f(x_k)-f(x^\star)
\leq
\frac{
(2\xi+T-1)^2
\bigl(f(x_0)-f(x^\star)\bigr)
+
\dfrac{2\xi}{h}d^2(x_0,x^\star)
}{
(k-1+2\xi+T)^2
},
\]
which proves \eqref{eq:accelerated-convex-rate}.
\end{proof}

\begin{remark}
\label{rem:convex-parameter-choice}
Condition~\eqref{eq:convex-distortion-condition} holds
for every $k\geq0$ whenever
\[
\xi\geq2\zeta_D-\delta_D,
\qquad
T\geq2(\xi-1).
\]
In particular, one may take
$\xi=2\zeta_D-\delta_D$ and $T=2\xi$.
\end{remark}

\begin{theorem}[Accelerated convergence from c-wDRG:
geodesically strongly convex case]
\label{thm:cwdrg-accelerated-strong}
Let the iterates and auxiliary points be generated by
Algorithm~\ref{alg:accelerated-cwdrg-strong}.
Assume $L_R>0, \mu_{\mathrm R}>0$, 
\[
\xi \geq \zeta_D, \quad \delta_D+2+3 \xi-4 \zeta_D>0 .
\]
and define
\begin{equation}
\label{eq:strong-distortion-threshold}
t_{\star}(\xi):=\frac{2 \xi\left(\delta_D+2+3 \xi-4 \zeta_D\right)}{\delta_D+\xi+2 \xi\left(\xi-\zeta_D\right)+\sqrt{\left[\delta_D+\xi+2 \xi\left(\xi-\zeta_D\right)\right]^2+8 \xi^2\left(\xi-\zeta_D\right)\left(\delta_D+2+3 \xi-4 \zeta_D\right)}} .
\end{equation}
Choose $h>0$ satisfying
\begin{equation}
\label{eq:strong-admissible-step}
0<h
\leq
\min\left\{
\frac{1}{L_{\mathrm R}},
\frac{t^2_\star(\xi)}{\xi\mu_{\mathrm R}}
\right\},\quad 1-2\gamma h>0, \quad h<\frac{1}{\xi \mu_R}
\end{equation}
and define
\[
\vartheta
:=
\sqrt{\xi\mu_{\mathrm R}h}.
\]
Then
\begin{equation}
\label{eq:accelerated-strong-potential}
\Phi_k^{\mathrm{SC}}
:=
\left(
1-\frac{\vartheta}{\xi}
\right)^{-k}
\left[
f(x_k)-f(x^\star)
+
\frac{\mu_{\mathrm R}}{2}
\left(
\left\|\bar v_k-\Log_{x_k}x^\star\right\|_{x_k}^2
+
(\xi-1)\|\bar v_k\|_{x_k}^2
\right)
\right]
\end{equation}
satisfies
\[
\Phi_{k+1}^{\mathrm{SC}}
\leq
\Phi_k^{\mathrm{SC}}.
\]
Consequently, for every $k\geq0$,
\begin{equation}
\label{eq:accelerated-strong-rate}
f(x_k)-f(x^\star)
\leq
\left[
f(x_0)-f(x^\star)
+
\frac{\mu_{\mathrm R}}{2}d^2(x_0,x^\star)
\right]
\left(
1-\sqrt{\frac{\mu_{\mathrm R}h}{\xi}}
\right)^k.
\end{equation}
\end{theorem}

\begin{proof}
For $p\in\mathcal X$ and $v\in T_p\mathcal M$, define
\[
\mathcal E_p(v)
:=
\left\|v-\Log_p x^\star\right\|_p^2
+
(\xi-1)\|v\|_p^2.
\]
Set
\[
\vartheta
:=
\sqrt{\xi\mu_{\mathrm R}h},
\qquad
s
:=
\frac{h}{1-2\gamma h}.
\]
By the stepsize condition,
\[
\vartheta
=
\sqrt{\xi\mu_{\mathrm R}h}
\leq
t_{\star}(\xi).
\]
Moreover, since $0\leq\theta\leq1$, $t_{\star}(\xi)\leq1$, and $\xi\geq1$,
\[
h\theta\mu_{\mathrm R}
\leq
\frac{\theta t_{\star}(\xi)^2}{\xi}
\leq1,
\]
so the stepsize conditions of
Lemma~\ref{thm:endpoint-reduction} are satisfied.

Let
\[
G_k
:=
\mathcal G_f(x_{k+1},z_k;c_k),
\qquad
g_k
:=
\frac{1}{1-2\gamma h}
P_{c_k\to z_k}G_k.
\]
By Lemma~\ref{thm:endpoint-reduction},
\begin{equation}
\label{eq:strong-endpoint-step}
x_{k+1}
=
\Exp_{z_k}(-hg_k),
\end{equation}
and, for every $x\in\mathcal X$,
\begin{equation}
\label{eq:strong-endpoint-oracle}
f(x_{k+1})-f(x)
\leq
-\left\langle g_k,\Log_{z_k}x\right\rangle_{z_k}
-\frac{h}{2}\|g_k\|_{z_k}^2
-\frac{\mu_{\mathrm R}}{2}d^2(z_k,x).
\end{equation}

Define
\[
v_k
:=
P_{x_k\to z_k}
\left(
\bar v_k-\Log_{x_k}z_k
\right),
\]
and
\[
\widetilde v_{k+1}
:=
\left(1-\frac{\vartheta}{\xi}\right)v_k
-
\frac{\vartheta}{\xi\mu_{\mathrm R}}g_k.
\]
Since
\[
g_k
=
\frac{s}{h}P_{c_k\to z_k}G_k,
\qquad
\frac{\vartheta}{\xi\mu_{\mathrm R}h}
=
\frac{1}{\vartheta},
\]
we have
\[
\frac{\vartheta}{\xi\mu_{\mathrm R}}g_k
=
\frac{s}{\vartheta}P_{c_k\to z_k}G_k,
\]
so this is exactly the momentum update used in
Algorithm~\ref{alg:accelerated-cwdrg-strong}. Moreover,
\[
\bar v_{k+1}
=
P_{z_k\to x_{k+1}}
\left(
\widetilde v_{k+1}
-\Log_{z_k}x_{k+1}
\right).
\]

Since
\[
z_k
=
\Exp_{x_k}
\left(
\frac{\vartheta}{1+\vartheta}\bar v_k
\right),
\]
we have
\[
v_k
=
P_{x_k\to z_k}
\left(
\frac{1}{1+\vartheta}\bar v_k
\right),
\qquad
\Log_{z_k}x_k
=
-\vartheta v_k.
\]

Using \eqref{eq:strong-endpoint-oracle} together with the above
momentum identities, the strongly convex square-completion argument
gives
\begin{equation}
\label{eq:strong-endpoint-space-estimate}
\begin{aligned}
f(x_{k+1})-f(x^\star)
+
\frac{\mu_{\mathrm R}}{2}\mathcal E_{z_k}(\widetilde v_{k+1})
\leq{}&
\left(1-\frac{\vartheta}{\xi}\right)
\left[
f(x_k)-f(x^\star)
+
\frac{\mu_{\mathrm R}}{2}\mathcal E_{z_k}(v_k)
\right]
\\
&-
\frac{\mu_{\mathrm R}}{2}
\vartheta
\left(1-\frac{\vartheta}{\xi}\right)
\|v_k\|_{z_k}^2.
\end{aligned}
\end{equation}

For the extrapolation $x_k\to z_k$, applying
Lemma~\ref{lem:kim-yang-distortion-I} with
$p_A=x_k$, $p_B=z_k$, $v_A=\bar v_k$, $v_B=v_k$, and $x=x^\star$
gives
\[
\left\|v_k-\Log_{z_k}x^\star\right\|_{z_k}^2
+
(\zeta_D-1)\|v_k\|_{z_k}^2
\leq
\left\|\bar v_k-\Log_{x_k}x^\star\right\|_{x_k}^2
+
(\zeta_D-1)\|\bar v_k\|_{x_k}^2.
\]
Since
\[
\|v_k\|_{z_k}
=
\frac{1}{1+\vartheta}\|\bar v_k\|_{x_k},
\]
and $\xi\geq\zeta_D$, it follows that
\begin{equation}
\label{eq:strong-first-distortion-bound}
\mathcal E_{z_k}(v_k)
\leq
\mathcal E_{x_k}(\bar v_k)
-
(\xi-\zeta_D)
\left[
(1+\vartheta)^2-1
\right]
\|v_k\|_{z_k}^2.
\end{equation}
Furthermore,
\[
\widetilde v_{k+1}
=
\left(1-\frac{\vartheta}{\xi}\right)v_k
-
\frac{\vartheta}{\xi\mu_{\mathrm R}}g_k,
\]
and, since $\vartheta^2=\xi\mu_{\mathrm R}h$,
\[
\Log_{z_k}x_{k+1}
=
-hg_k
=
\vartheta
\left(
-\frac{\vartheta}{\xi\mu_{\mathrm R}}g_k
\right).
\]
Therefore, applying Lemma~\ref{lem:kim-yang-distortion-II} with
$p_A=z_k$, $p_B=x_{k+1}$,
$a=(1-\vartheta/\xi)v_k$,
$b=-\vartheta g_k/(\xi\mu_{\mathrm R})$,
$v_A=\widetilde v_{k+1}$,
$v_B=\bar v_{k+1}$,
$r=\vartheta$, and $x=x^\star$ yields
\begin{equation}
\label{eq:strong-second-distortion-bound}
\mathcal E_{x_{k+1}}(\bar v_{k+1})
\leq
\mathcal E_{z_k}(\widetilde v_{k+1})
+
\frac{\xi-\delta_D}{2}
\left(
\frac{1}{1-\vartheta}-1
\right)
\left(1-\frac{\vartheta}{\xi}\right)^2
\|v_k\|_{z_k}^2.
\end{equation}

Combining
\eqref{eq:strong-endpoint-space-estimate},
\eqref{eq:strong-first-distortion-bound}, and
\eqref{eq:strong-second-distortion-bound} gives
\[
\begin{aligned}
f(x_{k+1})-f(x^\star)
+\frac{\mu_{\mathrm R}}{2}\mathcal E_{x_{k+1}}(\bar v_{k+1})
\leq{}&
\left(1-\frac{\vartheta}{\xi}\right)
\left[
f(x_k)-f(x^\star)
+\frac{\mu_{\mathrm R}}{2}\mathcal E_{x_k}(\bar v_k)
\right]
\\
&+
\frac{\mu_{\mathrm R}}{2}
\Bigg[
\frac{\xi-\delta_D}{2}
\left(
\frac{1}{1-\vartheta}-1
\right)
\left(1-\frac{\vartheta}{\xi}\right)^2
-
\vartheta
\left(1-\frac{\vartheta}{\xi}\right)
\\
&\hspace{20mm}
-
(\xi-\zeta_D)
\left(1-\frac{\vartheta}{\xi}\right)
\left[
(1+\vartheta)^2-1
\right]
\Bigg]
\|v_k\|_{z_k}^2.
\end{aligned}
\]

Thus, it remains to verify the metric-distortion condition
\begin{equation}
\label{eq:strong-distortion-condition}
\frac{\xi-\delta_D}{2}
\left(
\frac{1}{1-\vartheta}-1
\right)
\left(1-\frac{\vartheta}{\xi}\right)^2
-
\vartheta
\left(1-\frac{\vartheta}{\xi}\right)
\leq
(\xi-\zeta_D)
\left(1-\frac{\vartheta}{\xi}\right)
\left[
(1+\vartheta)^2-1
\right].
\end{equation}
Since $0<\vartheta<1$ and $\xi\geq\zeta_D$, cancelling the
positive factor $1-\vartheta/\xi$ and rearranging terms shows that
\eqref{eq:strong-distortion-condition} is equivalent to
\[
2\xi(\xi-\zeta_D)\vartheta^2
+
\left[
\delta_D+\xi+2\xi(\xi-\zeta_D)
\right]\vartheta
-
\xi
\left(
\delta_D+2+3\xi-4\zeta_D
\right)
\leq 0.
\]
The positive root of the corresponding quadratic equation is
$t_{\star}(\xi)$ defined in \eqref{eq:strong-distortion-threshold}.
Since
\[
\vartheta
=
\sqrt{\xi\mu_{\mathrm R}h}
\leq
t_{\star}(\xi),
\]
the distortion condition holds.

Consequently,
\[
f(x_{k+1})-f(x^\star)
+
\frac{\mu_{\mathrm R}}{2}\mathcal E_{x_{k+1}}(\bar v_{k+1})
\leq
\left(
1-\frac{\vartheta}{\xi}
\right)
\left[
f(x_k)-f(x^\star)
+
\frac{\mu_{\mathrm R}}{2}\mathcal E_{x_k}(\bar v_k)
\right].
\]
By the definition of $\Phi_k^{\mathrm{SC}}$,
\[
\Phi_{k+1}^{\mathrm{SC}}
\leq
\Phi_k^{\mathrm{SC}}.
\]

Iterating the preceding inequality and using $\bar v_0=0$ gives
\[
f(x_k)-f(x^\star)
\leq
\left[
f(x_0)-f(x^\star)
+
\frac{\mu_{\mathrm R}}{2}d^2(x_0,x^\star)
\right]
\left(
1-\frac{\vartheta}{\xi}
\right)^k.
\]
Finally,
\[
\frac{\vartheta}{\xi}
=
\sqrt{\frac{\mu_{\mathrm R}h}{\xi}},
\]
which proves \eqref{eq:accelerated-strong-rate}.

\end{proof}

\begin{remark}
\label{rem:accelerated-strong-optimal-step}
For the accelerated c-wDRG schemes induced by the methods in
Theorem~\ref{thm:structure-preserving-certificates}, consider first
the case $\xi>1$. Since $0<t_\star(\xi)<1$, the strict condition
$h<1/(\xi\mu_{\mathrm R})$ is automatically satisfied whenever
\[
0<h\leq
\min\left\{
\frac{1}{L_{\mathrm R}},
\frac{t_\star(\xi)^2}{\xi\mu_{\mathrm R}}
\right\}.
\]

Moreover, for all methods in
Theorem~\ref{thm:structure-preserving-certificates},
every stepsize satisfying the bound above automatically satisfies$1-2\gamma h>0.$

Hence the largest admissible stepsize is
\[
h_\star(\xi)
=
\min\left\{
\frac{1}{L_{\mathrm R}},
\frac{t_\star(\xi)^2}{\xi\mu_{\mathrm R}}
\right\},
\qquad
s_\star(\xi)
=
\frac{h_\star(\xi)}
{1-2\gamma h_\star(\xi)},
\]
and the corresponding certified contraction factor is
\[
q_\star(\xi)
=
1-\sqrt{
\frac{\mu_{\mathrm R}h_\star(\xi)}{\xi}
}.
\]

Since all methods in
Theorem~\ref{thm:structure-preserving-certificates}
satisfy $\beta\geq0$, we have
$
2\gamma
\leq
\mu_{\mathrm R}
=
2(\beta+\gamma).
$
Hence the condition
$
h<\frac{1}{\mu_{\mathrm R}}
$
 implies
$
1-2\gamma h>0.
$

We next consider the boundary case $\xi=1$.
Since $\xi\geq\zeta_D\geq1$, one necessarily has
\[
\zeta_D=1,
\qquad
t_\star(1)=1.
\]
The admissible stepsizes therefore reduce to
\[
0<h\leq\frac{1}{L_{\mathrm R}},
\qquad
h<\frac{1}{\mu_{\mathrm R}}.
\]
For the methods in
Theorem~\ref{thm:structure-preserving-certificates},
one has $L_{\mathrm R}\geq\mu_{\mathrm R}$.
We distinguish the following two cases.

\medskip
\noindent
\textit{(i) $L_{\mathrm R}>\mu_{\mathrm R}$.}
Then
\[
\frac{1}{L_{\mathrm R}}
<
\frac{1}{\mu_{\mathrm R}},
\]
and hence the largest admissible stepsize is attained:
\[
h_\star(1)
=
\frac{1}{L_{\mathrm R}},
\qquad
s_\star(1)
=
\frac{1}{L_{\mathrm R}-2\gamma}.
\]
The corresponding optimal certified contraction factor is
\[
q_\star(1)
=
1-
\sqrt{
\frac{\mu_{\mathrm R}}{L_{\mathrm R}}
}.
\]

\medskip
\noindent
\textit{(ii) $L_{\mathrm R}=\mu_{\mathrm R}$.}
In this case,
\[
0<h<\frac{1}{\mu_{\mathrm R}},
\]
so the endpoint is not attained. Instead,
\begin{equation}
\label{eq:strong-boundary-optimal-rate}
\sup\left\{
h>0:\ h\text{ is admissible}
\right\}
=
\frac{1}{\mu_{\mathrm R}},
\qquad
\inf_{0<h<1/\mu_{\mathrm R}}
\left(
1-\sqrt{\mu_{\mathrm R}h}
\right)
=
0.
\end{equation}
Moreover,
\[
s(h)
=
\frac{h}{1-2\gamma h},
\]
and therefore
\[
\sup_{0<h<1/\mu_{\mathrm R}} s(h)
=
\begin{cases}
\displaystyle
\frac{1}{\mu_{\mathrm R}-2\gamma}
=
\frac{1}{2\beta},
& \beta>0,
\\[3mm]
+\infty,
& \beta=0.
\end{cases}
\]
Thus, in this boundary case, the optimal contraction factor is
understood in the infimal sense.

Combining the preceding cases, the optimal certified contraction
factor can be written uniformly as
\begin{equation}
\label{eq:strong-optimal-rate}
q_\star(\xi)
=
1-
\min\left\{
\sqrt{
\frac{\mu_{\mathrm R}}
{\xi L_{\mathrm R}}
},
\frac{t_\star(\xi)}{\xi}
\right\},
\end{equation}
where, in the boundary case $L_{\mathrm R}=\mu_{\mathrm R}$,
$q_\star(\xi)$ is understood in the infimal sense, as
\eqref{eq:strong-boundary-optimal-rate} shows that
$h=1/\mu_{\mathrm R}$ is the supremum, but not an admissible
stepsize.

In particular, in the Euclidean case
\[
\delta_D=\zeta_D=1,
\]
one may take $\xi=1$, so that $t_\star(1)=1$ and
\[
L_{\mathrm R}
=
2(\alpha+\gamma),
\qquad
\mu_{\mathrm R}
=
2(\beta+\gamma).
\]
Hence \eqref{eq:strong-optimal-rate} reduces to
\[
q_\star^{\mathrm E}
=
1-
\sqrt{
\frac{\mu_{\mathrm R}}{L_{\mathrm R}}
}
=
1-
\sqrt{
\frac{\beta+\gamma}{\alpha+\gamma}
},
\]
with the equality case $L_{\mathrm R}=\mu_{\mathrm R}$ understood
in the infimal sense. This coincides with the accelerated strongly
convex convergence factor in the Euclidean weak discrete-gradient
framework of~\cite{Ushiyama2023}.

For example, for the Euclidean proximal-point method,
\[
L_{\mathrm R}=\mu_{\mathrm R}=\mu,
\qquad
\beta=0,
\qquad
\gamma=\frac{\mu}{2},
\]
and therefore
\[
\sup h=\frac{1}{\mu},
\qquad
\sup s=+\infty,
\qquad
\inf q=0.
\]
\end{remark}

The explicit method-specific expressions are summarized in
Table~\ref{tab:optimal-accelerated-strongly-convex-rates}
of Appendix~\ref{app:optimal-strongly-convex-rates}.

\begin{remark}
\label{rem:comparison-rapg}
For the Riemannian proximal-gradient approximation, our optimal
certified contraction factor in the geodesically strongly convex
setting is
\[
q_{\star}^{\mathrm{PG}}(\xi)
=
1-
\min\left\{
\sqrt{
\frac{\mu-\rho_{f_2}}
{\xi(L-\rho_{f_2})}
},
\frac{t_\star(\xi)}{\xi}
\right\}.
\]

We compare this result with the Riemannian accelerated
proximal-gradient (RAPG) method of Feng et al.~\cite{feng2025riemannian}.
To distinguish their proximal parameter from the center parameter
$\theta$ used in our framework, we denote it by
$\Theta_{\mathrm H}$. Their strongly convex contraction factor is
\[
q_{\mathrm F}(\Theta_{\mathrm H},\xi)
=
1-
\sqrt{
\frac{\mu-\rho_{f_2}}
{\xi(\Theta_{\mathrm H}L-\rho_{f_2})}
}.
\]
We compare the two results according to the three parameter regimes
in Proposition~3.1 of Feng et al.

\medskip
\noindent\textbf{(i) $\zeta_D=\delta_D=\xi=1$.}
In this case, $t_\star(1)=1$, and taking the infimum over the
admissible proximal parameter in Feng et al. gives
\[
q_{\star}^{\mathrm{PG}}
=
q_{\mathrm F,\star}
=
1-
\sqrt{
\frac{\mu-\rho_{f_2}}
{L-\rho_{f_2}}
}.
\]
Thus, the two certified contraction factors coincide.

\medskip
\noindent\textbf{(ii) $\zeta_D=\delta_D=1$ and $\xi>1$.}
Feng et al. require
\[
\Theta_{\mathrm H}>\max\{1,\Theta_1\},
\qquad
\Theta_1
=
\frac{(4\xi-1)^2(\mu-\rho_{f_2})}
{9\xi L}
+
\frac{\rho_{f_2}}{L}.
\]
Since $q_{\mathrm F}(\Theta_{\mathrm H},\xi)$ is increasing in
$\Theta_{\mathrm H}$,
\[
q_{\mathrm F,\star}(\xi)
:=
\inf_{\Theta_{\mathrm H}>\max\{1,\Theta_1\}}
q_{\mathrm F}(\Theta_{\mathrm H},\xi)
=
1-
\min\left\{
\sqrt{
\frac{\mu-\rho_{f_2}}
{\xi(L-\rho_{f_2})}
},
\frac{3}{4\xi-1}
\right\}.
\]

To compare the second branches in cases~(ii) and~(iii), define
\[
p_\xi(t)
:=
2\xi(\xi-\zeta_D)t^2
+
\bigl[
\delta_D+\xi+2\xi(\xi-\zeta_D)
\bigr]t
-
\xi\bigl(
\delta_D+2+3\xi-4\zeta_D
\bigr).
\]
By the definition of $t_\star(\xi)$ in
Theorem~\ref{thm:cwdrg-accelerated-strong},
$t_\star(\xi)$ is the unique positive root of
\[
p_\xi(t)=0.
\]

For case~(ii), define
\[
t_{\mathrm F}^{(ii)}
:=
\frac{3\xi}{4\xi-1}.
\]
When $\delta_D=\zeta_D=1$, direct substitution gives
\[
p_\xi\!\left(t_{\mathrm F}^{(ii)}\right)
=
-
\frac{
2\xi(\xi-1)(3\xi^2+\xi-1)
}{
(4\xi-1)^2
}
<0.
\]
Since $t_{\mathrm F}^{(ii)}>0$ and $t_\star(\xi)$ is the unique
positive root of $p_\xi$, it follows that
\[
t_{\mathrm F}^{(ii)}<t_\star(\xi),
\]
or equivalently,
\[
\frac{3}{4\xi-1}
<
\frac{t_\star(\xi)}{\xi}.
\]
Therefore,
\[
q_{\star}^{\mathrm{PG}}(\xi)
\leq
q_{\mathrm F,\star}(\xi).
\]

\medskip
\noindent\textbf{(iii) $\zeta_D>\delta_D$.}
Let
\[
\Delta_D:=\zeta_D-\delta_D>0.
\]
Our result requires
\[
\xi\geq\zeta_D,
\qquad
\delta_D+2+3\xi-4\zeta_D>0,
\]
whereas Proposition~3.1(iii) of Feng et al. requires, for some
$\lambda\in(1,4)$,
\[
\xi
\geq
\zeta_D+
\frac{\Delta_D}{\lambda-1}.
\]
The latter condition implies ours. Indeed,
\[
\xi-\zeta_D
\geq
\frac{\Delta_D}{\lambda-1},
\]
and hence
\[
\begin{aligned}
\delta_D+2+3\xi-4\zeta_D
&=
2+3(\xi-\zeta_D)-\Delta_D
\\
&\geq
2+
\left(
\frac{3}{\lambda-1}-1
\right)\Delta_D
>0.
\end{aligned}
\]
Thus, the sufficient condition on $\xi$ in our analysis is weaker
than that in Proposition~3.1(iii) of Feng et al.

In this case, Feng et al. further require
\[
\Theta_{\mathrm H}>\max\{1,\Theta_2\},
\]
where
\[
\Theta_2
=
\frac{
\left(\frac{4\xi}{\lambda}-1\right)^2
(\mu-\rho_{f_2})
}{
\left(\frac{4}{\lambda}-1\right)^2L\xi
}
+
\frac{\rho_{f_2}}{L}.
\]
Therefore,
\[
q_{\mathrm F,\star}(\xi,\lambda)
:=
\inf_{\Theta_{\mathrm H}>\max\{1,\Theta_2\}}
q_{\mathrm F}(\Theta_{\mathrm H},\xi)
=
1-
\min\left\{
\sqrt{
\frac{\mu-\rho_{f_2}}
{\xi(L-\rho_{f_2})}
},
\frac{4-\lambda}{4\xi-\lambda}
\right\}.
\]

To compare the second branches, define
\[
t_{\mathrm F}^{(iii)}
:=
\frac{\xi(4-\lambda)}{4\xi-\lambda}.
\]
Direct substitution gives
\[
\begin{aligned}
p_\xi\!\left(t_{\mathrm F}^{(iii)}\right)
=
\frac{2\xi(\xi-1)}{(4\xi-\lambda)^2}
\Big[
&-2(4\xi-\lambda)
\Big(
(\lambda-1)(\xi-\zeta_D)-\Delta_D
\Big)
\\
&+\lambda
\Big(
(\lambda-4)\xi(\xi-\zeta_D)
+\lambda-4\xi
\Big)
\Big].
\end{aligned}
\]
The parameter condition of Feng et al. implies
\[
(\lambda-1)(\xi-\zeta_D)-\Delta_D\geq0,
\qquad
\xi-\zeta_D>0,
\qquad
\xi>1.
\]
Together with $1<\lambda<4$, this shows that the expression in
brackets is strictly negative, and hence
\[
p_\xi\!\left(t_{\mathrm F}^{(iii)}\right)<0.
\]
Since $t_{\mathrm F}^{(iii)}>0$ and $t_\star(\xi)$ is the unique
positive root of $p_\xi$, we obtain
\[
t_{\mathrm F}^{(iii)}<t_\star(\xi),
\]
or equivalently,
\[
\frac{4-\lambda}{4\xi-\lambda}
<
\frac{t_\star(\xi)}{\xi}.
\]
Therefore,
\[
q_{\star}^{\mathrm{PG}}(\xi)
\leq
q_{\mathrm F,\star}(\xi,\lambda).
\]

Overall, the two certified contraction factors coincide in
case~(i). In cases~(ii) and~(iii), they share the same
condition-number-dependent branch, whereas our second branch is
strictly larger:
\[
\frac{t_\star(\xi)}{\xi}
>
\frac{3}{4\xi-1}
\quad\text{in case~(ii)},
\qquad
\frac{t_\star(\xi)}{\xi}
>
\frac{4-\lambda}{4\xi-\lambda}
\quad\text{in case~(iii)}.
\]
Hence, on the common admissible parameter range, our optimal
certified contraction factor is no larger than that obtained from
the sufficient conditions of Feng et al. Moreover, in case~(iii),
our sufficient condition on $\xi$ is weaker than that in
Proposition~3.1(iii) of Feng et al.

Table~\ref{tab:apg-rate-comparison} summarizes the accelerated
outer convergence rates of RAPG and the proximal-gradient
instances of c-wDRG for the objective gap $f(x_k)-f(x^\star)$,
where $f=f_1+f_2$ and $k$ counts outer iterations only.
Under their respective assumptions, both methods achieve an
$O(k^{-2})$ rate in the geodesically convex setting and accelerated
linear convergence in the strongly convex setting.
The strongly convex bounds require $\mu>\rho_{f_2}$, while the
standard convex case corresponds to $\mu=\rho_{f_2}=0$.

For the convex and strongly convex proximal-gradient instances
of c-wDRG, the rates in Table~\ref{tab:apg-rate-comparison}
use the stepsizes
\[
h_{\mathrm C}=\frac{1}{L-\rho_{f_2}},
\qquad
h_{\mathrm{SC}}
=
\min\left\{
\frac{1}{L-\rho_{f_2}},
\frac{t_\star(\xi)^2}{\xi(\mu-\rho_{f_2})}
\right\},
\]
respectively.
For RAPG, the basic initialization requirement is
\[
A_0>
\frac{\xi(\xi-1)}
{1-\dfrac{\xi(\mu-\rho_{f_2})}
{\Theta_{\mathrm H}L-\rho_{f_2}}}.
\]
For geodesically convex problems, the respective parameter
conditions do not generally determine whether $2\sqrt{A_0}$
is larger or smaller than $2\xi+T-1$.

\par\medskip
\noindent
\begin{minipage}{\linewidth}
\normalfont
\makeatletter
\def\@captype{table}
\makeatother

\caption{Accelerated convergence rates of Riemannian
proximal-gradient methods.}
\label{tab:apg-rate-comparison}

\centering
\small
\setlength{\tabcolsep}{6pt}
\renewcommand{\arraystretch}{1.5}

\begin{adjustbox}{max width=\linewidth}
\begin{tabular}{@{}lcc@{}}
\toprule
Method
& Convex
& Strongly convex \\
\midrule

RAPG~\cite[Thm.~3.1]{feng2025riemannian}
& $\displaystyle
O\!\left(
\frac{1}{(k+2\sqrt{A_0})^2}
\right)$
& $\displaystyle
O\!\left(
\left[
1-
\sqrt{
\frac{\mu-\rho_{f_2}}
{\xi(\Theta_{\mathrm H}L-\rho_{f_2})}
}
\right]^k
\right)$ \\

\addlinespace[6pt]

c-wDRG (PG)
& $\displaystyle
O\!\left(
\frac{1}{(k+2\xi+T-1)^2}
\right)$
& $\displaystyle
O\!\left(
\left[
1-
\min\left\{
\sqrt{
\frac{\mu-\rho_{f_2}}
{\xi(L-\rho_{f_2})}
},
\frac{t_\star(\xi)}{\xi}
\right\}
\right]^k
\right)$ \\

\bottomrule
\end{tabular}
\end{adjustbox}

\end{minipage}
\par\medskip

\end{remark}

\begin{remark}
\label{rem:app-comparison}
Mart\'inez-Rubio and Pokutta~\cite[Thm.~4]{martinez2023accelerated}
analyze an accelerated inexact proximal-point method on Hadamard
manifolds, allowing nonsmooth objectives and inexact proximal steps.
Our PP results allow positive sectional curvature, but require
smooth objectives, exact implicit steps, and that all iterates and
auxiliary points remain in the prescribed region.

Table~\ref{tab:app-comparison} compares the assumptions and outer complexity bounds at the same proximal stepsize $s$.
The strongly convex Riemacon bound is obtained by restarting
its convex algorithm.

\par\medskip
\noindent
\begin{minipage}{\linewidth}
\normalfont

\makeatletter
\def\@captype{table}
\makeatother

\caption{Geometric assumptions, proximal parameter ranges,
and outer complexity of accelerated proximal-point methods.
$K$: sectional curvature;
G: no condition-number-dependent local initialization;
F: acceleration in outer proximal calls;
C: the algorithm guarantees an explicit bound on the region
containing its iterates and auxiliary points.
All domains have diameter at most $D<\infty$.}
\label{tab:app-comparison}

\centering
\small
\setlength{\tabcolsep}{4pt}
\renewcommand{\arraystretch}{1.25}

\begin{adjustbox}{max width=\linewidth}
\begin{tabular}{@{}llllccc@{}}
\toprule
Method & Manifold & $K$ & Domain size & G & F & C \\
\midrule

Riemacon~\cite[Thm.~4]{martinez2023accelerated}
& Hadamard
& $[K_{\min},0]$
& $D<\infty$
& $\checkmark$ & $\checkmark$ & $\checkmark$ \\

\addlinespace[3pt]

c-wDRG (PP)
& Riemannian
& $[K_{\min},K_{\max}]$
& $\begin{cases}
D\sqrt{K_{\max}}<\omega_{\mathrm{PP}},
    & K_{\max}\geq0,\\
D<\infty,
    & K_{\max}<0
\end{cases}$
& $\checkmark$ & $\checkmark$ & $\times$ \\

\bottomrule
\end{tabular}
\end{adjustbox}

\par\medskip

\begin{adjustbox}{max width=\linewidth}
\begin{tabular}{@{}lcccc@{}}
\toprule
& \multicolumn{2}{c}{Geodesically convex}
& \multicolumn{2}{c}{$\mu$-strongly geodesically convex} \\
\cmidrule(lr){2-3}
\cmidrule(l){4-5}
Method
& Range of $s$ & Outer complexity
& Range of $s$ & Outer complexity \\
\midrule

Riemacon
& $s>0$
& $\displaystyle
O\!\left(
\frac{\zeta_D R}{\sqrt{s\varepsilon}}
\right)$
& $s>0$
& $\displaystyle
O\!\left(
\left[1+\frac{\zeta_D}{\sqrt{\mu s}}\right]
\log\frac{\mu R^2}{\varepsilon}
\right)$ \\

\addlinespace[5pt]

c-wDRG (PP)
& $s>0$
& $\displaystyle
O\!\left(
\sqrt{\frac{C_s}{\varepsilon}}
\right)$
& $\displaystyle
0<\frac{\mu s}{1+\mu s}\leq a_{\mathrm{PP}}$
& $\displaystyle
O\!\left(
\sqrt{\frac{\xi(1+\mu s)}{\mu s}}\,
\log\frac{\mathcal E_0}{\varepsilon}
\right)$ \\

\bottomrule
\end{tabular}
\end{adjustbox}

\end{minipage}
\par\medskip

For the c-wDRG (PP) instances in
Table~\ref{tab:app-comparison}, the condition $\delta_D>1/2$
is expressed using the unique root
$\omega_{\mathrm{PP}}\in(0,\pi/2)$ of
$\omega_{\mathrm{PP}}\cot\omega_{\mathrm{PP}}=1/2$,
with $\omega_{\mathrm{PP}}\approx1.16556$.
The quantities used in the complexity and step-size bounds are
\[
\begin{gathered}
R:=d(x_0,x^\star),
\qquad
\Delta_0:=f(x_0)-f(x^\star),
\qquad
\mathcal E_0:=\Delta_0+\frac{\mu R^2}{2},
\\[6pt]
C_s:=(2\xi+T-1)^2\Delta_0+\frac{2\xi R^2}{s},
\qquad
a_{\mathrm{PP}}:=\min\left\{
\frac{2\delta_D-1}{\delta_D},
\frac{t_\star(\xi)^2}{\xi}
\right\}.
\end{gathered}
\]
All complexity bounds in Table~\ref{tab:app-comparison}
count outer proximal calls and exclude the cost of solving
the proximal subproblems.
\end{remark}

\section{Experiments}
\label{sec:experiments}
In this section, we evaluate the proposed c-wDRG methods through numerical experiments. All experiments were implemented in Python and run on an Apple M3 Pro processor.

\subsection{Accelerated c-wDRG for SPD Karcher Means}

We compare the strongly convex accelerated steepest-descent,
implicit-midpoint, and proximal-point instances of c-wDRG, denoted by
SC-ASD-c-wDRG, SC-AMP-c-wDRG, and SC-APP-c-wDRG, respectively, on
\[
\min_{X\in\mathbb S_{++}^{d}}
f_\kappa(X)=\frac{1}{2n}\sum_{i=1}^{n}d^2(X,A_i^{(\kappa)}).
\]
We use the affine-invariant metric and, following
Kim et al.~\cite{KimYang2022}, set $d=100$ and $n=50$.
This manifold is Hadamard with $-1/2\le\operatorname{Sec}\le0$,
giving $\mu=1$ and $\delta_D=1$.
We fix $\mathcal X=B(I,D/2)$, where $D\approx1.11763$ is chosen so that $\xi=\zeta_D=1.2$,
satisfying the admissibility conditions
$
\xi\ge\zeta_D,
\delta_D+2+3\xi-4\zeta_D=1.8>0.
$ Thus $t_\star(\xi)/\xi=9/11$.

On Hadamard manifolds, for $\kappa=L/\mu>1$, the
condition-number-dependent terms in the certified factors satisfy
$
\frac{1}{\sqrt{\xi}}
>
\frac{2}{\sqrt{\xi(\kappa+3)}}
>
\frac{1}{\sqrt{\xi\kappa}}.
$
Since each convergence rate is one minus the minimum of its corresponding
term and $t_\star(\xi)/\xi$, we obtain
$
q_{\mathrm{PP}}\le q_{\mathrm{MP}}\le q_{\mathrm{SD}}.
$
For the parameter values considered here, the certified step sizes are
$
h_{\mathrm{SD}}=\frac1\kappa,
h_{\mathrm{MP}}=\frac4{\kappa+3},
h_{\mathrm{PP}}=\frac{486}{605}.
$

We vary $\kappa\in\{10,20,40,60,80\}$ by generating 25 inverse pairs:
\[
\begin{gathered}
H_j^{(\kappa)}
=Q_j\operatorname{diag}(a_\kappa I_{50},-a_\kappa I_{50})Q_j^\top,
\\
A_{2j-1}^{(\kappa)}=\exp(H_j^{(\kappa)}),
\qquad
A_{2j}^{(\kappa)}=\exp(-H_j^{(\kappa)}),
\end{gathered}
\]
where $Q_j$ are random orthogonal matrices.
This symmetry ensures that $X^\star=I$ and
$f_\kappa(I)=d a_\kappa^2/2$.
We choose $a_\kappa$ so that
\[
r_\kappa=\frac D2+\sqrt d\,a_\kappa,
\qquad
L_\kappa=\frac{r_\kappa}{\sqrt2}
\coth\!\left(\frac{r_\kappa}{\sqrt2}\right)=\kappa.
\]
Since $\mu=1$, we have $\kappa=L_\kappa/\mu=L_\kappa$.

We set $d(X_0,I)=0.45<D/2$, and verify that all outer iterates and auxiliary centers remain in $\mathcal X$.
Each run stops at the first iterate satisfying
$
f_\kappa(X_k)-f_\kappa(X^\star)\le10^{-10}.
$ The reported iteration counts include only outer iterations,
excluding the inner iterations used to solve the implicit equations,
whereas the running times include the cost of these inner solves.

Figure~\ref{fig:karcher-kappa-comparison} shows that
as $\kappa$ increases from 10 to 80, SC-ASD-c-wDRG requires
35--109 iterations and SC-AMP-c-wDRG requires 17--53, whereas
SC-APP-c-wDRG requires only six or seven.
SC-APP-c-wDRG is fastest for $\kappa\ge20$ in these experiments.
Although SC-AMP-c-wDRG approximately halves the outer iteration
count relative to SC-ASD-c-wDRG, its inner solves result in
longer running times.

\begin{figure*}[t]
    \centering
    \begin{subfigure}[t]{0.48\textwidth}
        \centering
        \includegraphics[width=\linewidth]{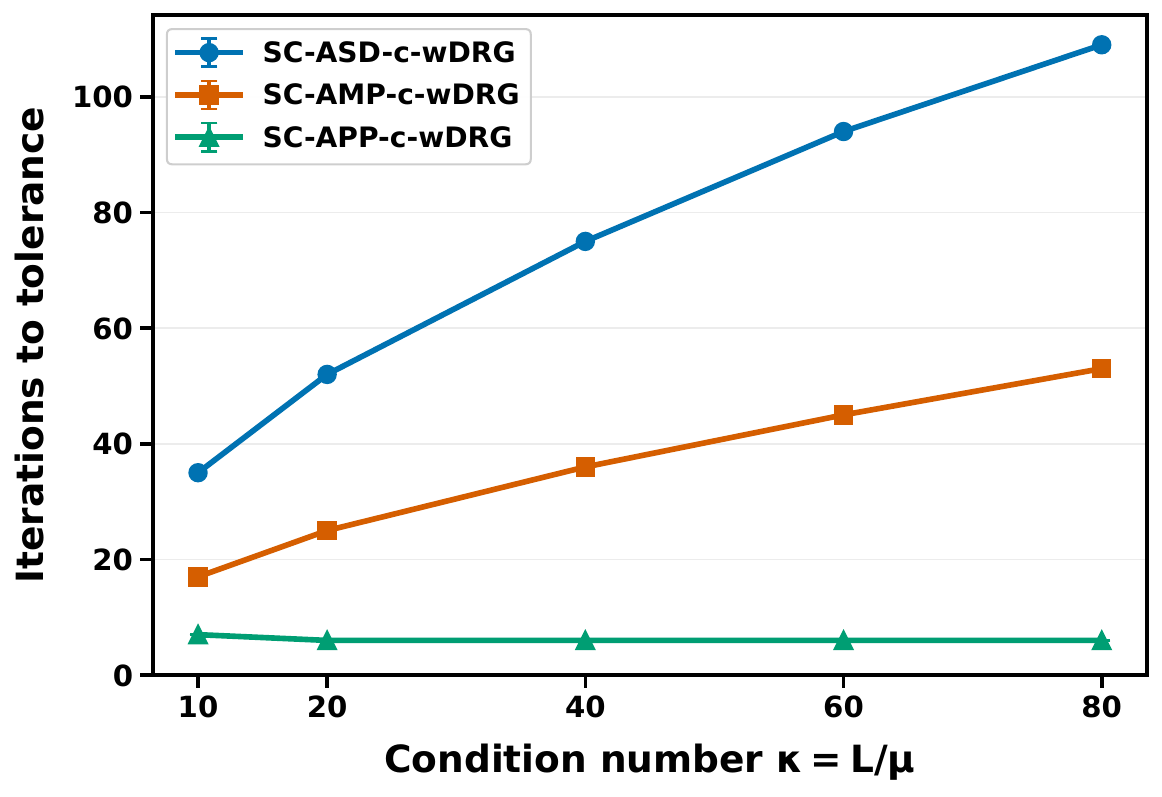}
        \caption{Iterations to reach the target accuracy.}
        \label{fig:karcher-kappa-iteration}
    \end{subfigure}
    \hfill
    \begin{subfigure}[t]{0.48\textwidth}
        \centering
        \includegraphics[width=\linewidth]{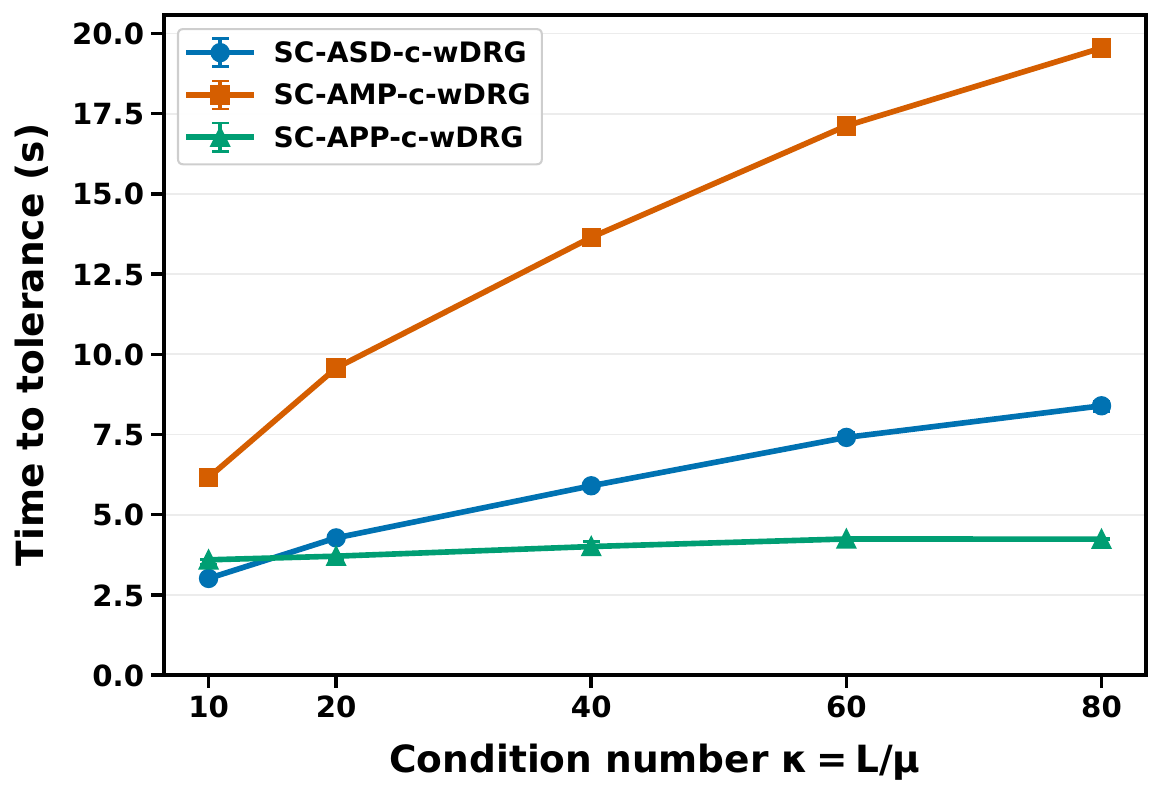}
        \caption{Running time to reach the target accuracy.}
        \label{fig:karcher-kappa-time}
    \end{subfigure}
    \caption{Iterations and running time required to achieve
    $f_\kappa(X_k)-f_\kappa(I)\le10^{-10}$ for varying
    certified condition numbers $\kappa=L_\kappa/\mu$.}
    \label{fig:karcher-kappa-comparison}
\end{figure*}

\subsection{Convergence Rate Comparison with RAPG}
In this subsection, we numerically validate the comparison with RAPG established in
Remark~\ref{rem:comparison-rapg}. Specifically, we examine two theoretical
advantages of SC-APG-c-wDRG: its admissible range of $\xi$ contains that
of RAPG, and, under the same admissible $\xi$, it yields a smaller certified
convergence rate. The first advantage is evaluated in
Fig.~\ref{fig:rapg-D-sweep}, where each method uses its respective best
certified parameter values, while the second is examined in
Fig.~\ref{fig:rapg-same-xi}, where both methods use the same value of $\xi$.

We compare SC-APG-c-wDRG with RAPG of Feng et al.~\cite{feng2025riemannian}
on the mixed-curvature product manifold
\[
\mathcal M
=
\mathbb S^2_{+1}
\times
\mathbb H^2_{-0.1}
\times
\mathbb R^{20},
\qquad
-0.1\le \operatorname{Sec}\le1.
\]
We consider
\[
f(x_S,x_H,u)
=
\frac12d_S^2(x_S,p_S)
+
\frac12d_H^2(x_H,p_H)
+
\frac12\|u\|^2
+
\lambda_{\ell_1}
\bigl(\|x_S-p_S\|_1+\|u\|_1\bigr),
\]
with $\lambda_{\ell_1}=10^{-2}$ and
$x^\star=(p_S,p_H,0)$.

To examine the effect of curvature distortion, we vary
\[
D\in
\{0.60,0.80,1.00,1.20,1.35,1.45,1.50,1.55\},
\]
using
\[
\mathcal X
=
B_S(p_S,r)
\times
B_H(p_H,r)
\times
B_{\mathbb R^{20}}(0,r),
\qquad
r=\frac{D}{2\sqrt3}.
\]
For the nonsmooth term, using
$\|D^2\Exp_z(\eta)\|_{\rm op}<3$ for $\|\eta\|\le D$ and the
$\sqrt3$-Lipschitz continuity of the ambient $\ell_1$ norm, we take
$
\rho_{f_2}
=
3\sqrt3\,\lambda_{\ell_1}
\approx0.05196.
$
Accordingly,
$
\mu_{\mathrm R}=\mu-\rho_{f_2},
L_{\mathrm R}=L-\rho_{f_2},
$
where
$
\mu=r\cot r,
L=\sqrt{0.1}\,r\coth(\sqrt{0.1}\,r).
$

The corresponding curvature constants are
$$
\delta_D=D\cot D,
\zeta_D=\sqrt{0.1}D\coth(\sqrt{0.1}D).
$$
As $D$ increases, $\zeta_D$ increases while $\delta_D$ decreases, so that
the distortion gap
$
\Delta_D:=\zeta_D-\delta_D
$
also increases substantially.

For SC-APG-c-wDRG, the admissible parameter $\xi$ satisfies
\[
\xi\ge\zeta_D,
\qquad
\delta_D+2+3\xi-4\zeta_D>0.
\]
Throughout the tested range of $D$, choosing $\xi=\zeta_D$ satisfies the
second condition since
$
\delta_D+2-\zeta_D>0.
$
Hence, the smallest admissible choice for our method remains
$
\xi_{\mathrm{ours}}=\zeta_D.
$

In contrast, RAPG requires
\[
\xi
\ge
\zeta_D+
\frac{\Delta_D}{\lambda_{\mathrm F}-1},
\qquad
\lambda_{\mathrm F}\in(1,4).
\]

Since $\Delta_D>0$, every admissible choice for RAPG satisfies
$
\xi_{\mathrm{RAPG}}>\zeta_D=\xi_{\mathrm{ours}}.
$
The condition-number-dependent terms satisfy
$
\sqrt{\frac{\mu_{\mathrm R}}
{\xi_{\mathrm{RAPG}}L_{\mathrm R}}}
<
\sqrt{\frac{\mu_{\mathrm R}}
{\xi_{\mathrm{ours}}L_{\mathrm R}}}.
$
Moreover, over the tested range, the curvature-dependent terms satisfy
\[
\frac{4-\lambda_{\mathrm F}}
{4\xi_{\mathrm{RAPG}}-\lambda_{\mathrm F}}
<
\frac{t_\star(\xi_{\mathrm{RAPG}})}
{\xi_{\mathrm{RAPG}}}
<
\frac{t_\star(\xi_{\mathrm{ours}})}
{\xi_{\mathrm{ours}}}.
\]
Thus, the smaller admissible $\xi$ and the tighter curvature-dependent
bound jointly improve the certified convergence rate of
SC-APG-c-wDRG.

For each $D$, both methods use their respective best certified parameter
choices.
Figure~\ref{fig:rapg-D-sweep} reports the number of iterations and running
time required to reach
$
f(x_k)-f(x^\star)\le10^{-12}.
$
As curvature distortion increases, the iteration count of RAPG grows from
$8$ to $25$, whereas SC-APG-c-wDRG requires only $2$--$5$ iterations.

\begin{figure*}[t]
    \centering
    \begin{subfigure}[t]{0.48\textwidth}
        \centering
        \includegraphics[width=\linewidth]
        {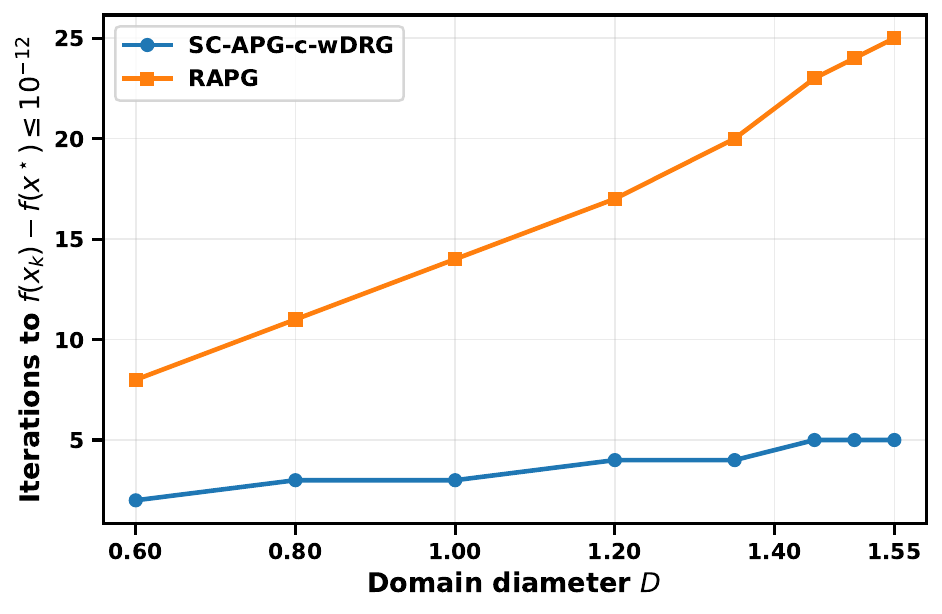}
        \caption{Iterations to reach an objective gap below $10^{-12}$.}
        \label{fig:rapg-D-sweep-iteration}
    \end{subfigure}
    \hfill
    \begin{subfigure}[t]{0.48\textwidth}
        \centering
        \includegraphics[width=\linewidth]
        {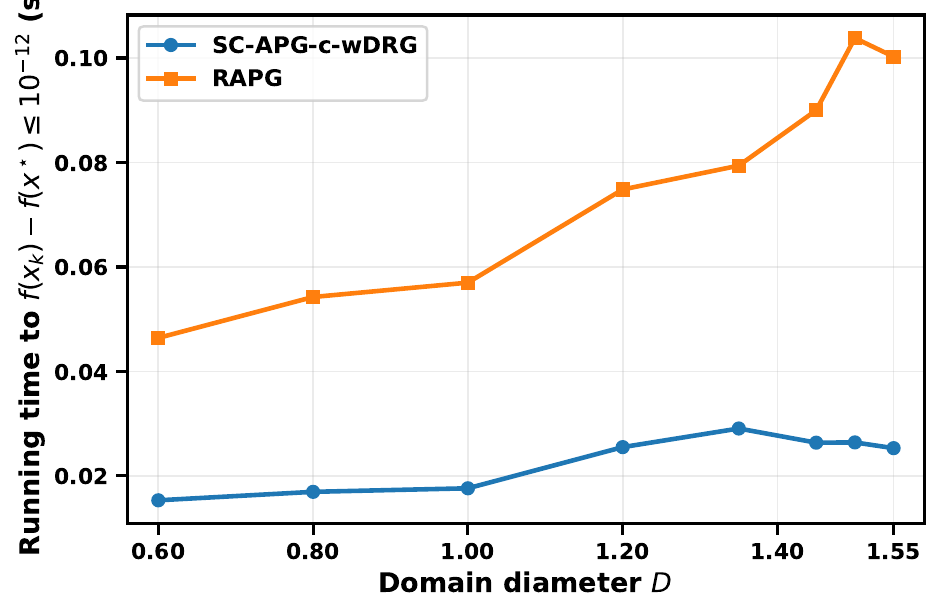}
        \caption{Running time to reach the same accuracy.}
        \label{fig:rapg-D-sweep-time}
    \end{subfigure}
    \caption{Comparison of SC-APG-c-wDRG and RAPG as $D$ increases, with each method using its respective best certified parameter values. Specifically, SC-APG-c-wDRG optimizes over $\xi$, whereas RAPG jointly optimizes over $\xi$ and $\lambda_{\mathrm F}$.}
    \label{fig:rapg-D-sweep}
\end{figure*}

For a direct comparison, we also set
$
\xi=\xi_{\mathrm{RAPG}}^\star
$
for both methods at each $D$.
As shown in Fig.~\ref{fig:rapg-same-xi}, the gap becomes smaller but
SC-APG-c-wDRG remains consistently faster, indicating that the improvement
is not solely due to allowing a smaller $\xi$.

\begin{figure*}[t]
    \centering
    \begin{subfigure}[t]{0.48\textwidth}
        \centering
        \includegraphics[width=\linewidth]
        {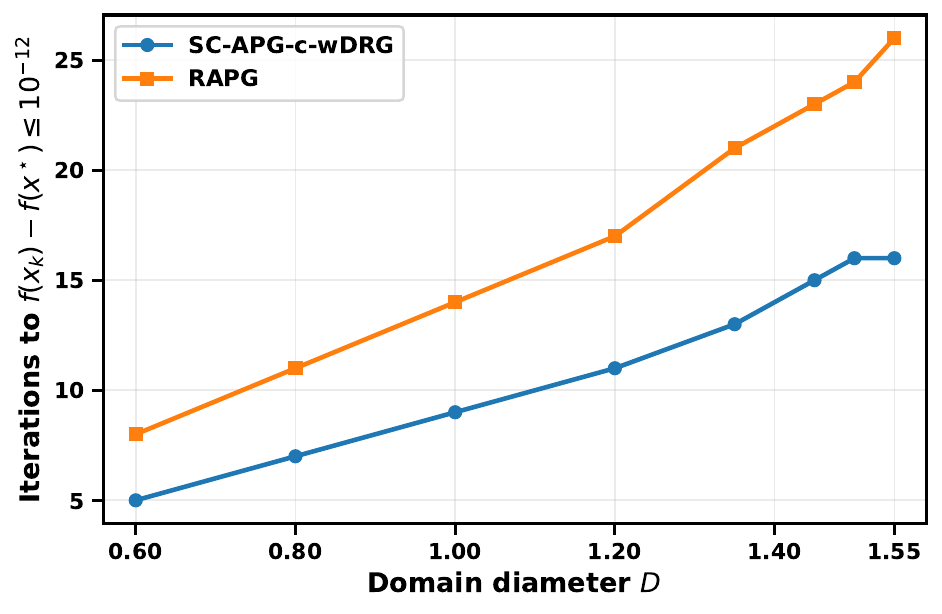}
        \caption{Iterations under a common value of $\xi$.}
        \label{fig:rapg-same-xi-iteration}
    \end{subfigure}
    \hfill
    \begin{subfigure}[t]{0.48\textwidth}
        \centering
        \includegraphics[width=\linewidth]
        {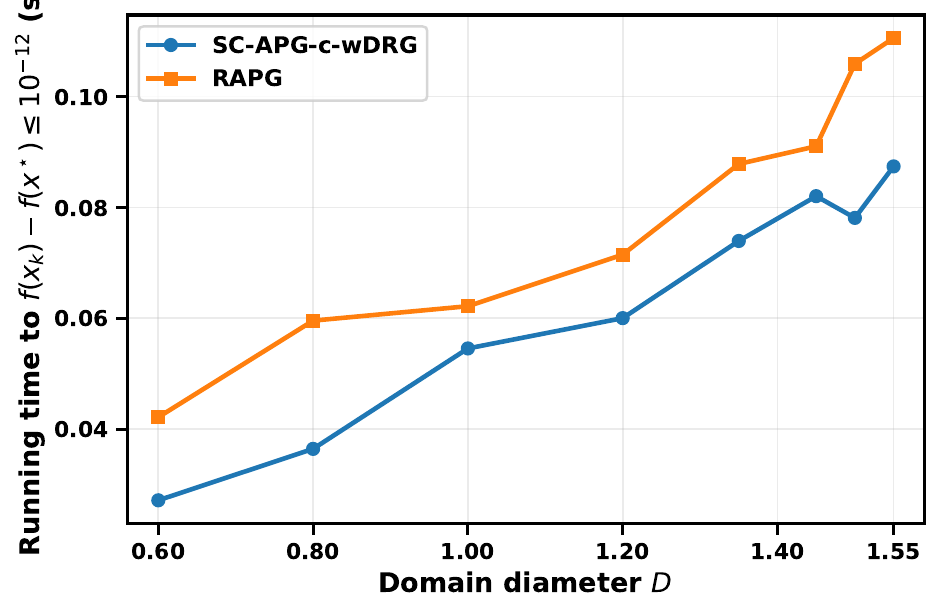}
        \caption{Running time under a common value of $\xi$.}
        \label{fig:rapg-same-xi-time}
    \end{subfigure}
    \caption{Comparison of SC-APG-c-wDRG and RAPG as $D$ increases under a common value of $\xi$.}
    \label{fig:rapg-same-xi}
\end{figure*}

\subsection{Acceleration of Riemannian Discrete-Gradient Methods}
\label{sec:dg-acceleration-experiment}

We compare the nonaccelerated geodesic AVF, Gonzalez, and Itoh--Abe
instances of c-wDRG, denoted by AVF-c-wDRG, G-c-wDRG, and IA-c-wDRG,
with their strongly convex accelerated counterparts, denoted by
SC-AAVF-c-wDRG, SC-AG-c-wDRG, and SC-AIA-c-wDRG, respectively.
On the hyperbolic plane
\[
\mathbb H^2
=\{x\in\mathbb R^3:-x_0^2+x_1^2+x_2^2=-1,\ x_0>0\}
\]
with the Lorentz-induced Riemannian metric and sectional curvature $-1$,
we minimize
\[
f_b(x)=\frac12(x_1^2+b x_2^2),
\qquad x^\star=(1,0,0),\qquad f_b(x^\star)=0.
\]
On $\mathcal X=B(x^\star,0.4)$, the objective is
$\mu$-strongly geodesically convex and $L$-smooth with
$\mu=1$ and $L=b\cosh(0.8)$.
We set $b=\kappa/\cosh(0.8)$ and consider certified condition numbers
$\kappa=L/\mu\in\{10,20\}$.
With $D=0.8$, we use $\delta_D=1$, $\zeta_D=D\coth D$,
$\bar\Lambda_D=2\sinh(D)/D-1$, and $\xi=\zeta_D$.

AVF and Gonzalez use the midpoint center, while Itoh--Abe uses the
old endpoint. AVF optimizes $\rho$ at $\theta=1/2$;
Itoh--Abe uses $\rho=\mu/2$.
Each method uses its respective best certified step size, with the
accelerated implicit parameter $s=h/(1-2\gamma h)$.
Five initial points are generated using seeds $11,22,33,44,55$:
\[
(x_1^{(0)},x_2^{(0)})
=\sqrt{0.02}\left(\cos\phi,\frac{\sin\phi}{\sqrt b}\right),
\qquad
x_0^{(0)}=\sqrt{1+(x_1^{(0)})^2+(x_2^{(0)})^2},
\]
where $\phi$ is sampled uniformly from $(0.25,1.2)$.
Thus $f_b(x^{(0)})=0.01$. Each run stops at the first iterate satisfying
$f_b(x_k)-f_b(x^\star)\le10^{-10}$.
All iterates and auxiliary points, including the complete Itoh--Abe
coordinate paths, are verified to remain in $\mathcal X$.

Figure~\ref{fig:dg-two-kappa} shows the objective gaps against outer
iterations and running time for the fixed seed 33.
Time measurements include implicit solves, geometric operations,
and stopping evaluations.
AVF and Gonzalez are evaluated analytically for this objective;
Itoh--Abe uses stable coordinate difference quotients.

Across five initializations, the median outer iteration counts decrease
from $103$ to $39$, $61$ to $29$, and $2687$ to $191$ at $\kappa=10$
for AVF, Gonzalez, and Itoh--Abe, respectively.
At $\kappa=20$, the corresponding reductions are from $371$ to $73$,
$230$ to $56$, and $10740$ to $573$.
The median running times decrease by factors of approximately
$2.63$, $2.11$, and $12.26$ at $\kappa=10$, and
$4.79$, $3.79$, and $16.29$ at $\kappa=20$.

\begin{figure*}[t]
\centering
\includegraphics[width=\textwidth]{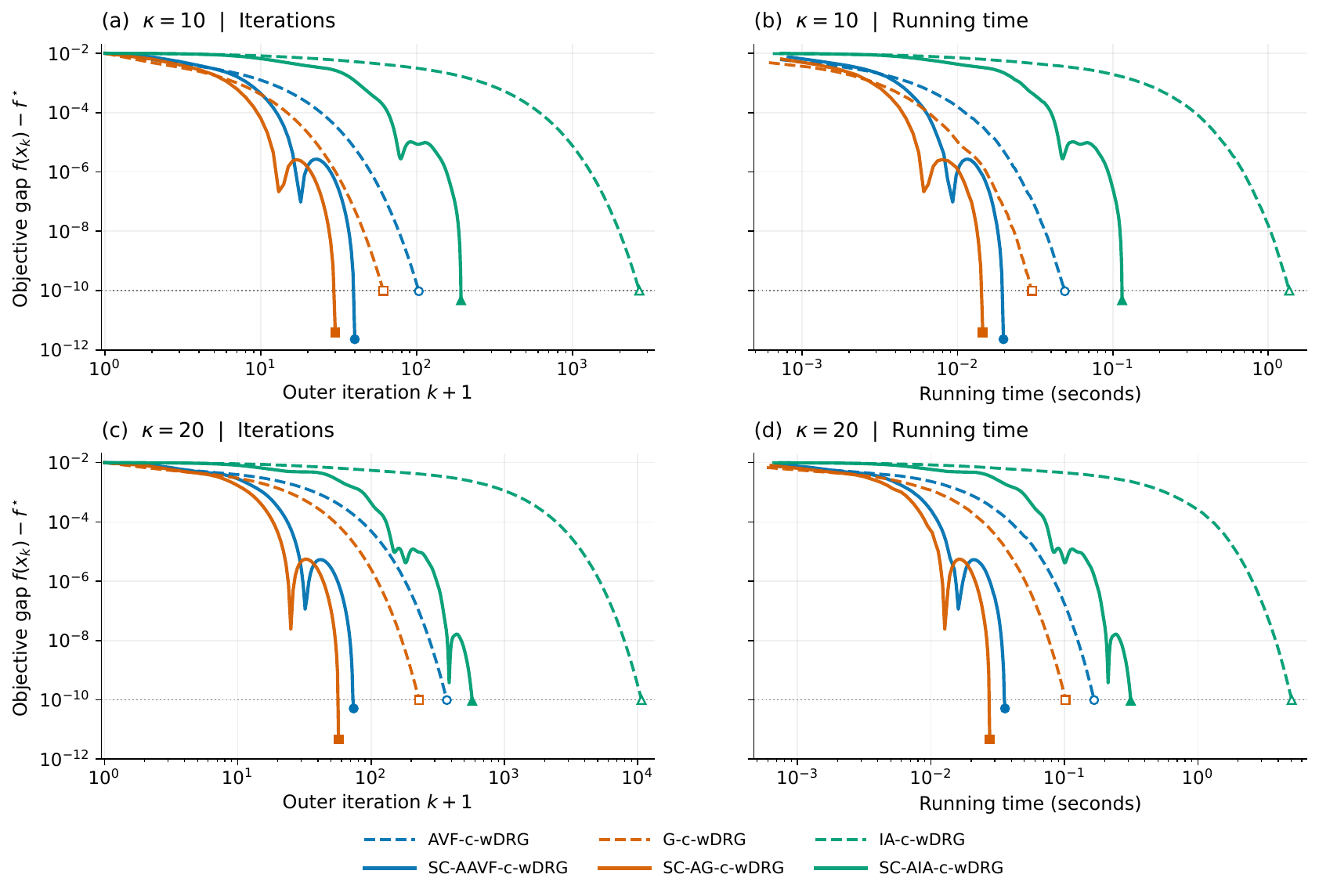}
\caption{Objective gaps for nonaccelerated and strongly convex
accelerated discrete-gradient methods on $\mathbb H^2$, with
$\kappa=10$ (top) and $\kappa=20$ (bottom).
Left: outer iterations; right: running time including implicit solves.
Dashed and solid lines denote nonaccelerated and accelerated methods,
respectively. Colors identify AVF, Gonzalez, and Itoh--Abe. The dotted horizontal line marks the target gap $10^{-10}$.
Each method uses its respective best certified step size.}
\label{fig:dg-two-kappa}
\end{figure*}

\section{Conclusion}
\label{sec:conclusion}
In this work, we introduced the centered weak discrete Riemannian
gradient (c-wDRG) framework for the unified analysis of optimization
methods on Riemannian manifolds. By representing different schemes
through method-specific c-wDRG certificates, the framework covers
Riemannian steepest descent, proximal point, proximal gradient,
implicit midpoint, AVF, Gonzalez, and Itoh--Abe methods.

We established curvature-aware convergence results for both
nonaccelerated and accelerated c-wDRG schemes in geodesically convex
and strongly convex settings. The analysis yields explicit
curvature-dependent stepsize conditions and convergence rates, and
clarifies how the choice of center and manifold geometry affect the
resulting algorithms. Future work includes extensions to general
retractions and broader nonsmooth and composite optimization problems.

\section*{Acknowledgment}
The authors used ChatGPT to assist with language refinement and editing in parts of this manuscript. All generated content was subsequently reviewed and edited by the authors, who take full responsibility for the final manuscript. This work was supported by JST, CREST Grant Number JPMJCR22D3, Japan.

\section*{Declarations}
\textbf{Conflict of interest.}
The authors declare that they have no conflict of interest.

\begin{appendices}

\section{Proof}\label{secA1}

\subsection{Proof of Lemma~\ref{lem:exp-distortion-bound}}
\label{app:exp-distortion-bound}

\begin{proof}
Fix $c\in\mathcal{X}$ and $u\in T_c\mathcal{M}$ such that
$q=\Exp_c(u)\in\mathcal{X}$, and consider the geodesic
\[
\gamma(t):=\Exp_c(tu),
\qquad
t\in[0,1].
\]
Since $\gamma$ is a geodesic,
\[
\|\dot{\gamma}(t)\|_{\gamma(t)}
=
\|u\|_c.
\]
For each $t\in[0,1]$, the curvature operator
\[
v\mapsto R(v,\dot{\gamma}(t))\dot{\gamma}(t)
\]
is self-adjoint and vanishes in the direction of $\dot{\gamma}(t)$.
Recall that, for two linearly independent tangent vectors $v,w$,
the sectional curvature of the plane spanned by $v$ and $w$ is
defined by
\[
\operatorname{Sec}(v,w)
:=
\frac{
\left\langle R(v,w)w,v\right\rangle
}{
\|v\|^2\|w\|^2-\langle v,w\rangle^2
}.
\]
Therefore, taking $w=\dot{\gamma}(t)$, we have
\[
\left\langle
R(v,\dot{\gamma}(t))\dot{\gamma}(t),v
\right\rangle_{\gamma(t)}
=
\operatorname{Sec}(v,\dot{\gamma}(t))
\left(
\|v\|_{\gamma(t)}^2
\|\dot{\gamma}(t)\|_{\gamma(t)}^2
-
\left\langle
v,\dot{\gamma}(t)
\right\rangle_{\gamma(t)}^2
\right).
\]
In particular, if
$\left\langle v,\dot{\gamma}(t)\right\rangle_{\gamma(t)}=0,$
then,
\begin{equation}
\label{eq:sectional-curvature-orthogonal}
\left\langle
R(v,\dot{\gamma}(t))\dot{\gamma}(t),v
\right\rangle_{\gamma(t)}
=
\operatorname{Sec}(v,\dot{\gamma}(t))
\|v\|_{\gamma(t)}^2
\|\dot{\gamma}(t)\|_{\gamma(t)}^2.
\end{equation}
Since
\[
K_{\min}\leq\operatorname{Sec}\leq K_{\max},
\qquad
K:=\max\{|K_{\min}|,|K_{\max}|\},
\]
we have $|\operatorname{Sec}|\leq K$. Define the curvature operator
\[
\mathcal{R}_{\dot{\gamma}(t)}
:
T_{\gamma(t)}\mathcal{M}
\to
T_{\gamma(t)}\mathcal{M},
\qquad
\mathcal{R}_{\dot{\gamma}(t)}(v)
:=
R(v,\dot{\gamma}(t))\dot{\gamma}(t).
\]
Since $\mathcal{R}_{\dot{\gamma}(t)}$ is self-adjoint and vanishes
in the direction of $\dot{\gamma}(t)$, it follows from
\eqref{eq:sectional-curvature-orthogonal} and
$|\operatorname{Sec}|\leq K$ that
\[
\left\|
\mathcal{R}_{\dot{\gamma}(t)}
\right\|_{\mathrm{op}}
=
\sup_{\|v\|_{\gamma(t)}=1}
\left|
\left\langle
\mathcal{R}_{\dot{\gamma}(t)}(v),v
\right\rangle_{\gamma(t)}
\right|
\leq
K\|\dot{\gamma}(t)\|_{\gamma(t)}^2.
\]
Since
$\|\dot{\gamma}(t)\|_{\gamma(t)}
=
\|u\|_c,$
we obtain
\begin{equation}
\label{eq:app-curvature-operator-bound}
\left\|
\mathcal{R}_{\dot{\gamma}(t)}
\right\|_{\mathrm{op}}
\leq
K\|u\|_c^2.
\end{equation}

Recall that a vector field $J$ along a geodesic
$\gamma:[0,1]\to\mathcal{M}$ is a smooth map satisfying
\[
J(t)\in T_{\gamma(t)}\mathcal{M},
\qquad
t\in[0,1].
\]
Such a vector field $J$ is called a \emph{Jacobi field} along $\gamma$
if it satisfies the Jacobi equation
\begin{equation}
\label{eq:app-jacobi-equation}
\nabla_{\dot{\gamma}(t)}
\left(
\nabla_{\dot{\gamma}(t)}J(t)
\right)
+
R\bigl(J(t),\dot{\gamma}(t)\bigr)\dot{\gamma}(t)
=
0,
\qquad t\in[0,1].
\end{equation}

Now fix $v\in T_c\mathcal{M}$ and consider the smooth variation
\[
F(s,t)
:=
\Exp_c\bigl(t(u+sv)\bigr).
\]
For each fixed $s$, the curve
\[
t\longmapsto F(s,t)
\]
is a geodesic. Hence 
\[
\gamma(t)=F(0,t)=\Exp_c(tu).
\]

Moreover, we define,
\begin{equation}
\label{eq:app-jacobi-dexp}
J_v(t)
=
\mathrm{D}\Exp_c(tu)[tv].
\end{equation}
In particular,
\[
J_v(0)=0,
\qquad
J_v(1)=\mathrm{D}\Exp_c(u)[v].
\]
Moreover, differentiating at $t=0$ gives
\[
\nabla_{\dot{\gamma}(0)}J_v(0)=v.
\]
Since $F(s,t)=\Exp_c(t(u+sv))$ is a variation through geodesics,
its variational vector field $J_v$ is a Jacobi field
\cite[Proposition~10.1]{Lee2018}.
\begin{equation}
\label{eq:app-jacobi-field-equation}
\nabla_{\dot{\gamma}(t)}
\left(
\nabla_{\dot{\gamma}(t)}J_v(t)
\right)
+
R\bigl(J_v(t),\dot{\gamma}(t)\bigr)\dot{\gamma}(t)
=
0,
\qquad t\in[0,1].
\end{equation}
Parallel transport $J_v(t)$ back to $T_c\mathcal{M}$ and define
\[
j_v(t):=P_{\gamma(t)\to c}J_v(t).
\]
Then
\[
j_v(0)=0,
\qquad
\dot{j}_v(0)
=
P_{c\to c}
\nabla_{\dot{\gamma}(0)}J_v(0)
=
v,
\qquad
j_v(1)=A_c(u)v.
\]
Moreover,
\begin{equation}
\label{eq:app-parallel-jacobi-equation}
\ddot{j}_v(t)
+
\widetilde{\mathcal{R}}(t)j_v(t)
=
0.
\end{equation}
where
\[
\widetilde{\mathcal{R}}(t)j_v(t)
:=
P_{\gamma(t)\to c}
\left[
R\left(
P_{c\to\gamma(t)}j_v(t),
\dot{\gamma}(t)
\right)
\dot{\gamma}(t)
\right].
\]
Since parallel transport is an isometry,
\eqref{eq:app-curvature-operator-bound} implies
\[
\|\widetilde{\mathcal{R}}(t)\|_{\mathrm{op}}
\leq
K\|u\|_c^2.
\]
Integrating \eqref{eq:app-parallel-jacobi-equation} twice and using
$j_v(0)=0$ and $\dot{j}_v(0)=v$ yields
\begin{equation}
\label{eq:app-jacobi-integral}
j_v(t)
=
tv
-
\int_0^t
(t-s)\widetilde{\mathcal{R}}(s)j_v(s)\,ds.
\end{equation}
Hence
\[
\|j_v(t)\|_c
\leq
t\|v\|_c
+
K\|u\|_c^2
\int_0^t
(t-s)\|j_v(s)\|_c\,ds.
\]
By the standard comparison principle for Volterra integral
inequalities,
\[
\|j_v(t)\|_c
\leq
\frac{
\sinh\!\left(\sqrt{K}\|u\|_c t\right)
}{
\sqrt{K}\|u\|_c
}
\|v\|_c.
\]
Evaluating at $t=1$ gives
\[
\|A_c(u)v\|_c
\leq
\frac{
\sinh\!\left(\sqrt{K}\|u\|_c\right)
}{
\sqrt{K}\|u\|_c
}
\|v\|_c.
\]
Therefore,
\[
\|A_c(u)\|_{\mathrm{op}}
\leq
\frac{
\sinh\!\left(\sqrt{K}\|u\|_c\right)
}{
\sqrt{K}\|u\|_c
}.
\]
Since $\|u\|_c\leq D$ and
\[
t\mapsto\frac{\sinh(\sqrt{K}t)}{\sqrt{K}t}
\]
is nondecreasing for $t\geq0$, we obtain
\begin{equation}
\label{eq:app-JD-bound}
J_D
\leq
\frac{\sinh(\sqrt{K}D)}{\sqrt{K}D}.
\end{equation}
Next, setting $t=1$ in \eqref{eq:app-jacobi-integral} gives
\[
(A_c(u)-I)v
=
-\int_0^1
(1-s)\widetilde{\mathcal{R}}(s)j_v(s)\,ds.
\]
Thus
\[
\begin{aligned}
\|(A_c(u)-I)v\|_c
&\leq
K\|u\|_c^2
\int_0^1
(1-s)\|j_v(s)\|_c\,ds
\\
&\leq
K\|u\|_c^2
\int_0^1
(1-s)
\frac{
\sinh\!\left(\sqrt{K}\|u\|_c s\right)
}{
\sqrt{K}\|u\|_c
}
\,ds\,
\|v\|_c
\\
&=
\left[
\frac{
\sinh\!\left(\sqrt{K}\|u\|_c\right)
}{
\sqrt{K}\|u\|_c
}
-1
\right]\|v\|_c.
\end{aligned}
\]
Hence
\[
\|A_c(u)-I\|_{\mathrm{op}}
\leq
\frac{
\sinh\!\left(\sqrt{K}\|u\|_c\right)
}{
\sqrt{K}\|u\|_c
}
-1.
\]
Therefore,
\[
\frac{\|A_c(u)-I\|_{\mathrm{op}}}{\|u\|_c}
\leq
\frac{1}{\|u\|_c}
\left[
\frac{
\sinh\!\left(\sqrt{K}\|u\|_c\right)
}{
\sqrt{K}\|u\|_c
}
-1
\right].
\]
The right-hand side is nondecreasing in $\|u\|_c>0$, since
\[
\frac{1}{t}
\left(
\frac{\sinh(\sqrt{K}t)}{\sqrt{K}t}-1
\right)
=
\sum_{m=1}^{\infty}
\frac{K^m t^{2m-1}}{(2m+1)!},
\qquad t>0.
\]
Thus, using $\|u\|_c\leq D$,
\[
H_D
\leq
\frac{1}{D}
\left[
\frac{\sinh(\sqrt{K}D)}{\sqrt{K}D}-1
\right],
\]
and consequently
\begin{equation}
\label{eq:app-HD-bound}
DH_D
\leq
\frac{\sinh(\sqrt{K}D)}{\sqrt{K}D}-1.
\end{equation}
Combining \eqref{eq:app-JD-bound} and
\eqref{eq:app-HD-bound}, we obtain
\[
\begin{aligned}
\Lambda_D
=
J_D+DH_D
&\leq
\frac{\sinh(\sqrt{K}D)}{\sqrt{K}D}
+
\frac{\sinh(\sqrt{K}D)}{\sqrt{K}D}
-1
\\
&=
2\frac{\sinh(\sqrt{K}D)}{\sqrt{K}D}-1,
\end{aligned}
\]
which proves \eqref{eq:lambda-D-bound}.
For $K=0$, the curvature tensor vanishes along the relevant radial
geodesics, so the Jacobi equation reduces to
\[
\ddot{j}_v(t)=0.
\]
Since
\[
j_v(0)=0,
\qquad
\dot{j}_v(0)=v,
\]
we have
\[
j_v(t)=tv.
\]
Therefore,
\[
A_c(u)v=j_v(1)=v,
\]
and hence
\[
A_c(u)=I,
\qquad
J_D=1,
\qquad
H_D=0,
\qquad
\Lambda_D=1.
\]
\end{proof}

\subsection{Proof of Theorem~\ref{thm:structure-preserving-certificates}}
\label{app:proof-structure-preserving-certificates}

\begin{proof}
We verify the c-wDRG inequality \eqref{eq:c-wdrg} for each method.


\medskip
\noindent\textbf{(i) Riemannian steepest descent.}

By the Riemannian descent lemma,
\[
f(y)
\leq
f(z)
+
\left\langle
\operatorname{grad}f(z),\Log_z y
\right\rangle_z
+
\frac{L}{2}d^2(y,z),
\]
whereas $\mu$-strong geodesic convexity gives
\[
f(x)
\geq
f(z)
+
\left\langle
\operatorname{grad}f(z),\Log_z x
\right\rangle_z
+
\frac{\mu}{2}d^2(z,x).
\]
Subtracting the second inequality from the first yields
\[
\begin{aligned}
f(y)-f(x)
\leq{}&
\left\langle
\operatorname{grad}f(z),
\Log_z y-\Log_z x
\right\rangle_z
+
\frac{L}{2}d^2(y,z)
-
\frac{\mu}{2}d^2(z,x).
\end{aligned}
\]
Since
\[
\Delta_z(y,x)=\Log_z y-\Log_z x,
\qquad
\|\Delta_z(y,z)\|_z^2=d^2(y,z),
\qquad
\|\Delta_z(z,x)\|_z^2=d^2(z,x),
\]
we obtain
\[
f(y)-f(x)
\leq
\left\langle
\mathcal{G}^{\mathrm{SD}}_f(y,z;z),
\Delta_z(y,x)
\right\rangle_z
+
\frac{L}{2}\|\Delta_z(y,z)\|_z^2
-
\frac{\mu}{2}\|\Delta_z(z,x)\|_z^2.
\]
Hence
\[
(\alpha,\beta,\gamma)
=
\left(
\frac{L}{2},
\frac{\mu}{2},
0
\right).
\]


\medskip
\noindent\textbf{(ii) Riemannian proximal point.}

Strong geodesic convexity applied from $y$ to $x$ gives
\[
f(x)
\geq
f(y)
+
\left\langle
\operatorname{grad}f(y),\Log_y x
\right\rangle_y
+
\frac{\mu}{2}d^2(y,x).
\]
Since
\[
\Delta_y(y,x)
=
-\Log_y x,
\qquad
\|\Delta_y(y,x)\|_y^2
=
d^2(y,x),
\]
it follows that
\[
f(y)-f(x)
\leq
\left\langle
\mathcal{G}^{\mathrm{PP}}_f(y,z;y),
\Delta_y(y,x)
\right\rangle_y
-
\frac{\mu}{2}\|\Delta_y(y,x)\|_y^2.
\]
Thus
\[
(\alpha,\beta,\gamma)
=
\left(
0,
0,
\frac{\mu}{2}
\right).
\]

\noindent\textbf{(iii) Riemannian proximal gradient.}

Let
\[
u:=\Log_z x,
\qquad
\eta_z:=\Log_z y.
\]
By geodesic $L$-smoothness and $\mu$-strong geodesic convexity of
$f_1$,
\[
f_1(y)-f_1(x)
\leq
\left\langle
\operatorname{grad}f_1(z),
\eta_z-u
\right\rangle_z
+
\frac{L}{2}\|\eta_z\|_z^2
-
\frac{\mu}{2}\|u\|_z^2.
\]
By the stationarity condition, there exists
$g_z\in\partial\widetilde f_{2,z}(\eta_z)$ such that
\[
g_z
=
-\operatorname{grad}f_1(z)
-
\frac{1}{h}\eta_z.
\]
Since
\[
\Exp_z(u)=x\in\mathcal X,
\qquad
\Exp_z(\eta_z)=y\in\mathcal X,
\]
the $\rho_{f_2}$-retraction-convexity of $f_2$ gives
\[
\widetilde f_{2,z}(u)
\geq
\widetilde f_{2,z}(\eta_z)
+
\left\langle
g_z,u-\eta_z
\right\rangle_z.
\]
Recalling that
\[
\widetilde f_{2,z}(\eta)
=
f_2\!\left(\Exp_z(\eta)\right)
+
\frac{\rho_{f_2}}{2}\|\eta\|_z^2,
\]
we obtain
\[
f_2(y)-f_2(x)
\leq
-\left\langle
\operatorname{grad}f_1(z),\eta_z-u
\right\rangle_z
-
\frac{1}{h}
\left\langle
\eta_z,\eta_z-u
\right\rangle_z
+
\frac{\rho_{f_2}}{2}
\left(
\|u\|_z^2-\|\eta_z\|_z^2
\right).
\]
Adding the preceding two inequalities and using $f=f_1+f_2$ gives
\[
f(y)-f(x)
\leq
-\frac{1}{h}
\left\langle
\eta_z,\eta_z-u
\right\rangle_z
+
\frac{L-\rho_{f_2}}{2}
\|\eta_z\|_z^2
-
\frac{\mu-\rho_{f_2}}{2}
\|u\|_z^2.
\]
Since
\[
\mathcal G_{f,h}^{\mathrm{PG}}(y,z;z)
=
-\frac{1}{h}\eta_z,
\]
and
\[
\Delta_z(y,x)=\eta_z-u,
\qquad
\Delta_z(y,z)=\eta_z,
\qquad
\|\Delta_z(z,x)\|_z^2=\|u\|_z^2,
\]
it follows that
\[
f(y)-f(x)
\leq
\left\langle
\mathcal G_{f,h}^{\mathrm{PG}}(y,z;z),
\Delta_z(y,x)
\right\rangle_z
+
\frac{L-\rho_{f_2}}{2}
\|\Delta_z(y,z)\|_z^2
-
\frac{\mu-\rho_{f_2}}{2}
\|\Delta_z(z,x)\|_z^2.
\]
Thus
\[
(\alpha,\beta,\gamma)
=
\left(
\frac{L-\rho_{f_2}}{2},
\frac{\mu-\rho_{f_2}}{2},
0
\right).
\]

\medskip
\noindent\textbf{(iv) Riemannian implicit midpoint rule.}

Let
\[
m:=\operatorname{mid}(y,z).
\]
Since $m$ is the midpoint of the selected shortest geodesic from
$z$ to $y$, we have
\[
\Log_m z=-\Log_m y.
\]
By the Riemannian descent lemma,
\[
f(y)
\leq
f(m)
+
\left\langle
\operatorname{grad}f(m),\Log_m y
\right\rangle_m
+
\frac{L}{2}\|\Log_m y\|_m^2,
\]
whereas strong geodesic convexity gives
\[
f(x)
\geq
f(m)
+
\left\langle
\operatorname{grad}f(m),\Log_m x
\right\rangle_m
+
\frac{\mu}{2}\|\Log_m x\|_m^2.
\]
Therefore,
\[
\begin{aligned}
f(y)-f(x)
\leq{}&
\left\langle
\operatorname{grad}f(m),
\Log_m y-\Log_m x
\right\rangle_m
\\
&+
\frac{L}{2}\|\Log_m y\|_m^2
-
\frac{\mu}{2}\|\Log_m x\|_m^2.
\end{aligned}
\]
Moreover,
\[
\Delta_m(y,z)=2\Log_m y,
\]
\[
\Delta_m(z,x)
=
-\Log_m y-\Log_m x,
\]
and
\[
\Delta_m(y,x)
=
\Log_m y-\Log_m x.
\]
Using the parallelogram identity,
\[
\left\|
\Log_m y+\Log_m x
\right\|_m^2
+
\left\|
\Log_m y-\Log_m x
\right\|_m^2
=
2\|\Log_m y\|_m^2
+
2\|\Log_m x\|_m^2,
\]
we obtain
\[
\begin{aligned}
&
\frac{L+\mu}{8}\|\Delta_m(y,z)\|_m^2
-
\frac{\mu}{4}\|\Delta_m(z,x)\|_m^2
-
\frac{\mu}{4}\|\Delta_m(y,x)\|_m^2
\\
&\qquad=
\frac{L+\mu}{2}\|\Log_m y\|_m^2
-
\frac{\mu}{4}
\left(
\|\Log_m y+\Log_m x\|_m^2
+
\|\Log_m y-\Log_m x\|_m^2
\right)
\\
&\qquad=
\frac{L}{2}\|\Log_m y\|_m^2
-
\frac{\mu}{2}\|\Log_m x\|_m^2.
\end{aligned}
\]
Since
\[
\Log_m y-\Log_m x
=
\Delta_m(y,x),
\]
it follows that
\[
\begin{aligned}
f(y)-f(x)
\leq{}&
\left\langle
\mathcal{G}^{\mathrm{MP}}_f(y,z;m),
\Delta_m(y,x)
\right\rangle_m
+
\frac{L+\mu}{8}\|\Delta_m(y,z)\|_m^2
\\
&-
\frac{\mu}{4}\|\Delta_m(z,x)\|_m^2
-
\frac{\mu}{4}\|\Delta_m(y,x)\|_m^2.
\end{aligned}
\]
Therefore,
\[
(\alpha,\beta,\gamma)
=
\left(
\frac{L+\mu}{8},
\frac{\mu}{4},
\frac{\mu}{4}
\right).
\]


\medskip
\noindent\textbf{(v) Geodesic AVF discrete gradient.}

Since $c_\theta$ lies on the selected shortest geodesic from $z$ to
$y$, we have
\[
\Log_{c_\theta} z
=
-\theta\Delta_{c_\theta}(y,z),
\qquad
\Log_{c_\theta} y
=
(1-\theta)\Delta_{c_\theta}(y,z).
\]
For $t\in[0,1]$, let
\[
u_t
:=
(t-\theta)\Delta_{c_\theta}(y,z),
\qquad
q_t
:=
\Exp_{c_\theta}(u_t)
=
\gamma_{z,y}(t).
\]

We first establish the exact centered chain rule. Since
\[
\dot q_t
=
\mathrm{D}\Exp_{c_\theta}(u_t)
\bigl[\dot u_t\bigr]
=
\mathrm{D}\Exp_{c_\theta}(u_t)
\bigl[\Delta_{c_\theta}(y,z)\bigr].
\]
the fundamental theorem of calculus gives
\[
\begin{aligned}
f(y)-f(z)
&=
\int_0^1
\left\langle
\operatorname{grad}f(q_t),
\mathrm{D}\Exp_{c_\theta}(u_t)
\bigl[\Delta_{c_\theta}(y,z)\bigr]
\right\rangle_{q_t}\,dt
\\
&=
\int_0^1
\left\langle
\left(\mathrm{D}\Exp_{c_\theta}(u_t)\right)^\ast
\operatorname{grad}f(q_t),
\Delta_{c_\theta}(y,z)
\right\rangle_{c_\theta}\,dt
\\
&=
\left\langle
\mathcal{G}^{\mathrm{AVF},\theta}_f(y,z;c_\theta),
\Delta_{c_\theta}(y,z)
\right\rangle_{c_\theta}.
\end{aligned}
\]
Thus $\mathcal{G}^{\mathrm{AVF},\theta}_f$ satisfies the exact
centered chain rule.

Next, the Riemannian descent lemma from $c_\theta$ to $y$ and strong
geodesic convexity from $c_\theta$ to $x$ give
\[
\begin{aligned}
f(y)-f(x)
\leq{}&
\left\langle
\operatorname{grad}f(c_\theta),
(1-\theta)\Delta_{c_\theta}(y,z)-\Log_{c_\theta}x
\right\rangle_{c_\theta}
\\
&+
\frac{L}{2}(1-\theta)^2
\|\Delta_{c_\theta}(y,z)\|_{c_\theta}^2
-
\frac{\mu}{2}\|\Log_{c_\theta}x\|_{c_\theta}^2.
\end{aligned}
\]
Adding and subtracting
$\mathcal{G}^{\mathrm{AVF},\theta}_f(y,z;c_\theta)$ yields
\[
\begin{aligned}
f(y)-f(x)
\leq{}&
\left\langle
\mathcal{G}^{\mathrm{AVF},\theta}_f(y,z;c_\theta),
(1-\theta)\Delta_{c_\theta}(y,z)-\Log_{c_\theta}x
\right\rangle_{c_\theta}
\\
&-
(1-\theta)
\left\langle
\mathcal{G}^{\mathrm{AVF},\theta}_f(y,z;c_\theta)
-
\operatorname{grad}f(c_\theta),
\Delta_{c_\theta}(y,z)
\right\rangle_{c_\theta}
\\
&+
\left\langle
\mathcal{G}^{\mathrm{AVF},\theta}_f(y,z;c_\theta)
-
\operatorname{grad}f(c_\theta),
\Log_{c_\theta}x
\right\rangle_{c_\theta}
\\
&+
\frac{L}{2}(1-\theta)^2
\|\Delta_{c_\theta}(y,z)\|_{c_\theta}^2
-
\frac{\mu}{2}\|\Log_{c_\theta}x\|_{c_\theta}^2.
\end{aligned}
\]

On the other hand, using the exact chain rule,
\[
f(y)-f(x)
=
\left\langle
\mathcal{G}^{\mathrm{AVF},\theta}_f(y,z;c_\theta),
\Delta_{c_\theta}(y,z)
\right\rangle_{c_\theta}
+
f(z)-f(x).
\]
The Riemannian descent lemma from $c_\theta$ to $z$ and strong
geodesic convexity from $c_\theta$ to $x$ give
\[
\begin{aligned}
f(z)-f(x)
\leq{}&
\left\langle
\operatorname{grad}f(c_\theta),
-\theta\Delta_{c_\theta}(y,z)-\Log_{c_\theta}x
\right\rangle_{c_\theta}
\\
&+
\frac{L}{2}\theta^2
\|\Delta_{c_\theta}(y,z)\|_{c_\theta}^2
-
\frac{\mu}{2}\|\Log_{c_\theta}x\|_{c_\theta}^2.
\end{aligned}
\]
Therefore,
\[
\begin{aligned}
f(y)-f(x)
\leq{}&
\left\langle
\mathcal{G}^{\mathrm{AVF},\theta}_f(y,z;c_\theta),
(1-\theta)\Delta_{c_\theta}(y,z)-\Log_{c_\theta}x
\right\rangle_{c_\theta}
\\
&+
\theta
\left\langle
\mathcal{G}^{\mathrm{AVF},\theta}_f(y,z;c_\theta)
-
\operatorname{grad}f(c_\theta),
\Delta_{c_\theta}(y,z)
\right\rangle_{c_\theta}
\\
&+
\left\langle
\mathcal{G}^{\mathrm{AVF},\theta}_f(y,z;c_\theta)
-
\operatorname{grad}f(c_\theta),
\Log_{c_\theta}x
\right\rangle_{c_\theta}
\\
&+
\frac{L}{2}\theta^2
\|\Delta_{c_\theta}(y,z)\|_{c_\theta}^2
-
\frac{\mu}{2}\|\Log_{c_\theta}x\|_{c_\theta}^2.
\end{aligned}
\]

Multiplying the first estimate by $\theta$ and the second one by
$1-\theta$, and then adding them, cancels the terms involving
$\Delta_{c_\theta}(y,z)$ in the gradient error and yields
\begin{equation}
\label{eq:avf-pre-error-bound}
\begin{aligned}
f(y)-f(x)
\leq{}&
\left\langle
\mathcal{G}^{\mathrm{AVF},\theta}_f(y,z;c_\theta),
\Delta_{c_\theta}(y,x)
\right\rangle_{c_\theta}
\\
&+
\left\langle
\mathcal{G}^{\mathrm{AVF},\theta}_f(y,z;c_\theta)
-
\operatorname{grad}f(c_\theta),
\Log_{c_\theta}x
\right\rangle_{c_\theta}
\\
&+
\frac{L}{2}\theta(1-\theta)
\|\Delta_{c_\theta}(y,z)\|_{c_\theta}^2
-
\frac{\mu}{2}\|\Log_{c_\theta}x\|_{c_\theta}^2.
\end{aligned}
\end{equation}

We now estimate
\[
\left\|
\mathcal{G}^{\mathrm{AVF},\theta}_f(y,z;c_\theta)
-
\operatorname{grad}f(c_\theta)
\right\|_{c_\theta}.
\]
By \eqref{eq:exp-distortion-operator},
\[
\left(\mathrm{D}\Exp_{c_\theta}(u_t)\right)^\ast
=
A_{c_\theta}(u_t)^\ast P_{q_t\to c_\theta}.
\]
Hence
\[
\begin{aligned}
&
\mathcal{G}^{\mathrm{AVF},\theta}_f(y,z;c_\theta)
-
\operatorname{grad}f(c_\theta)
\\
&\qquad=
\int_0^1
\left[
A_{c_\theta}(u_t)^\ast
P_{q_t\to c_\theta}\operatorname{grad}f(q_t)
-
\operatorname{grad}f(c_\theta)
\right]dt.
\end{aligned}
\]
Adding and subtracting
$A_{c_\theta}(u_t)^\ast\operatorname{grad}f(c_\theta)$ gives
\[
\begin{aligned}
&
\left\|
\mathcal{G}^{\mathrm{AVF},\theta}_f(y,z;c_\theta)
-
\operatorname{grad}f(c_\theta)
\right\|_{c_\theta}
\\
&\leq
\int_0^1
\|A_{c_\theta}(u_t)\|_{\mathrm{op}}
\left\|
P_{q_t\to c_\theta}\operatorname{grad}f(q_t)
-
\operatorname{grad}f(c_\theta)
\right\|_{c_\theta}\,dt
\\
&\quad+
\int_0^1
\|A_{c_\theta}(u_t)-I\|_{\mathrm{op}}
\|\operatorname{grad}f(c_\theta)\|_{c_\theta}\,dt.
\end{aligned}
\]
Since $x^\star$ is an interior minimizer,
\[
\operatorname{grad}f(x^\star)=0.
\]
By $L$-smoothness and $\operatorname{diam}(\mathcal X)\leq D$,
\[
\begin{aligned}
\|\operatorname{grad}f(c_\theta)\|_{c_\theta}
&=
\left\|
\operatorname{grad}f(c_\theta)
-
P_{x^\star\to c_\theta}\operatorname{grad}f(x^\star)
\right\|_{c_\theta}
\\
&\leq
L\,d(c_\theta,x^\star)
\leq
LD.
\end{aligned}
\]
Moreover,
\[
d(q_t,c_\theta)
=
\|u_t\|_{c_\theta}
=
|t-\theta|
\|\Delta_{c_\theta}(y,z)\|_{c_\theta}.
\]
Using \eqref{eq:exp-distortion-constants}, we obtain
\[
\begin{aligned}
\left\|
\mathcal{G}^{\mathrm{AVF},\theta}_f(y,z;c_\theta)
-
\operatorname{grad}f(c_\theta)
\right\|_{c_\theta}
&\leq
L(J_D+DH_D)
\int_0^1 |t-\theta|\,dt\,
\|\Delta_{c_\theta}(y,z)\|_{c_\theta}
\\
&=
L\Lambda_D\omega_\theta
\|\Delta_{c_\theta}(y,z)\|_{c_\theta}
\\
&\leq
L\bar{\Lambda}_D\omega_\theta
\|\Delta_{c_\theta}(y,z)\|_{c_\theta}.
\end{aligned}
\]
where the last inequality follows from
Lemma~\ref{lem:exp-distortion-bound}, which gives
$\Lambda_D\leq\bar{\Lambda}_D$.
Therefore, for every $0<\rho<\mu$, Young's inequality gives
\[
\begin{aligned}
&
\left\langle
\mathcal{G}^{\mathrm{AVF},\theta}_f(y,z;c_\theta)
-
\operatorname{grad}f(c_\theta),
\Log_{c_\theta}x
\right\rangle_{c_\theta}
\\
&\leq
\frac{L^2\bar{\Lambda}_D^2\omega_\theta^2}{2\rho}
\|\Delta_{c_\theta}(y,z)\|_{c_\theta}^2
+
\frac{\rho}{2}\|\Log_{c_\theta}x\|_{c_\theta}^2.
\end{aligned}
\]
Substituting this estimate into
\eqref{eq:avf-pre-error-bound} yields
\[
\begin{aligned}
f(y)-f(x)
\leq{}&
\left\langle
\mathcal{G}^{\mathrm{AVF},\theta}_f(y,z;c_\theta),
\Delta_{c_\theta}(y,x)
\right\rangle_{c_\theta}
+
\left(
\frac{L}{2}\theta(1-\theta)
+
\frac{L^2\bar{\Lambda}_D^2\omega_\theta^2}{2\rho}
\right)
\|\Delta_{c_\theta}(y,z)\|_{c_\theta}^2
\\
&-
\frac{\mu-\rho}{2}
\|\Log_{c_\theta}x\|_{c_\theta}^2.
\end{aligned}
\]
We set
\[
\beta
:=
\frac{\mu-\rho}{2}(1-\theta),
\qquad
\gamma
:=
\frac{\mu-\rho}{2}\theta.
\]
Since
\[
\Delta_{c_\theta}(z,x)
=
-\theta\Delta_{c_\theta}(y,z)-\Log_{c_\theta}x,
\]
and
\[
\Delta_{c_\theta}(y,x)
=
(1-\theta)\Delta_{c_\theta}(y,z)-\Log_{c_\theta}x,
\]
we have
\[
\begin{aligned}
&
\beta\|\Delta_{c_\theta}(z,x)\|_{c_\theta}^2
+
\gamma\|\Delta_{c_\theta}(y,x)\|_{c_\theta}^2
\\
&\qquad=
\frac{\mu-\rho}{2}
\|\Log_{c_\theta}x\|_{c_\theta}^2
+
\frac{\mu-\rho}{2}\theta(1-\theta)
\|\Delta_{c_\theta}(y,z)\|_{c_\theta}^2.
\end{aligned}
\]
Equivalently,
\[
\begin{aligned}
-\frac{\mu-\rho}{2}
\|\Log_{c_\theta}x\|_{c_\theta}^2
={}&
-\beta\|\Delta_{c_\theta}(z,x)\|_{c_\theta}^2
-\gamma\|\Delta_{c_\theta}(y,x)\|_{c_\theta}^2
\\
&+
\frac{\mu-\rho}{2}\theta(1-\theta)
\|\Delta_{c_\theta}(y,z)\|_{c_\theta}^2.
\end{aligned}
\]
Therefore,
\[
\begin{aligned}
f(y)-f(x)
\leq{}&
\left\langle
\mathcal{G}^{\mathrm{AVF},\theta}_f(y,z;c_\theta),
\Delta_{c_\theta}(y,x)
\right\rangle_{c_\theta}
+
\alpha\|\Delta_{c_\theta}(y,z)\|_{c_\theta}^2
\\
&-
\beta\|\Delta_{c_\theta}(z,x)\|_{c_\theta}^2
-
\gamma\|\Delta_{c_\theta}(y,x)\|_{c_\theta}^2,
\end{aligned}
\]
where
\[
\alpha
=
\frac{L+\mu-\rho}{2}\theta(1-\theta)
+
\frac{L^2\bar{\Lambda}_D^2\omega_\theta^2}{2\rho}.
\]
This proves the claimed AVF certificate.

\medskip
\noindent\textbf{(vi) Riemannian Gonzalez discrete gradient.}

Let
\[
m:=\operatorname{mid}(y,z).
\]
By definition,
\[
\mathcal{G}^{\mathrm{G}}_f(y,z;m)
=
\operatorname{grad}f(m)
+
\frac{
f(y)-f(z)
-
\left\langle
\operatorname{grad}f(m),
\Delta_m(y,z)
\right\rangle_m
}{
\|\Delta_m(y,z)\|_m^2
}
\Delta_m(y,z).
\]
Consequently,
\[
\left\langle
\mathcal{G}^{\mathrm{G}}_f(y,z;m),
\Delta_m(y,z)
\right\rangle_m
=
f(y)-f(z),
\]
so the Gonzalez discrete gradient satisfies the exact centered chain
rule.

Since $m$ is the midpoint of the selected shortest geodesic from
$z$ to $y$,
\[
\Log_m z=-\Log_m y,
\qquad
\Delta_m(y,z)=2\Log_m y.
\]
The Riemannian descent lemma from $m$ to $y$ gives
\[
f(y)
\leq
f(m)
+
\left\langle
\operatorname{grad}f(m),\Log_m y
\right\rangle_m
+
\frac{L}{2}\|\Log_m y\|_m^2,
\]
whereas strong geodesic convexity from $m$ to $z$ gives
\[
f(z)
\geq
f(m)
+
\left\langle
\operatorname{grad}f(m),\Log_m z
\right\rangle_m
+
\frac{\mu}{2}\|\Log_m z\|_m^2.
\]
Therefore,
\[
f(y)-f(z)
-
\left\langle
\operatorname{grad}f(m),
\Delta_m(y,z)
\right\rangle_m
\leq
\frac{L-\mu}{8}
\|\Delta_m(y,z)\|_m^2.
\]
Interchanging $y$ and $z$ gives the reverse estimate, and hence
\[
\left|
f(y)-f(z)
-
\left\langle
\operatorname{grad}f(m),
\Delta_m(y,z)
\right\rangle_m
\right|
\leq
\frac{L-\mu}{8}
\|\Delta_m(y,z)\|_m^2.
\]

From part~(iv),
\[
\begin{aligned}
f(y)-f(x)
-
\left\langle
\operatorname{grad}f(m),
\Delta_m(y,x)
\right\rangle_m
\leq{}&
\frac{L+\mu}{8}
\|\Delta_m(y,z)\|_m^2
-
\frac{\mu}{4}
\|\Delta_m(z,x)\|_m^2
\\
&-
\frac{\mu}{4}
\|\Delta_m(y,x)\|_m^2.
\end{aligned}
\]
Moreover,
\[
\begin{aligned}
&
\left\langle
\mathcal{G}^{\mathrm{G}}_f(y,z;m),
\Delta_m(y,x)
\right\rangle_m
\\
&\qquad=
\left\langle
\operatorname{grad}f(m),
\Delta_m(y,x)
\right\rangle_m
\\
&\qquad\quad+
\frac{
f(y)-f(z)
-
\left\langle
\operatorname{grad}f(m),
\Delta_m(y,z)
\right\rangle_m
}{
\|\Delta_m(y,z)\|_m^2
}
\left\langle
\Delta_m(y,z),
\Delta_m(y,x)
\right\rangle_m.
\end{aligned}
\]
Hence, by the preceding residual bound and the Cauchy--Schwarz
inequality,
\[
\begin{aligned}
&
f(y)-f(x)
-
\left\langle
\mathcal{G}^{\mathrm{G}}_f(y,z;m),
\Delta_m(y,x)
\right\rangle_m
\\
&\leq
\frac{L+\mu}{8}
\|\Delta_m(y,z)\|_m^2
-
\frac{\mu}{4}
\|\Delta_m(z,x)\|_m^2
-
\frac{\mu}{4}
\|\Delta_m(y,x)\|_m^2
\\
&\qquad+
\frac{L-\mu}{8}
\|\Delta_m(y,z)\|_m
\|\Delta_m(y,x)\|_m.
\end{aligned}
\]
By Young's inequality,
\[
\frac{L-\mu}{8}
\|\Delta_m(y,z)\|_m
\|\Delta_m(y,x)\|_m
\leq
\frac{(L-\mu)^2}{16\mu}
\|\Delta_m(y,z)\|_m^2
+
\frac{\mu}{16}
\|\Delta_m(y,x)\|_m^2.
\]
Therefore,
\[
\begin{aligned}
f(y)-f(x)
\leq{}&
\left\langle
\mathcal{G}^{\mathrm{G}}_f(y,z;m),
\Delta_m(y,x)
\right\rangle_m
+
\left(
\frac{L+\mu}{8}
+
\frac{(L-\mu)^2}{16\mu}
\right)
\|\Delta_m(y,z)\|_m^2
\\
&-
\frac{\mu}{4}
\|\Delta_m(z,x)\|_m^2
-
\frac{3\mu}{16}
\|\Delta_m(y,x)\|_m^2.
\end{aligned}
\]
Dropping the last nonpositive term gives
\[
\begin{aligned}
f(y)-f(x)
\leq{}&
\left\langle
\mathcal{G}^{\mathrm{G}}_f(y,z;m),
\Delta_m(y,x)
\right\rangle_m
+
\left(
\frac{L+\mu}{8}
+
\frac{(L-\mu)^2}{16\mu}
\right)
\|\Delta_m(y,z)\|_m^2
\\
&-
\frac{\mu}{4}
\|\Delta_m(z,x)\|_m^2.
\end{aligned}
\]
Hence
\[
(\alpha,\beta,\gamma)
=
\left(
\frac{L+\mu}{8}
+
\frac{(L-\mu)^2}{16\mu},
\frac{\mu}{4},
0
\right).
\]
\medskip
\noindent\textbf{(vii) Riemannian Itoh--Abe discrete gradient.}

Write
\[
\Delta_z(y,z)
=
\Log_z y
=
\sum_{j=1}^d s_jE_j.
\]
For each $j$, define
\[
u_j(t)
:=
u^{[j-1]}+t s_jE_j,
\qquad
\eta_j(t)
:=
\Exp_z(u_j(t)),
\qquad
0\leq t\leq1.
\]
By the fundamental theorem of calculus, 
the $a_j(y,z)$ satisfies
\[
a_j(y,z)
=
\int_0^1
\left\langle
\operatorname{grad}f(\eta_j(t)),
\mathrm{D}\Exp_z(u_j(t))[E_j]
\right\rangle_{\eta_j(t)}
\,dt.
\]
This formula also agrees with the continuous definition when $s_j=0$.
By the definition of $A_z(u)$,
\[
A_z(u)
=
P_{\Exp_z(u)\to z}\circ\mathrm{D}\Exp_z(u),
\]
and hence
\[
\begin{aligned}
&a_j(y,z)
-
\left\langle
\operatorname{grad}f(z),E_j
\right\rangle_z
\\
&\qquad=
\int_0^1
\Big\langle
A_z(u_j(t))^\ast
P_{\eta_j(t)\to z}
\operatorname{grad}f(\eta_j(t))
-
\operatorname{grad}f(z),
E_j
\Big\rangle_z
\,dt.
\end{aligned}
\]
Because the basis $\{E_1,\ldots,E_d\}$ is orthonormal,
\[
\|u_j(t)\|_z
\leq
\|\Delta_z(y,z)\|_z.
\]
Moreover, since $x^\star$ is an interior minimizer,
\[
\operatorname{grad}f(x^\star)=0,
\]
and $L$-smoothness gives
\[
\begin{aligned}
\|\operatorname{grad}f(z)\|_z
&=
\left\|
\operatorname{grad}f(z)
-
P_{x^\star\to z}\operatorname{grad}f(x^\star)
\right\|_z
\\
&\leq
L\,d(z,x^\star)
\leq
LD.
\end{aligned}
\]
Using \eqref{eq:exp-distortion-constants}, we therefore obtain
\[
\begin{aligned}
\left|
a_j(y,z)
-
\left\langle
\operatorname{grad}f(z),E_j
\right\rangle_z
\right|
&\leq
LJ_D\|\Delta_z(y,z)\|_z
+
LDH_D\|\Delta_z(y,z)\|_z
\\
&=
L(J_D+DH_D)\|\Delta_z(y,z)\|_z
\\
&=
L\Lambda_D\|\Delta_z(y,z)\|_z
\\
&\leq
L\bar{\Lambda}_D\|\Delta_z(y,z)\|_z.
\end{aligned}
\]
Consequently,
\[
\begin{aligned}
\left\|
\mathcal{G}^{\mathrm{IA}}_f(y,z;z)
-
\operatorname{grad}f(z)
\right\|_z^2
&=
\sum_{j=1}^d
\left|
a_j(y,z)
-
\left\langle
\operatorname{grad}f(z),E_j
\right\rangle_z
\right|^2
\\
&\leq
dL^2\bar{\Lambda}_D^2
\|\Delta_z(y,z)\|_z^2,
\end{aligned}
\]
and hence
\begin{equation}
\label{eq:ia-gradient-error}
\left\|
\mathcal{G}^{\mathrm{IA}}_f(y,z;z)
-
\operatorname{grad}f(z)
\right\|_z
\leq
\sqrt{d}\,L\bar{\Lambda}_D
\|\Delta_z(y,z)\|_z.
\end{equation}
By the definition of the coordinate quotients,
\[
\begin{aligned}
\left\langle
\mathcal{G}^{\mathrm{IA}}_f(y,z;z),
\Log_z y
\right\rangle_z
&=
\sum_{j=1}^d a_j(y,z)s_j
\\
&=
\sum_{j=1}^d
\left[
f(w^{[j]})-f(w^{[j-1]})
\right]
\\
&=
f(y)-f(z).
\end{aligned}
\]
Hence the Riemannian Itoh--Abe discrete gradient satisfies the exact
centered chain rule.
Since
\[
\Delta_z(z,x)
=
-\Log_z x,
\]
we have
\[
\Delta_z(y,x)
=
\Delta_z(y,z)+\Delta_z(z,x).
\]
Strong geodesic convexity applied from $z$ to $x$ gives
\[
f(z)-f(x)
\leq
\left\langle
\operatorname{grad}f(z),
\Delta_z(z,x)
\right\rangle_z
-
\frac{\mu}{2}
\|\Delta_z(z,x)\|_z^2.
\]
Combining this inequality with the exact chain rule yields
\[
\begin{aligned}
f(y)-f(x)
\leq{}&
\left\langle
\mathcal{G}^{\mathrm{IA}}_f(y,z;z),
\Delta_z(y,x)
\right\rangle_z
\\
&+
\left\langle
\operatorname{grad}f(z)
-
\mathcal{G}^{\mathrm{IA}}_f(y,z;z),
\Delta_z(z,x)
\right\rangle_z
-
\frac{\mu}{2}
\|\Delta_z(z,x)\|_z^2.
\end{aligned}
\]
Using \eqref{eq:ia-gradient-error} and Young's inequality, for every
$0<\rho<\mu$,
\[
\begin{aligned}
&
\left\langle
\operatorname{grad}f(z)
-
\mathcal{G}^{\mathrm{IA}}_f(y,z;z),
\Delta_z(z,x)
\right\rangle_z
\\
&\leq
\sqrt{d}\,L\bar{\Lambda}_D
\|\Delta_z(y,z)\|_z
\|\Delta_z(z,x)\|_z
\\
&\leq
\frac{dL^2\bar{\Lambda}_D^2}{2\rho}
\|\Delta_z(y,z)\|_z^2
+
\frac{\rho}{2}
\|\Delta_z(z,x)\|_z^2.
\end{aligned}
\]
Therefore,
\[
\begin{aligned}
f(y)-f(x)
\leq{}&
\left\langle
\mathcal{G}^{\mathrm{IA}}_f(y,z;z),
\Delta_z(y,x)
\right\rangle_z
+
\frac{dL^2\bar{\Lambda}_D^2}{2\rho}
\|\Delta_z(y,z)\|_z^2
\\
&-
\frac{\mu-\rho}{2}
\|\Delta_z(z,x)\|_z^2.
\end{aligned}
\]
Thus
\[
(\alpha,\beta,\gamma)
=
\left(
\frac{dL^2\bar{\Lambda}_D^2}{2\rho},
\frac{\mu-\rho}{2},
0
\right).
\]

This completes the proof.
\end{proof}

\section{Strongly Convex Convergence Rates}
\label{app:optimal-strongly-convex-rates}

\begin{sidewaystable*}[p]
\centering
\caption{
Optimal certified linear contraction factors and corresponding
step sizes for the nonaccelerated Riemannian methods in the
geodesically strongly convex setting.
For the geodesic AVF and Itoh--Abe methods, the auxiliary
parameter $\rho\in(0,\mu)$ is also optimized.
}
\label{tab:best-step-rate}


\begin{minipage}{0.96\textheight}
\raggedright
\small
For the implicit-midpoint and Gonzalez methods, define
\[
\begin{aligned}
S_{\mathrm{MP}}
&:=
(L+3\mu)(\zeta_D+3)
+(L-5\mu)(1-\delta_D),
\\
S_{\mathrm G}
&:=
(L^2+3\mu^2)(\zeta_D+3)
+(L^2-5\mu^2)(1-\delta_D).
\end{aligned}
\]

For the geodesic AVF method, for each fixed
$\theta\in[0,1]$, define
\[
m_\theta
:=
\frac{
\mu
\sqrt{
L\bar{\Lambda}_D^2
[\theta^2+(1-\theta)^2]^2
+
4\mu\theta(1-\theta)
}
}{
\bar{\Lambda}_D\sqrt{L}\,
[\theta^2+(1-\theta)^2]
+
\sqrt{
L\bar{\Lambda}_D^2
[\theta^2+(1-\theta)^2]^2
+
4\mu\theta(1-\theta)
}
}.
\]
Furthermore, let
\[
\begin{aligned}
D_\theta
:={}&
\frac{L+m_\theta}{2}\theta(1-\theta)
+
\frac{
L^2\bar{\Lambda}_D^2
[\theta^2+(1-\theta)^2]^2
}{
8(\mu-m_\theta)
}
+
\frac{m_\theta}{2}
[\zeta_D(1-\theta)^2-\theta-\delta_D\theta^2]
\\
&+
\Bigg\{
\Bigg[
\frac{L+m_\theta}{2}\theta(1-\theta)
+
\frac{
L^2\bar{\Lambda}_D^2
[\theta^2+(1-\theta)^2]^2
}{
8(\mu-m_\theta)
}
+
\frac{m_\theta}{2}
[\zeta_D(1-\theta)^2-\theta-\delta_D\theta^2]
\Bigg]^2
\\[-1mm]
&\qquad
+
2m_\theta\theta
\Bigg[
(L+m_\theta)\theta(1-\theta)
+
\frac{
L^2\bar{\Lambda}_D^2
[\theta^2+(1-\theta)^2]^2
}{
4(\mu-m_\theta)
}
\\[-1mm]
&\hspace{55mm}
+
m_\theta(1-\theta)
[\delta_D\theta+\zeta_D(1-\theta)]
\Bigg]
\Bigg\}^{1/2}.
\end{aligned}
\]
\end{minipage}

\vspace{3mm}


\renewcommand{\arraystretch}{1.7}
\setlength{\tabcolsep}{7pt}
\small

\resizebox{\textheight}{!}{%
\begin{tabular}{@{}lcc@{}}
\toprule
\textbf{Method}
&
\textbf{Optimal certified step size $h_\star$}
&
\textbf{Optimal certified contraction factor $q_\star$}
\\
\midrule

Steepest descent
&
$\displaystyle
\frac{2}{L+\zeta_D\mu}
$
&
$\displaystyle
1-\frac{2\mu}{L+\zeta_D\mu}
$
\\[3mm]

Proximal point
&
$\displaystyle h\to+\infty$
&
$\displaystyle \inf_{h>0}q_h=0$
\\[3mm]

Proximal gradient
&
$\displaystyle
\frac{
2
}{
(L-\rho_{f_2})+\zeta_D(\mu-\rho_{f_2})
}
$
&
$\displaystyle
1-
\frac{
2(\mu-\rho_{f_2})
}{
(L-\rho_{f_2})+\zeta_D(\mu-\rho_{f_2})
}
$
\\[3mm]

Implicit midpoint
&
$\displaystyle
\begin{cases}
\displaystyle
\frac{16}{
L+\mu(\zeta_D-\delta_D-1)
+
\sqrt{
[L+\mu(\zeta_D-\delta_D-1)]^2
+
16\mu[L+\mu(\zeta_D+\delta_D+1)]
}
},
& S_{\mathrm{MP}}\geq0,
\\[4mm]
\displaystyle
\frac{2}{\mu},
& S_{\mathrm{MP}}<0
\end{cases}
$
&
$\displaystyle
\begin{cases}
\displaystyle
1-
\frac{16\mu}{
L+\mu(\zeta_D-\delta_D+7)
+
\sqrt{
[L+\mu(\zeta_D-\delta_D-1)]^2
+
16\mu[L+\mu(\zeta_D+\delta_D+1)]
}
},
& S_{\mathrm{MP}}\geq0,
\\[4mm]
\displaystyle
\frac{
L+\mu(\zeta_D-2)
}{
\mu(\zeta_D+3)
},
& S_{\mathrm{MP}}<0
\end{cases}
$
\\[5mm]

Geodesic AVF
&
$\displaystyle
\frac{2}{D_\theta}
$
&
$\displaystyle
1-
\frac{
2m_\theta
}{
D_\theta+2m_\theta\theta
}
$
\\[4mm]

Gonzalez
&
$\displaystyle
\begin{cases}
\displaystyle
\frac{
32\mu
}{
L^2+(7+\zeta_D-\delta_D)\mu^2
+
\sqrt{
[L^2+(7+\zeta_D-\delta_D)\mu^2]^2
-
32(1-\delta_D)\mu^4
}
},
& S_{\mathrm G}\geq0,
\\[4mm]
\displaystyle
\frac{2}{\mu},
& S_{\mathrm G}<0
\end{cases}
$
&
$\displaystyle
\begin{cases}
\displaystyle
1-
\frac{
16\mu^2
}{
L^2+(7+\zeta_D-\delta_D)\mu^2
+
\sqrt{
[L^2+(7+\zeta_D-\delta_D)\mu^2]^2
-
32(1-\delta_D)\mu^4
}
},
& S_{\mathrm G}\geq0,
\\[4mm]
\displaystyle
1-
\frac{
5\mu^2-L^2
}{
\mu^2(\zeta_D+3)
},
& S_{\mathrm G}<0
\end{cases}
$
\\[5mm]

Itoh--Abe
&
$\displaystyle
\frac{
4\mu
}{
4dL^2\bar{\Lambda}_D^2+\zeta_D\mu^2
}
$
&
$\displaystyle
1-
\frac{
2\mu^2
}{
4dL^2\bar{\Lambda}_D^2+\zeta_D\mu^2
}
$
\\[3mm]

\bottomrule
\end{tabular}%
}

\vspace{2mm}

\begin{minipage}{0.96\textheight}
\raggedright
\footnotesize
\textit{Notes.}
The bounds use the nonaccelerated geometric and step-size
assumptions, with $L\geq\mu>0$.
For proximal gradient, assume $\mu>\rho_{f_2}$.
When $\delta_D\geq0$, the first branch applies to both
the implicit-midpoint and Gonzalez methods.
For proximal point, the optimal factor is an infimum
approached as $h\to+\infty$.
For geodesic AVF,
$\rho_\theta^\star=\mu-m_\theta$ for each fixed $\theta$;
for Itoh--Abe, $\rho_\star=\mu/2$.
\end{minipage}

\end{sidewaystable*}

\begin{table*}[t]
\centering
\caption{
Largest certified admissible stepsizes and the corresponding best
accelerated linear contraction factors for the structure-preserving
Riemannian methods in the geodesically strongly convex setting.
For the geodesic AVF and Itoh--Abe methods, the auxiliary parameter
$\rho\in(0,\mu)$ is also optimized.
}
\label{tab:optimal-accelerated-strongly-convex-rates}


\begin{minipage}{0.96\textwidth}
\raggedright
\small
Let $\xi \geq \zeta_D, \quad \delta_D+2+3 \xi-4 \zeta_D>0 $, and let $t_{\star}(\xi)$ denote the distortion
threshold defined in \eqref{eq:strong-distortion-threshold}.

For the geodesic AVF method, for each fixed
$\theta\in[0,1)$ satisfying
$1-2(1-\delta_D)\theta>0$, define
\[
\omega_\theta
:=
\frac{\theta^2+(1-\theta)^2}{2},
\qquad
N_\theta
:=
1-2(1-\delta_D)\theta.
\]
The remaining quantities used below are
\[
m_\theta
:=
\frac{
\mu\sqrt{
L\bar{\Lambda}_D^2\omega_\theta^2
+\mu\theta(1-\theta)
}
}{
\bar{\Lambda}_D\sqrt{L}\,\omega_\theta
+
\sqrt{
L\bar{\Lambda}_D^2\omega_\theta^2
+\mu\theta(1-\theta)
}
},
\]
and
\[
A_\theta
:=
L\theta(1-\theta)
+
m_\theta
\bigl[2\theta-(2-\delta_D)\theta^2\bigr]
+
\frac{
L^2\bar{\Lambda}_D^2\omega_\theta^2
}{
\mu-m_\theta
}.
\]

\end{minipage}

\vspace{3mm}


\renewcommand{\arraystretch}{2.0}
\setlength{\tabcolsep}{5pt}
\small

\begin{tabular*}{\textwidth}{
@{\extracolsep{\fill}}
lcc
}
\toprule

\textbf{Method}
&
\textbf{Best certified step size $h_\star$}
&
\textbf{Best accelerated convergence rate}
\\

\midrule

Steepest descent
&
$\displaystyle
\min\left\{
\frac{1}{L},
\frac{t_{\star}(\xi)^2}{\xi\mu}
\right\}
$
&
$\displaystyle
1-
\min\left\{
\sqrt{\frac{\mu}{\xi L}},
\frac{t_{\star}(\xi)}{\xi}
\right\}
$
\\[3mm]

Proximal point
&
$\displaystyle
\min\left\{
\frac{2\delta_D-1}{\mu\delta_D},
\frac{t_{\star}(\xi)^2}{\xi\mu}
\right\}
$
&
$\displaystyle
1-
\min\left\{
\sqrt{
\frac{2\delta_D-1}{\xi\delta_D}
},
\frac{t_{\star}(\xi)}{\xi}
\right\}
$
\\[3mm]

Proximal gradient
&
$\displaystyle
\min\left\{
\frac{1}{L-\rho_{f_2}},
\frac{t_{\star}(\xi)^2}
{\xi(\mu-\rho_{f_2})}
\right\}
$
&
$\displaystyle
1-
\min\left\{
\sqrt{
\frac{\mu-\rho_{f_2}}
{\xi(L-\rho_{f_2})}
},
\frac{t_{\star}(\xi)}{\xi}
\right\}
$
\\[3mm]

Implicit midpoint
&
$\displaystyle
\min\left\{
\frac{
4\delta_D
}{
L+(2+\delta_D)\mu
},
\frac{t_{\star}(\xi)^2}{\xi\mu}
\right\}
$
&
$\displaystyle
1-
\min\left\{
2\sqrt{
\frac{
\delta_D\mu
}{
\xi\bigl[L+(2+\delta_D)\mu\bigr]
}
},
\frac{t_{\star}(\xi)}{\xi}
\right\}
$
\\[4mm]

Geodesic AVF
&
$\displaystyle
\min\left\{
\frac{N_\theta}{A_\theta},
\frac{t_{\star}(\xi)^2}{\xi m_\theta}
\right\}
$
&
$\displaystyle
1-
\min\left\{
\sqrt{
\frac{
m_\theta N_\theta
}{
\xi A_\theta
}
},
\frac{t_{\star}(\xi)}{\xi}
\right\}
$
\\[4mm]

Gonzalez
&
$\displaystyle
\min\left\{
\frac{
8\delta_D\mu
}{
L^2+(2+\delta_D)\mu^2
},
\frac{
2t_{\star}(\xi)^2
}{
\xi\mu
}
\right\}
$
&
$\displaystyle
1-
\min\left\{
\frac{
2\mu\sqrt{\delta_D}
}{
\sqrt{
\xi\bigl[L^2+(2+\delta_D)\mu^2\bigr]
}
},
\frac{t_{\star}(\xi)}{\xi}
\right\}
$
\\[4mm]

Itoh--Abe
&
$\displaystyle
\min\left\{
\frac{
\mu
}{
2dL^2\bar{\Lambda}_D^2
},
\frac{
2t_{\star}(\xi)^2
}{
\xi\mu
}
\right\}
$
&
$\displaystyle
1-
\min\left\{
\frac{
\mu
}{
2\sqrt{\xi d}\,L\bar{\Lambda}_D
},
\frac{t_{\star}(\xi)}{\xi}
\right\}
$
\\[3mm]

\bottomrule
\end{tabular*}

\vspace{2mm}

\begin{minipage}{0.96\textwidth}
\raggedright
\small
For the geodesic AVF method,
$m_\theta=\mu-\rho_\theta^\star$ corresponds to the optimal choice of
$\rho$ for each fixed admissible $\theta$.
For the Itoh--Abe method, the optimal auxiliary parameter is
$\rho_\star=\mu/2$.
\end{minipage}

\end{table*}

\end{appendices}


\bibliography{sn-bibliography}

@inproceedings{Alimisis2020,
  title     = {A Continuous-time Perspective for Modeling Acceleration in Riemannian Optimization},
  author    = {Alimisis, Foivos and Orvieto, Antonio and B{\'e}cigneul, Gary and Lucchi, Aurelien},
  booktitle = {Proceedings of the Twenty Third International Conference on Artificial Intelligence and Statistics},
  pages     = {1297--1307},
  year      = {2020},
  volume    = {108},
  series    = {Proceedings of Machine Learning Research},
  publisher = {PMLR},
  address   = {Online}
}

@book{Lee2018,
  author    = {Lee, John M.},
  title     = {Introduction to Riemannian Manifolds},
  edition   = {2},
  series    = {Graduate Texts in Mathematics},
  volume    = {176},
  publisher = {Springer},
  address   = {Cham},
  year      = {2018},
  doi       = {10.1007/978-3-319-91755-9}
}

@inproceedings{KimYang2022,
  author    = {Kim, Jungbin and Yang, Insoon},
  title     = {Accelerated Gradient Methods for Geodesically Convex Optimization:
               Tractable Algorithms and Convergence Analysis},
  booktitle = {Proceedings of the 39th International Conference on Machine Learning},
  series    = {Proceedings of Machine Learning Research},
  volume    = {162},
  pages     = {11255--11282},
  year      = {2022},
  publisher = {PMLR},
  address   = {Baltimore, Maryland, USA}
}

@inproceedings{Ushiyama2023,
  author    = {Kansei Ushiyama and Shun Sato and Takayasu Matsuo},
  title     = {A Unified Discretization Framework for Differential Equation
               Approach with Lyapunov Arguments for Convex Optimization},
  booktitle = {Advances in Neural Information Processing Systems},
  volume    = {36},
  pages     = {26092--26120},
  year      = {2023},
  doi       = {10.52202/075280-1136}
}

@article{Celledoni2018,
  author  = {Elena Celledoni and S{\o}lve Eidnes and Brynjulf Owren
             and Torbj{\o}rn Ringholm},
  title   = {Dissipative Numerical Schemes on Riemannian Manifolds
             with Applications to Gradient Flows},
  journal = {SIAM Journal on Scientific Computing},
  volume  = {40},
  number  = {6},
  year    = {2018},
  doi     = {10.1137/18M1190628}
}

@article{Celledoni2020,
  author  = {Elena Celledoni and S{\o}lve Eidnes and Brynjulf Owren
             and Torbj{\o}rn Ringholm},
  title   = {Energy-Preserving Methods on Riemannian Manifolds},
  journal = {Mathematics of Computation},
  volume  = {89},
  number  = {322},
  pages   = {699--716},
  year    = {2020},
  doi     = {10.1090/mcom/3470}
}

@article{feng2025riemannian,
  title={A Riemannian Accelerated Proximal Gradient Method},
  author={Feng, Shuailing and Jiang, Yuhang and Huang, Wen and Ying, Shihui},
  journal={arXiv preprint arXiv:2509.21897},
  year={2025}
}

@inproceedings{zhang2016first,
  title={First-order methods for geodesically convex optimization},
  author={Zhang, Hongyi and Sra, Suvrit},
  booktitle={Conference on learning theory},
  pages={1617--1638},
  year={2016},
  organization={PMLR}
}

@inproceedings{zhang2018estimate,
  title={An estimate sequence for geodesically convex optimization},
  author={Zhang, Hongyi and Sra, Suvrit},
  booktitle={Conference On Learning Theory},
  pages={1703--1723},
  year={2018},
  organization={PMLR}
}

@article{liu2017accelerated,
  title={Accelerated first-order methods for geodesically convex optimization on Riemannian manifolds},
  author={Liu, Yuanyuan and Shang, Fanhua and Cheng, James and Cheng, Hong and Jiao, Licheng},
  journal={Advances in Neural Information Processing Systems},
  volume={30},
  year={2017}
}

@inproceedings{martinez2023accelerated,
  title={Accelerated riemannian optimization: Handling constraints with a prox to bound geometric penalties},
  author={Mart{\'\i}nez-Rubio, David and Pokutta, Sebastian},
  booktitle={The Thirty Sixth Annual Conference on Learning Theory},
  pages={359--393},
  year={2023},
  organization={PMLR}
}

@inproceedings{nesterov1983method,
  title={A method for solving the convex programming problem with convergence rate O (1/k2)},
  author={Nesterov, Yurii},
  booktitle={Dokl akad nauk Sssr},
  volume={269},
  pages={543},
  year={1983}
}

@book{nesterov2013introductory,
  title     = {Introductory Lectures on Convex Optimization: A Basic Course},
  author    = {Nesterov, Yurii},
  year      = {2013},
  publisher = {Springer},
  address   = {New York, NY},
 }

@article{guler1991convergence,
  title={On the convergence of the proximal point algorithm for convex minimization},
  author={G{\"u}ler, Osman},
  journal={SIAM journal on control and optimization},
  volume={29},
  number={2},
  pages={403--419},
  year={1991},
  publisher={SIAM}
}

@article{guler1992new,
  title={New proximal point algorithms for convex minimization},
  author={G{\"u}ler, Osman},
  journal={SIAM Journal on Optimization},
  volume={2},
  number={4},
  pages={649--664},
  year={1992},
  publisher={SIAM}
}

@article{lin2018catalyst,
  title={Catalyst acceleration for first-order convex optimization: from theory to practice},
  author={Lin, Hongzhou and Mairal, Julien and Harchaoui, Zaid},
  journal={Journal of Machine Learning Research},
  volume={18},
  number={212},
  pages={1--54},
  year={2018}
}

@article{beck2009fast,
  title={A fast iterative shrinkage-thresholding algorithm for linear inverse problems},
  author={Beck, Amir and Teboulle, Marc},
  journal={SIAM journal on imaging sciences},
  volume={2},
  number={1},
  pages={183--202},
  year={2009},
  publisher={SIAM}
}

@article{nesterov2013gradient,
  title={Gradient methods for minimizing composite functions},
  author={Nesterov, Yu},
  journal={Mathematical programming},
  volume={140},
  number={1},
  pages={125--161},
  year={2013},
  publisher={Springer}
}

@article{ehrhardt2025geometric,
  title={A geometric integration approach to smooth optimization: foundations of the discrete gradient method},
  author={Ehrhardt, Matthias J and Riis, Erlend S and Ringholm, Torbj{\o}rn and Sch{\"o}nlieb, Carola-Bibiane},
  journal={IMA Journal of Numerical Analysis},
  volume={45},
  number={3},
  pages={1269--1299},
  year={2025},
  publisher={Oxford University Press}
}

@inproceedings{ahn2020nesterov,
  title={From Nesterov’s estimate sequence to Riemannian acceleration},
  author={Ahn, Kwangjun and Sra, Suvrit},
  booktitle={Conference on Learning Theory},
  pages={84--118},
  year={2020},
  organization={PMLR}
}

@inproceedings{alimisis2021momentum,
  title={Momentum improves optimization on Riemannian manifolds},
  author={Alimisis, Foivos and Orvieto, Antonio and Becigneul, Gary and Lucchi, Aurelien},
  booktitle={International conference on artificial intelligence and statistics},
  pages={1351--1359},
  year={2021},
  organization={PMLR}
}

@inproceedings{martinez2022global,
  title={Global Riemannian acceleration in hyperbolic and spherical spaces},
  author={Mart{\'\i}nez-Rubio, David},
  booktitle={International conference on algorithmic learning theory},
  pages={768--826},
  year={2022},
  organization={PMLR}
}

@article{huang2022riemannian,
  title={Riemannian proximal gradient methods},
  author={Huang, Wen and Wei, Ke},
  journal={Mathematical Programming},
  volume={194},
  number={1},
  pages={371--413},
  year={2022},
  publisher={Springer}
}

@article{bergmann2507intrinsic,
  title={The intrinsic Riemannian proximal gradient method for convex optimization (2025)},
  author={Bergmann, R and Jasa, H and John, P and Pfeffer, M},
  journal={arXiv preprint arxiv:2507.16055}
}

@article{ferreira2002proximal,
  title={Proximal point algorithm on Riemannian manifolds},
  author={Ferreira, OP and Oliveira, Paulo Roberto},
  journal={Optimization},
  volume={51},
  number={2},
  pages={257--270},
  year={2002},
  publisher={Taylor \& Francis}
}

@article{bento2017iteration,
  title={Iteration-complexity of gradient, subgradient and proximal point methods on Riemannian manifolds},
  author={Bento, Glaydston C and Ferreira, Orizon P and Melo, Jefferson G},
  journal={Journal of Optimization Theory and Applications},
  volume={173},
  number={2},
  pages={548--562},
  year={2017},
  publisher={Springer}
}

@article{wang2016proximal,
  title={Proximal point algorithms on Hadamard manifolds: linear convergence and finite termination},
  author={Wang, Jinhua and Li, Chong and Lopez, Genaro and Yao, Jen-Chih},
  journal={SIAM Journal on Optimization},
  volume={26},
  number={4},
  pages={2696--2729},
  year={2016},
  publisher={SIAM}
}

@inproceedings{jin2022understanding,
  title={Understanding Riemannian acceleration via a proximal extragradient framework},
  author={Jin, Jikai and Sra, Suvrit},
  booktitle={Conference on Learning Theory},
  pages={2924--2962},
  year={2022},
  organization={PMLR}
}

@article{ohta2015discrete,
  title={Discrete-time gradient flows and law of large numbers in Alexandrov spaces},
  author={Ohta, Shin-ichi and P{\'a}lfia, Mikl{\'o}s},
  journal={Calculus of Variations and Partial Differential Equations},
  volume={54},
  number={2},
  pages={1591--1610},
  year={2015},
  publisher={Springer}
}

\end{document}